\documentclass[11pt]{amsart}
\usepackage{amsmath,amssymb,amsfonts,url,mathptmx,txfonts,color}
\usepackage[numbers,sort&compress]{natbib}
\usepackage[colorlinks, linkcolor=blue, citecolor=blue]{hyperref}
\usepackage{fancyhdr}
\usepackage{enumerate}
\usepackage{appendix}
\usepackage{fullpage}
\numberwithin{equation}{section}

\usepackage[marginal]{footmisc}

\usepackage{amsthm,a0size,bm,indentfirst,wasysym}

\def\XXint#1#2#3{{\setbox0=\hbox{$#1{#2#3}{\int}$ }
\vcenter{\hbox{$#2#3$ }}\kern-.6\wd0}}

\newtheorem{theorem}{Theorem}[section]

\newtheorem{lemma}[theorem]{Lemma}
\newtheorem{proposition}[theorem]{Proposition}
\newtheorem{corollary}[theorem]{Corollary}
\newtheorem{definition}[theorem]{Definition}

\newtheorem{rmk}{Remark}
\allowdisplaybreaks[4]
\usepackage{tikz}
\usepackage{amscd}
\usepackage{makecell}
\usetikzlibrary{matrix}
\usepackage{multirow}
\usepackage{longtable}
\usepackage{cancel}
\usepackage{multirow}
\usepackage[misc]{ifsym}

\begin{document}
\title[Blow-up analysis: Affine Toda System]{The blow-up analysis on $\mathbf{B}_2^{(1)}$ affine Toda system: Local mass and affine Weyl group}
	
\author{Leilei Cui}

\author{Jun-cheng Wei}
	
\author{Wen Yang}
	
\author{Lei Zhang}
	
\address{ Leilei Cui\\
	      School of Mathematics and Statistics \\
	      Central China Normal University \\
	      Wuhan, 430079, China }
	      \email{leileicuiccnu@163.com}
	
\address{ Jun-cheng Wei\\
	      Department of Mathematics, The University of British Columbia\\
	      Vancouver, BC Canada V6T 1Z2}
          \email{jcwei@math.ubc.ca}
	
\address{ Wen Yang\\
          Wuhan Institute of Physics and Mathematics \\
          Innovation Academy for Precision Measurement Science and Technology\\
          Chinese Academy of Sciences\\
          Wuhan 430071, P. R. China}
	      \email{math.yangwen@gmail.com}
	
\address{ Lei Zhang\\
	      Department of Mathematics\\
	      University of Florida\\
	      1400 Stadium Rd\\
		  Gainesville FL 32611}
	      \email{leizhang@ufl.edu}
\date{\today}
	
\begin{abstract}
It has been established that the local mass of blow-up solutions to Toda systems associated with the simple Lie algebras $\mathbf{A}_n,~\mathbf{B}_n,~\mathbf{C}_n$ and $\mathbf{G}_2$ can be represented by a finite Weyl group. In particular, at each blow-up point, after a sequence of bubbling steps (via scaling) is performed, the transformation of the local mass at each step corresponds to the action of an element in the Weyl group. In this article, we present the results in the same spirit for the affine $\mathbf{B}_2^{(1)}$ Toda system with singularities. Compared with the Toda system with simple Lie algebras, the computation of local masses is more challenging due to the infinite number of elements of the  {affine Weyl group of type $\mathbf{B}_{2}^{(1)}$}. In order to give an explicit expression for the local mass formula we introduce two free integers and write down all the possibilities into 8 types. This shows a striking difference to previous results on Toda systems with simple Lie algebras. The main result of this article seems to provide the first major advance in understanding the relation between the blow-up analysis of affine Toda system and the {affine Weyl group} of the associated Lie algebras.
\end{abstract}
\maketitle

\\

\vspace{0.25cm}


\tableofcontents

\section{Introduction and main results}\label{CYL-Section-1}
\setcounter{equation}{0}
The main focus of this article is to derive the relation between the quantization result from the blow up analysis and the associated algebraic structure for the following system
\begin{equation}
\label{1.1}
\begin{cases}
\Delta u_i+\sum\limits_{j=1}^3a_{ij}e^{u_j}=4\pi\gamma_i\delta_0\quad &\mbox{in}\quad B_1(0)\subseteq\mathbb{R}^2,\qquad i=1,2,3,\\
u_1+u_2+2u_3=0\quad&\mathrm{in}\quad B_1(0),
\end{cases}
\end{equation}
where $\gamma_i>-1$ for $i=1,2,3$, $\delta_0$ stands for the Dirac measure at the point $0$ and
\begin{equation}
	\label{1.2}
	A:=(a_{ij})_{3\times 3}=\left(\begin{matrix}
		1 & 0 &-1\\
		0 & 1 &-1\\
		-\frac12 & -\frac12 & 1
	\end{matrix}\right).
\end{equation}
Up to a constant one sees that the coefficient matrix $A$ is exactly the Cartan matrix of affine Lie algebra $\mathbf{B}_2^{(1)}$ and we call \eqref{1.1} the affine $\mathbf{B}_2^{(1)}$ Toda system. Let
\begin{equation*}
u_1=-\omega+\eta,\quad u_2=-\omega-\eta,\quad u_3=\omega,
\end{equation*}
then \eqref{1.1} is equivalent to
\begin{equation}
\label{1.3}
\begin{cases}
\Delta \omega+e^{\omega}-\frac12e^{-\omega+\eta}-\frac12e^{-\omega-\eta}=4\pi\gamma_3\delta_0,\\
\\
\Delta \eta+\frac12e^{-\omega+\eta}-\frac12e^{-\omega-\eta}=2\pi(\gamma_1-\gamma_2)\delta_0.
\end{cases}
\end{equation}
When $\gamma_1=\gamma_2=0$, system \eqref{1.3} is related to minimal surface into $\mathbb{S}^4$ without superminimal points \cite{Februs-Pedit-Pinkall-Sterlig-1992} and also appeared in the work of integrable system by Fordy-Gibbons in \cite{Fordy-Gibbons}. If $\eta=0$ then system \eqref{1.3} is reduced to the simplest affine Toda system, i.e., the well-known sinh-Gordon equation
\begin{equation}
\label{1.4}
\Delta u+e^u-e^{-u}=0,
\end{equation}
which was originally introduced by Edmond Bour \cite{Bour-1862} in the study of surfaces of constant negative curvature as the Gauss-Codazzi equation for surfaces of curvature $-1$ in $3$-sphere, see \cite{Jost-Wang-Ye-Zhou-2008,Spruck-1986,Pinkall-Sterling-1989} for its further applications in many other mathematical and physical problems. It is also interesting to remark that system \eqref{1.3} shared some similar structures with the rank two $TT^*$ Toda system \eqref{1.5} which  appeared in the work of Cecotti-Vafa (see \cite{Cecotti-1991}) on topological-anti-topological fusion,
\begin{equation}\label{1.5}
\left\{
\begin{aligned}
\Delta w_0-e^{w_0}+e^{\frac12(w_1-w_0)}&=0,\\
\Delta w_1+e^{-w_1}-e^{\frac{1}{2}(w_1-w_0)}&=0.
\end{aligned}
\right.
\end{equation}
Indeed, if we write
\begin{equation*}
v_1=w_0,\quad v_2=-w_1,\quad v_3=\frac12(w_1-w_0),
\end{equation*}
then $(v_1,v_2,v_3)$ satisfies
\begin{equation*}
\Delta \left(\begin{matrix}v_1\\ v_2\\ v_3\end{matrix}\right)-A\left(\begin{matrix}v_1\\ v_2\\ v_3\end{matrix}\right)=0,\quad v_1+v_2+2v_3=0,
\end{equation*}
which differs from \eqref{1.1} only in the sign of $A$. For more backgrounds and recent developments on $TT^*$ Toda system we refer the readers to \cite{Guest-Its-Lin-2015-1,Guest-Its-Lin-2015-2,Guest-Its-Lin-2015-3,Guest-Lin-2014} and references therein.

From the analytic point of view, one of the fundamental issue is to establish the compactness result of the solution space of system \eqref{1.1}. It is a very challenging problem due to the critical nonlinearity in dimension two and lack of compactness in general situation. One of the important techniques attacking this issue is the so called blow up analysis method, which dates back at least to the famous work of Sacks-Uhlenbeck \cite{Sacks-Uhlenbeck-1982} on the investigation of the blow up phenomena for two-dimensional harmonic maps. Since then, there are a lot of works concerning the blow up analysis for harmonic maps \cite{Eells-Wood-1983,Ding-Tian-1995}, minimal surfaces \cite{Struwe-1986}, the Yamabe equation \cite{Barhi-Coron-1991,Chang-Gursky-Yang-1993}, the Liouville equation and systems \cite{Bartolucci-Tarantello-2017,Brezis-Merle-1991,Chang-Yang-1987,Chen-Lin-2002,Chen-Lin-2003,
Chipot-Shafrir-Wolansky-1997,Li-1999,Li-Shafrir-1994,Ohtsuka-Suzuki} and Toda system of simple Lie algebras \cite{Jost-Lin-Wang-2006,Jost-Wang-2001,Jost-Wang-2002,Lin-Wei-Zhang-2015,Lin-Wei-Yang-Zhang-2018,Lin-Yang-2021,
Lin-Yang-Zhong-2020,Lin-Zhang-2016}.

Among the aforementioned equations, Liouville equation and Toda system of simple Lie algebras have the closest relation to the system we considered in the current article. In the pioneering work of Brezis-Merle \cite{Brezis-Merle-1991}, Li-Shafrir \cite{Li-Shafrir-1994}, Li \cite{Li-1999}, Bartolucci-Tarantello \cite{Bartolucci-Tarantello-2002,Bartolucci-Tarantello-2017}, Lin-Kuo \cite{Kuo-Lin-2016}, Wei-Zhang \cite{Wei-Zhang-2021,Wei-Zhang-2022,Wei-Zhang-2022-1}, the blow up phenomena for the single Liouvlle equation with singularity has been fully understood
\begin{equation*}
\Delta u+e^u=4\pi\alpha\delta_0,\quad  \alpha>-1.
\end{equation*}
After that, the following Toda system has been widely studied (See \cite{Bolton-Jensen-Rigoli-Woodward-1988,Bolton-Woodward-1997,Doliwa-1997,Jost-Wang-2001,Jost-Wang-2002,
Jost-Lin-Wang-2006,Karmakar-Lin-Nie-2020,klnw,Lin-Wei-Yang-Zhang-2018,Lin-Wei-Zhang-2015,Lin-Yang-2021,
Lin-Yang-Zhong-2020,Lin-Wei-Ye-2012,Lin-Zhang-2016,Nie-2016,Ohtsuka-Suzuki-2007})
\begin{equation}
\label{1.6}
\Delta u_i+\sum\limits_{j=1}^na_{ij}e^{u_j}=4\pi\alpha_{i}\delta_0,\quad \alpha_i>-1,
\quad i=1,\cdots,n,
\end{equation}
where $A=(a_{ij})_{n\times n}$ is the Cartan matrix of a simple Lie algebra of rank $n.$ It is known that all the simple Lie algebras are $\mathbf{A}_n,~\mathbf{B}_n,~\mathbf{C}_n,~\mathbf{D}_n,~\mathbf{E}_6,$~$\mathbf{E}_7$,
~$\mathbf{E}_8$,~$\mathbf{F}_4$ and $\mathbf{G}_2$. In \cite{Jost-Wang-2002,Jost-Lin-Wang-2006}, the authors initiated the study on the blow up analysis for the $\mathbf{A}_2$ Toda system. The first step of the blow up analysis is to classify the possible values of the local mass at a blow up point $p$, which is defined as
\begin{equation*}
\sigma_{i}(p):=\frac{1}{2\pi}\lim\limits_{r\to0}\lim\limits_{k\to\infty}\int_{B_r(p)}e^{u_i^k(x)}\mathrm{d}x,\quad i=1,2,
\end{equation*}
for a sequence of blow up solutions $\{(u_1^k,u_2^k)\}$. Under some mild assumptions Jost-Lin-Wang \cite{Jost-Lin-Wang-2006} proved that the local mass value $(\sigma_1(p),\sigma_2(p))$ belongs to the following set
\begin{equation*}
\left\{(0,2),(2,0),(0,4),(4,0),(4,4)\right\}.
\end{equation*}
Since then there has appeared several works \cite{Lin-Wei-Yang-Zhang-2018,Lin-Wei-Zhang-2015,Lin-Yang-2021,Lin-Zhang-2016} concerning the compactness issue for rank two Toda system with simple Lie algebras $\mathbf{A}_2$, $\mathbf{B}_2(\mathbf{C}_2)$ and $\mathbf{G}_2$. While for Toda system with rank $n\ge3$, Lin-Yang-Zhong \cite{Lin-Yang-Zhong-2020} considered the cases for $\mathbf{A}_n$, $\mathbf{B}_n$, $\mathbf{C}_n$ and $\mathbf{G}_2$. Especially, they have shown that the formula of the local mass has a deep connection with the Weyl group of the corresponding Lie algebra. In more precise terms, we consider the $\mathbf{A}_n$ type Toda system, we set
\begin{equation*}
\beta_i=1+\alpha_i,\quad i=1,\cdots,n.
\end{equation*}
Then the local mass for the blow up solutions to \eqref{1.6} at the singular point $0$ can be represented by
\begin{equation}\label{1.7}
\sigma_i(0)=2\sum_{j=0}^{i-1}\left(\sum_{\ell=1}^{f(j)}\beta_\ell-\sum_{\ell=1}^j\beta_\ell\right)+2N_i,\quad i=1,\cdots,n,
\end{equation}
where $N_i\in\mathbb{Z}$ for any $i=1,\cdots,n$ and $f$ is a permutation map from $\{0,\cdots,n\}$ to itself, which is exactly isomorphic to the associated Weyl group of the  $\mathbf{A}_n$ type Lie algebra. The essential point of obtaining the above formula \eqref{1.7} is that for the Lie algebras of types $\mathbf{A}_n$, $\mathbf{B}_n$, $\mathbf{C}_n$ and $\mathbf{G}_2$ the solution is related to an complex ODE whose coefficients are the $W$-invariants of the Toda system. The crucial fact of the corresponding ODE is that the local monodromy matrix is unitary. It is interesting to mention that the representation of the underlying Lie algebra plays a vital role in deriving the ODE. For the other types of Lie algebras, even though we are only able to get a pseudo differential operator instead of an ODE (see \cite{Balog-1990}), one can derive the local mass value for $\mathbf{D}_4$ and $\mathbf{F}_4$ Toda systems. The reason is that any sub-system belongs to one of $\mathbf{A}_n$, $\mathbf{B}_n$ or $\mathbf{C}_n$ type Toda systems with lower rank, see \cite{Karmakar-Lin-Nie-2020}.

In the current article we shall compute the local mass value of the blow up solutions to \eqref{1.1} and try to connect it with the associated algebraic structure of the  {affine} Lie algebra $\mathbf{B}_2^{(1)}$. It is known that (see \cite{Lusztig-1985} for instance) the corresponding algebraic structure of $\mathbf{B}_2^{(1)}$ is the following affine Weyl group: letting $G$ be a group generated by the following generators
\begin{equation}\label{1.8}
G=\langle\mathfrak{R}_1,~\mathfrak{R}_2,~\mathfrak{R}_3\rangle
\end{equation}
with $\mathfrak{R}_i$, $i=1,2,3$ satisfying
\begin{equation}\label{1.9}
\begin{aligned}
\mathfrak{R}_1\circ\mathfrak{R}_2=\mathfrak{R}_2\circ\mathfrak{R}_1, \quad &|\mathfrak{R}_1\circ\mathfrak{R}_2|=2,\\
(\mathfrak{R}_1\circ\mathfrak{R}_3)^2=(\mathfrak{R}_3\circ\mathfrak{R}_1)^2,\quad &|\mathfrak{R}_1\circ\mathfrak{R}_3|=4,\\
(\mathfrak{R}_2\circ\mathfrak{R}_3)^2=(\mathfrak{R}_3\circ\mathfrak{R}_2)^2,\quad &|\mathfrak{R}_2\circ\mathfrak{R}_3|=4,\\
\mathfrak{R}_1^2=\mathfrak{R}_2^2=\mathfrak{R}_3^2=\mathfrak{I}, \quad &
|\mathfrak{R}_i|=2,~i=1,2,3,
\end{aligned}
\end{equation}
where $\mathfrak{I}$ and $\circ$ denotes the identity and the operation of $G$ respectively, $|\mathfrak{R}|$ represents the order of the element $\mathfrak{R}$, i.e., $\mathfrak{R}^{|\mathfrak{R}|}=\mathfrak{I}$. {To simplify the notation, we omit the operation notation $"\circ"$ in this article and denote}
\begin{equation*}
{\mathfrak{R}_{12}=\mathfrak{R}_1\mathfrak{R}_2,~\mathfrak{R}_{13}=\left(\mathfrak{R}_1\mathfrak{R}_3\right)^2,
  ~\mathfrak{R}_{23}=\left(\mathfrak{R}_2\mathfrak{R}_3\right)^2.}
\end{equation*}
To state the main result, we formulate the problem as follows.

Let ${\bf u}^k=(u_1^k,u_2^k,u_3^k)$ be a sequence of blow up solutions to \eqref{1.1} satisfying the following conditions
\begin{equation}\label{1.10}
\begin{cases}
(i): &0~\mbox{is the only blow up point of }{\bf u}^k~\mbox{in}~B_1(0),~\mbox{i.e.},\\ &\max\limits_i\sup\limits_{x\in B_1(0)}{u_i^k(x)+2\gamma_i\log|x|}\to+\infty~\mbox{and}~\max\limits_i\sup\limits_{K\subseteq B_1(0)\setminus\{0\}} u_i^k\leq C(K),\\
(ii): &|u_i^k(x)-u_i^k(y)|\leq C,~\forall x,y~\mbox{on}~\partial B_1(0),~i=1,2,3,\\
(iii): &\int_{B_1(0)}e^{u_i^k}\mathrm{d}x\leq C,~i=1,2,3.
\end{cases}
\end{equation}
For this sequence of blow up solutions we define the local mass by
\begin{equation}\label{1.11}
\sigma_i=\frac{1}{2\pi}\lim_{r\to0}\lim_{k\to+\infty}\int_{B_r(0)}e^{u_i^k}\mathrm{d}x,\quad i=1,2,3.
\end{equation}
We shall see in next section that $\bm{\sigma}=(\sigma_1,\sigma_2,\sigma_3)$ always satisfies the Pohozaev identity
\begin{equation}\label{1.12}
(\sigma_1-\sigma_3)^2+(\sigma_2-\sigma_3)^2=4(\mu_1\sigma_1+\mu_2\sigma_2+2\mu_3\sigma_3),\quad\text{where}
~\mu_i=1+\gamma_i,~i=1,2,3.
\end{equation}
For a given $\bm{\mu}=(\mu_1,\mu_2,\mu_3)$ we introduce the set $\Gamma(\bm{\mu})=\Gamma(\mu_1,\mu_2,\mu_3)$ for the local mass quantity $\bm{\sigma}$ via the following way:
\begin{enumerate}[(i)]
\item ${\bf 0}=(0,0,0)\in\Gamma(\bm{\mu})$.
	
\item If $\bm{\sigma}=(\sigma_1,\sigma_2,\sigma_3)\in\Gamma(\bm{\mu})$, then $\mathfrak{R}\bm{\sigma}\in\Gamma(\bm{\mu})$ for any $\mathfrak{R}\in G$ satisfying \eqref{1.8} and \eqref{1.9}, where each generator   $\mathfrak{R}_i,~i=1,2,3$ sends $\bm{\sigma}$ to $\mathfrak{R}_i\bm{\sigma}$ with
	\begin{equation*}
	(\mathfrak{R}_i\bm{\sigma})_j
	=\begin{cases}
	4\mu_i-2\sum\limits_{j=1}^3a_{ij}\sigma_j+\sigma_i,\quad &\mbox{if}~j=i,\\	
	\sigma_j,\quad &\mbox{if}~j\neq i.	
	\end{cases}
	\end{equation*}
Here, $(a_{ij})_{3\times3}$ is given in \eqref{1.2}.
\end{enumerate}
It is easy to see that for any $\bm{\sigma}=(\sigma_1,\sigma_2,\sigma_3)\in\Gamma(\bm{\mu})$, $\sigma_i$ is a degree one polynomial of $\mu_1,~\mu_2$ and $\mu_3$. Now we are able to state the first result of this article

\begin{theorem}\label{CYL-Theorem-1}
Let $\mathbf{u}^{k}=(u^{k}_{1},u^{k}_{2},u^{k}_{3})$ be a sequence of solutions of system \eqref{1.1}
satisfying \eqref{1.10} and the local mass $\bm{\sigma}=(\sigma_1,\sigma_2,\sigma_3)$ be defined by \eqref{1.11}. Then there exists $\hat{\bm{\sigma}}=(\hat{\sigma}_1,\hat{\sigma}_2,\hat{\sigma}_3)\in\Gamma(\bm{\mu})$ such that
\begin{equation*}
{\sigma}_i=\hat{\sigma}_i+4m_{i},~m_i\in\mathbb{Z},~i=1,2,3.
\end{equation*}
\end{theorem}

In section \ref{CYL-Section-3} we can give an equivalent definition of $\Gamma(\bm{\mu})$ in some concrete way. We have already pointed out $\bm{\sigma}$ satisfies \eqref{1.12}, based on this fact we shall prove that (see Proposition \ref{CYL-Ch-3-Proposition-5})
\begin{equation}\label{1.13}
\Gamma(\bm{\mu})=\Gamma_{N}(\bm{\mu}):=
\left\{\bm{\sigma}~\big|\big.~\bm{\sigma}~\mbox{satisfies}~\eqref{1.12}~\mbox{and}~\sigma_i
=4\sum_{j=1}^3n_{ij}\mu_j,~n_{ij}
\in\mathbb{N}\cup\{0\},~i=1,2,3\right\}.
\end{equation}
In fact, we can give explicit formulas for the coefficients $n_{ij},~i,j=1,2,3$, see Theorem \ref{CYL-Ch-3-Theorem-6}. As a consequence, we can restate the Theorem \ref{CYL-Theorem-1} as the following corollary

\begin{corollary}\label{corollary-1}
Under the setting of Theorem \ref{CYL-Theorem-1} we can find $\hat{\bm{\sigma}}=(\hat{\sigma}_1,\hat{\sigma}_2,\hat{\sigma}_3)\in\Gamma_{ {N}}(\bm{\mu})$ such that
\begin{equation*}
	{\sigma}_i=\hat{\sigma}_i+4m_{i},~m_i\in\mathbb{Z},~i=1,2,3.
\end{equation*}
\end{corollary}

\begin{rmk}\label{CYL-Ch-1-Remark-1}
By the fourth equation of \eqref{1.1}, $\gamma_1,~\gamma_2,~\gamma_3$ are forced to verify that $\gamma_1+\gamma_2+2\gamma_3=0$. Actually, the conclusion of Theorem \ref{CYL-Theorem-1} still holds if we replace the fourth equation of \eqref{1.1} by $u_1^k+u_2^k+2u_3^k=C_k$ with $C_k$ being uniformly bounded from above for any $k$.
\end{rmk}

\begin{rmk}\label{CYL-Ch-1-Remark-2}
The computation on the local mass for \eqref{1.1} without singularity is given by Liu-Wang \cite{Liu-Wang-2021}, they have shown that each $\sigma_i$ is  multiple of $4$, see Remark \ref{CYL-Ch-3-Remark-7} for the explicit expressions. When $u_1=u_2$, system \eqref{1.1} is reduced to the sinh-Gordon equation \eqref{1.4} with singular source (see system \eqref{CYL-Ch-8-Eq-2}). In Section \ref{CYL-Section-8} we shall present the results on the computation of its local mass, which extends the corresponding quantitative results of the sinh-Gordon equation \eqref{1.4} in \cite{Jost-Wang-Ye-Zhou-2008,Jevnikar-Wei-Yang-2018DIE,Jevnikar-Wei-Yang-2018IU} to the case with singularity.
\end{rmk}

\begin{rmk}\label{CYL-Ch-1-Remark-3}
In \cite{Lin-Yang-Zhong-2020} we have already seen that the Weyl group of the corresponding Lie algebra plays an important role in determining the local mass of the blow-up solutions. From the result of Theorem \ref{CYL-Theorem-1} and Theorem \ref{CYL-Ch-3-Theorem-6}, we find that similar things also happen for the affine Toda system. However, there is a major difference between these two types of systems. In the former one, the Weyl group is just the permutation map from $\{0,1,\cdots,n\}$ to itself and the number of elements is finite; while for the later one, from the work \cite{Jevnikar-Wei-Yang-2018DIE} and \cite{Liu-Wang-2021} we see that there are infinite choices for the values of local mass due to the fact that the corresponding affine Weyl group has infinite number of elements. Even though it is not difficult to guess that there are two free integers in the expression of the local mass for $\mathbf{B}_2^{(1)}$ Toda system ({there is only one free integer for the sinh-Gordon equation and Tzitz\'eica equation, see \cite{Jost-Wang-Ye-Zhou-2008,Jevnikar-Wei-Yang-2018DIE,Jevnikar-Wei-Yang-2018IU} and \cite{Jevnikar-Yang-2016,Tzitzeica-1910} for these two equations respectively}), it is not easy to write all of them in a clean form. Instead, we have found that there are 8 types of expression formulas according to the remainder integers after modulo the number $4$, see Theorem \ref{CYL-Ch-3-Theorem-6} for details.
\end{rmk}
\medskip

Theorem \ref{CYL-Theorem-1} has several important applications for the affine Toda system defined on compact Riemann surface. For example, let $(M,g)$ be a compact Riemann surface, $\Delta_g$ be the Beltrami operator on $M$, we consider a sequence of solutions $\mathbf{u}^k=(u^k_1,u^k_2,u^k_3)$ to the following system
\begin{equation}\label{1.14}
\begin{cases}
\Delta_{g} u^k_{i}+\sum\limits_{j=1}^{3}a_{ij}\rho^k_{j}\left(\frac{h^k_{j}e^{u^k_j}}{\int_{M}h^k_{j}e^{u^k_j}\mathrm{d}V_{g}}
-\frac{1}{|M|}\right)
=4\pi{\sum\limits_{p\in S}\alpha_{p,i}\left(\delta_{p}-\frac{1}{|M|}\right)},\quad i=1,2,3,\\
u^k_1+u^k_2+2u^k_3=0\quad\text{on}\quad M,
\end{cases}
\end{equation}
where $h^k_1,h^k_2,h^k_3$ are positive and smooth functions on $M$ with their $C^3(M)$ norm being uniformly bounded in $M$, $S$ is a finite set of $M$, $\alpha_{p,i}>-1$ is the strength of the Dirac mass $\delta_{p}$, and $\bm{\rho}^k=(\rho^k_1,\rho^k_2,\rho^k_3)$ is a sequence of constant vectors with nonnegative components satisfying  $\lim\limits_{k\to+\infty}\bm{\rho}^k=(\rho_1,\rho_2,\rho_3)$. {As the Toda system with simple Lie algebra and Liouville equation, \eqref{1.14} can be regarded as a natural extension of the local version \eqref{1.1} to the case on Riemannian manifold. Equation \eqref{1.14} together with \eqref{1.1} are closely related to the nonlinear Klein-Gordon equations (\cite{Fordy-Gibbons}) and the minimal tori in $\mathbb{S}^4$ (\cite{Februs-Pedit-Pinkall-Sterlig-1992}).} We notice that \eqref{1.14} remains the same if $u_i$ is replaced by $u_i+c_i$ for any constant $c_i$. Thus it is reasonable to assume that each component of $\mathbf{u}^k$ is in
\begin{equation*}
	\mathring{H}^1(M)=\left\{f\in H^{1}(M)~\big|~\int_{M}f\mathrm{d}V_{g}=0\right\}.
\end{equation*}
Simultaneously, from the equation $u_1^k+u_2^k+2u_3^k=0$ one can easily verify that
\begin{equation}\label{1.15}
\alpha_{p,1}+\alpha_{p,2}+2\alpha_{p,3}=0,\quad \forall p\in S.
\end{equation}

By introducing the following Green function on Riemann surface $M$
\begin{equation*}
\Delta G(x,p)+\left(\delta_p-\frac{1}{|M|}\right)=0,\quad \int_MG(x,p)\mathrm{d}V_g=0,
\end{equation*}
we decompose
\begin{equation*}
u_i^k(x)=\tilde u_i^k(x)-4\pi\sum\limits_{p\in S}\alpha_{p,i}G(x,p),\quad i=1,2,3.
\end{equation*}
Then \eqref{1.14} can be rewritten as
\begin{equation}\label{1.16}
\begin{cases}
\Delta_g\tilde u_i^k+\sum\limits_{j=1}^{3}a_{ij}\rho^k_{j}\left(\frac{\tilde h^k_{j}e^{\tilde u^k_j}}{\int_{M}\tilde h^k_{j}e^{\tilde u^k_j}\mathrm{d}V_{g}}
-\frac{1}{|M|}\right)
=0,\quad i=1,2,3,\\
\tilde u^k_1+\tilde u^k_2+2\tilde u^k_3=0\quad\text{on}\quad M,
\end{cases}
\end{equation}
where
\begin{equation*}
\tilde h_i^k=h_i^k \exp\left(-4\pi\sum\limits_{p\in S}\alpha_{p,i}G(x,p)\right),\quad i=1,2,3.
\end{equation*}
Let $\mu_{p,i}=1+\alpha_{p,i}$ for $i=1,2,3$ and denote
\begin{equation*}
	\Gamma_{i}=\left\{2\pi\sum\limits_{p\in R}\sigma_{p,i}+8\pi m_i~\big|~(\sigma_{p,1},\sigma_{p,2},\sigma_{p,3})\in\Gamma(\mu_{p,1},\mu_{p,2},\mu_{p,3}),~R\subseteq S,~m_i\in\mathbb{Z}\right\}.
\end{equation*}
The second result in this article is the following a priori estimate for the system \eqref{1.16}.

\begin{theorem}\label{CYL-Theorem-2}
Suppose that $\rho_i\notin\Gamma_i$ for $i=1,2,3$. Then there exists a constant $C>0$ depending on $\rho_i$ such that for any solution $(\tilde u_1,\tilde u_2,\tilde u_3)$ of system \eqref{1.16} in $\mathring{H}^1(M)$,
\begin{equation*}
|\tilde u_i(x)|\leq C,~\forall x\in M,\quad i=1,2,3.
\end{equation*}
\end{theorem}
To prove Theorem \ref{CYL-Theorem-2}, we shall determine the local mass at each blow-up point of solutions $\mathbf{\tilde{u}}^k$. Here the blow up point $p$ is defined by
\begin{equation*}
\exists \mbox{~a sequence }p^k\to p~\mbox{such that }\max_i\tilde u_i^k(p^k)\to+\infty.
\end{equation*}
The set $B$ of all blow up points is called the blow up set. For each point $p\in B$ we define the local mass by
\begin{equation*}
\sigma_i(p)=\frac{1}{2\pi}\lim\limits_{r\to 0}\lim\limits_{k\to +\infty}\frac{\rho^k_{i}\int_{B(p,r)}\tilde h^k_ie^{\tilde u^k_i}\mathrm{d}V_{g}}{\int_{M}\tilde h^k_ie^{\tilde u^k_i}\mathrm{d}V_{g}},\quad i=1,2,3.
\end{equation*}
By Theorem \ref{CYL-Theorem-1} we are able to compute $\sigma_i(p)$. Together with the assumption $\rho_i\notin\Gamma_i$ we shall see that $B=\emptyset$ and it leads to the compactness of \eqref{1.16}.
\medskip

The organization of this article is as follows. In Section \ref{CYL-Section-2}, we establish a selection process of a sequence of blow up solutions, an oscillation estimate outside the blow-up set, and the local Pohozaev identities corresponding to the system \eqref{1.1}. In Section \ref{CYL-Section-3}, we study the Pohozaev identity of blow up solutions and classify all the possible values of $\Gamma(\bm{\mu})$. In Section \ref{CYL-Section-4}, we present several important lemmas for later use. In Section \ref{CYL-Section-5} and Section \ref{CYL-Section-6}, we discuss the local mass on each bubbling disk centered at $0$ and not at $0$ respectively, then we prove Theorem \ref{CYL-Theorem-1} by grouping these bubbling areas together. Furthermore, we sketch the proof of Theorem \ref{CYL-Theorem-2}. In Section \ref{CYL-Section-8}, we carry out the blow-up analysis for sinh-Gordon equation with singular source. In Section \ref{CYL-Section-7}, we collect some classification results which are used in our proof.

In this article, we always use the notation $r_k=o(1)s_k$ to express $r_k/s_k\to 0$ as $k\to\infty$, and the notation $r_k=O(1)s_k$ to represent $C^{-1}\leq r_k/s_k\leq C$ as $k\to\infty$ for some constant $C>0$. For the sake of brevity, we use boldface type for sequence and vectors, such as $\mathbf{s}=\{s_k\}$, $\mathbf{r}=\{r_k\}$, $\bm{\tau}=\{\tau_k\}$ and $\bm{\sigma}=(\sigma_1,\sigma_2,\sigma_3)$. We will not distinguish sequence and subsequence in this article, i.e., we shall still denote $\mathbf{u}^k$ the subsequence of itself if necessary.

\section{Bubbling analysis and Pohozaev identity}\label{CYL-Section-2}
\setcounter{equation}{0}
In this section, we analyze the bubbling areas by a standard selection procedure and establish some type of Pohozaev identity for local mass.
\begin{proposition}\label{CYL-Ch-2-Proposition-1}
Let $\mathbf{u}^{k}$ be a sequence of solutions of system \eqref{1.1} satisfying  \eqref{1.10}. Then there exists a sequence of finite points $\Sigma_{k}:=\{0,x^{k}_{1},\cdots,x^{k}_{m}\}$ (if $0$ is not a singular point, then $0$ can be deleted from $\Sigma_k$) and a sequence of positive numbers $l^{k}_{1},\cdots,l^{k}_{m}$ such that
\begin{enumerate}[(1)]
\item $x^{k}_{j}\to 0$ and $l^{k}_{j}\to 0$ as $k\to +\infty$, $l^{k}_{j}\leq\frac{1}{2}\mathrm{dist}(x^{k}_{j},\Sigma_{k}\setminus\{x^{k}_{j}\})$, $j=1,\cdots,m$. Furthermore, $B(x^{k}_{i},l_{i}^{k})\cap B(x^{k}_{j},l_{j}^{k})=\emptyset$ for $1\leq i,j\leq m,~i\neq j$.
\item $\max\limits_{i=1,2,3}u^k_{i}(x^{k}_{j})=\max\limits_{i=1,2,3}\max\limits_{B(x^k_{j},l^k_j)}u^k_i(x)\to+\infty$ as $k\to+\infty$, $j=1,\cdots,m$. Denote
\begin{equation*}
\varepsilon_{j}^{k}:=e^{-\frac{1}{2}\max\limits_{i=1,2,3}u^{k}_{i}(x^{k}_{j})},\quad j=1,\cdots,m.
\end{equation*}
Then $R^k_j:=\frac{l^{k}_{j}}{\varepsilon_{j}^{k}}\to+\infty$ as $k\to+\infty$, $j=1,\cdots,m$.
\item In each $B(x^{k}_{i},l_{i}^{k})$, set
\begin{equation*}
v^k_{i}(y):=u^k_i(x^k_j+\varepsilon_{j}^{k}y)+2\log\varepsilon_{j}^{k}, \quad i=1,2,3.
\end{equation*}
Then one of the following alternatives holds:
\begin{enumerate}
\item [(a)] $v^k_1$ and $v^k_3$ (or $v^k_2$ and $v^k_3$) converge to a solution of the following Toda system in $C_{\mathrm{loc}}^2(\mathbb{R}^2)$,
\begin{equation*}
\begin{cases}
-\Delta v =e^{v }-e^{w},\\
-\Delta w=-\frac{1}{2}e^{v}+e^{w},
\end{cases}
~~\text{in}\quad\mathbb{R}^2,
\end{equation*}
    while $v^k_2$ (or $v^k_1$) converges to $-\infty$ in $L_{\mathrm{loc}}^{\infty}(\mathbb{R}^2)$;
\item [(b)] $v^k_1$ and $v^k_2$ converge to solutions of Liouville equation in $C_{\mathrm{loc}}^2(\mathbb{R}^2)$, while $v^k_3$ converges to $-\infty$ in $L_{\mathrm{loc}}^{\infty}(\mathbb{R}^2)$;
\item [(c)] One component of $\mathbf{v}^k$ converges to a solution of Liouville equation in $C_{\mathrm{loc}}^2(\mathbb{R}^2)$, while the left ones converge to $-\infty$ in $L_{\mathrm{loc}}^{\infty}(\mathbb{R}^2)$.
\end{enumerate}
\item There exists a constant $C$ independent of $k$ such that the following Harnack-type inequality holds:
\begin{equation}\label{CYL-Ch-2-Eq-2}
\max\limits_{i=1,2,3}\left\{u^k_i(x)+2\log\mathrm{dist}(x,\Sigma_k)\right\}\leq C,\quad \forall x\in B_{1}(0).
\end{equation}
\end{enumerate}
\end{proposition}
\begin{proof}
{We construct $\Sigma_k$ by induction. If $\gamma_i=0~\text{for}~i=1,2,3$, i.e., the system \eqref{1.1} has no singularity at $0$, we start with $\Sigma_k=\emptyset$ and the proof is in line with that of \cite[Lemma 2.1]{Liu-Wang-2021}. Otherwise, we start with $\Sigma_k=\{0\}$ and the proof is similar. But it should be mentioned that all components of ${\mathbf{v}}^k$ cannot be bounded at the same time by the fourth equation of \eqref{1.1}.}
\end{proof}

We have the following oscillation estimate generated by the Harnack-type inequality \eqref{CYL-Ch-2-Eq-2}.

\begin{proposition}\label{CYL-Ch-2-Proposition-2}
Let $\Sigma_{k}=\{0,x^{k}_{1},\cdots,x^{k}_{m}\}$ be the blow-up set (if $0$ is not a singular point, then $0$ can be deleted from $\Sigma_k$). Suppose that $\mathbf{u}^{k}$ is a sequence of solutions to system \eqref{1.1} satisfying \eqref{1.10}. Then for any $x_{0}\in B_{1}(0)\setminus\Sigma_{k}$, there exists a constant $C_{0}$ independent of $x_{0}$ and $k$ such that
\begin{equation*}
|u^{k}_{i}(x_{1})-u^{k}_{i}(x_{2})|\leq C_{0},~~\forall x_{1},x_{2}\in B(x_{0},\mathrm{dist}(x_{0},\Sigma_{k})/2),\quad i=1,2,3.
\end{equation*}
\end{proposition}
\begin{proof}
One can prove this result by following the same arguments of \cite[Lemma 2.1]{Lin-Wei-Zhang-2015} and \cite[Lemma 2.2]{Liu-Wang-2021}, here we omit the details.
\end{proof}

Next we introduce the definitions of fast decay and slow decay.

\begin{definition}\label{CYL-Ch-2-Definition-3}
\begin{enumerate}[(i)]
\item We say $u^{k}_{i}$ has fast decay on $\partial B(x_{k},r_{k})$ if
\begin{equation*}
u^{k}_{i}(x)+2\log|x-x_{k}|\leq-N_{k},\quad\forall x\in \partial B(x_{k},r_{k}),
\end{equation*}
for some $N_{k}\to+\infty$ as $k\to+\infty$.
\item We say $u^{k}_{i}$ has slow decay on $\partial B(x_{k},r_{k})$ if
\begin{equation*}
u^{k}_{i}(x)+2\log|x-x_{k}|\geq -C, \quad\forall x\in \partial B(x_{k},r_{k}),
\end{equation*}
for some constant $C>0$ which is independent of $k$.
\end{enumerate}
\end{definition}

At the end of this section, we can establish a Pohozaev identity of local masses for the system \eqref{1.1}. In Proposition \ref{CYL-Ch-2-Proposition-5} and Remark \ref{CYL-Ch-2-Remark-6}, we denote
\begin{equation*}
\sigma(r,x_{0};u):=\frac{1}{2\pi}\int_{B_{r}(x_{0})}e^{u(x)}\mathrm{d}x,
~~\sigma_{i}^{k}(r,x_{0}):=\sigma(r,x_{0};u^{k}_{i}),
~~\sigma_{i}^{k}(r):=\sigma(r,0;u^{k}_{i}),\quad i=1,2,3.
\end{equation*}
Let $x^k_l\in\Sigma_k$ and $\tau^k_l=\frac{1}{2}\mathrm{dist}\left(x^k_l,\Sigma_k\setminus\{x^k_l\}\right)$, then we can derive from system \eqref{1.1} and Proposition \ref{CYL-Ch-2-Proposition-2} that
\begin{equation}\label{CYL-Ch-2-Eq-5}
u^k_i(x)=\overline{u}^k_{x^k_l,i}(r)+O(1),\quad x\in B(x^k_l,\tau^k_l),
\end{equation}
where $r=|x-x^k_l|$ and
\begin{equation*}
\overline{u}^k_{x^k_l,i}(r)=\frac{1}{2\pi r}\int_{\partial B(x^k_l,r)}u^k_i~\mathrm{d}S,
\end{equation*}
and $O(1)$ is independent of $r$ and $k$.

\begin{proposition}\label{CYL-Ch-2-Proposition-5}
Let $\Sigma_{k}^{\prime}\subseteq\Sigma_{k}$ be a subset of $\Sigma_{k}$ with $0\in\Sigma_{k}^{\prime}\subseteq B(x_{k},r_{k})\subseteq B_{1}(0)$, and there holds
\begin{equation*}
\mathrm{dist}(\Sigma_{k}^{\prime},\partial B(x_{k},r_{k}))=o(1)\mathrm{dist}(\Sigma_{k}\setminus\Sigma_{k}^{\prime},\partial B(x_{k},r_{k})).
\end{equation*}
Suppose $u^{k}_{1},u^{k}_{2},u^{k}_{3}$ have fast decay on $\partial B(x_{k},r_{k})$. Then we have the following Pohozaev identity
\begin{equation}\label{CYL-Ch-2-Eq-6}
\begin{aligned}
&\left(\sigma^{k}_1(r_{k},x_{k})-\sigma^{k}_3(r_{k},x_{k})-2\gamma_{1}\right)^2
+\left(\sigma^{k}_2(r_{k},x_{k})-\sigma^{k}_3(r_{k},x_{k})-2\gamma_{2}\right)^2\\
&=4\left(\sigma^{k}_1(r_{k},x_{k})+\sigma^{k}_2(r_{k},x_{k})+2\sigma^{k}_3(r_{k},x_{k})\right)
+4\left(\gamma_{1}^2+\gamma_{2}^2\right)+o(1).
\end{aligned}
\end{equation}
Furthermore, let $\mu_i=1+\gamma_i$ for $i=1,2,3$, then ${\bm{\sigma}}=(\sigma_1,\sigma_2,\sigma_3)$ satisfies the Pohozaev identity
\begin{equation}\label{CYL-Ch-2-Eq-Pohozaev-identity}
(\sigma_1-\sigma_3)^2+(\sigma_2-\sigma_3)^2=4(\mu_1\sigma_1+\mu_2\sigma_2+2\mu_3\sigma_3),
\end{equation}
where $\mu_1+\mu_2+2\mu_3=4$ and $\sigma_i=\lim\limits_{k\to+\infty}\sigma^k_i(r_k,x_k)$ for $i=1,2,3$.
\end{proposition}
\begin{proof} Without loss of generality, we assume $x_k=0$. For fixed $\theta\in(0,1)$, let $\Omega=B_{\theta}(0)\setminus B_{\varepsilon}(0)$, where $\varepsilon>0$ is small enough. Then applying the standard Pohozaev identity for the first two equations of \eqref{1.1} on $\Omega$, we have
\begin{equation*}
\begin{aligned}
&\frac{1}{2}\int_{\partial\Omega}|\nabla u^k_{i}|^2(x\cdot\nu)\mathrm{d}S-\int_{\partial\Omega}(x\cdot\nabla u^k_{i})(\nu\cdot\nabla u^k_{i})\mathrm{d}S\\
&=-2\int_{\Omega}e^{u^k_i}\mathrm{d}x+\int_{\partial\Omega}e^{u^k_i}(x\cdot\nu)\mathrm{d}S
-\int_{\Omega}e^{u^k_3}(x\cdot\nabla u^k_i)\mathrm{d}x,\quad i=1,2,
\end{aligned}
\end{equation*}
where $\nu$ denotes the unit outer normal vector. Combining the above equations and the fourth equation of \eqref{1.1} we derive that
\begin{equation}\label{CYL-Ch-2-Eq-7}
\begin{aligned}
&\frac{1}{2}\int_{\partial\Omega}\left(|\nabla u^k_{1}|^2+|\nabla u^k_{2}|^2\right)(x\cdot\nu)\mathrm{d}S-\int_{\partial\Omega}\left((x\cdot\nabla u^k_{1})(\nu\cdot\nabla u^k_{1})+(x\cdot\nabla u^k_{2})(\nu\cdot\nabla u^k_{2})\right)\mathrm{d}S\\
&=-2\int_{\Omega}(e^{u^k_1}+e^{u^k_2})\mathrm{d}x+\int_{\partial\Omega}(e^{u^k_1}+e^{u^k_2})(x\cdot\nu)\mathrm{d}S
-\int_{\Omega}e^{u^k_3}\left(x\cdot(\nabla u^k_1+\nabla u^k_2)\right)\mathrm{d}x\\
&=-2\int_{\Omega}(e^{u^k_1}+e^{u^k_2})\mathrm{d}x+\int_{\partial\Omega}(e^{u^k_1}+e^{u^k_2})(x\cdot\nu)\mathrm{d}S
+2\int_{\Omega}e^{u^k_3}(x\cdot\nabla u^k_3)\mathrm{d}x\\
&=-2\int_{\Omega}(e^{u^k_1}+e^{u^k_2}+2e^{u^k_3})\mathrm{d}x+\int_{\partial\Omega}(e^{u^k_1}
+e^{u^k_2}+2e^{u^k_3})(x\cdot\nu)\mathrm{d}S.
\end{aligned}
\end{equation}

We calculate the L.H.S. and R.H.S. of \eqref{CYL-Ch-2-Eq-7} respectively. By (iii) of \eqref{1.10}, we obtain
\begin{equation*}
\begin{aligned}
\text{R.H.S.}
&=-2\int_{\Omega}(e^{u^k_1}+e^{u^k_2}+2e^{u^k_3})\mathrm{d}x
+\int_{\partial\Omega}(e^{u^k_1}+e^{u^k_2}+2e^{u^k_3})(x\cdot\nu)\mathrm{d}S\\
&=-2\int_{B_{\theta}(0)\setminus B_{\varepsilon}(0)}(e^{u^k_1}+e^{u^k_2}+2e^{u^k_3})\mathrm{d}x
+\theta\int_{\partial B_{\theta}(0)}(e^{u^k_1}+e^{u^k_2}+2e^{u^k_3})\mathrm{d}S\\
&\quad -\varepsilon\int_{\partial B_{\varepsilon}(0)}(e^{u^k_1}+e^{u^k_2}+2e^{u^k_3})\mathrm{d}S+o(1)\\
&=-4\pi(\sigma^k_1(\theta)+\sigma^k_2(\theta)+2\sigma^k_3(\theta))+\theta\int_{\partial B_{\theta}(0)}(e^{u^k_1}+e^{u^k_2}+2e^{u^k_3})\mathrm{d}S+o(1).
\end{aligned}
\end{equation*}
On the other hand, we have
\begin{equation*}
\begin{aligned}
\text{L.H.S.}
&=\frac{1}{2}\int_{\partial\Omega}\left(|\nabla u^k_{1}|^2+|\nabla u^k_{2}|^2\right)(x\cdot\nu)\mathrm{d}S-\int_{\partial\Omega}\left((x\cdot\nabla u^k_{1})(\nu\cdot\nabla u^k_{1})+(x\cdot\nabla u^k_{2})(\nu\cdot\nabla u^k_{2})\right)\mathrm{d}S\\
&=\frac{\theta}{2}\int_{\partial B_{\theta}(0)}\left(|\nabla u^k_{1}|^2+|\nabla u^k_{2}|^2\right)\mathrm{d}S-\frac{\varepsilon}{2}\int_{\partial B_{\varepsilon}(0)}(|\nabla u^k_{1}|^2+|\nabla u^k_{2}|^2)\mathrm{d}S\\
&\quad-\theta\int_{\partial B_{\theta}(0)}\left(\left(\nu\cdot\nabla{u^k_1}\right)^2+\left(\nu\cdot\nabla{ u^k_2}\right)^2\right)\mathrm{d}S+\varepsilon\int_{\partial B_{\varepsilon}(0)}\left(\left(\nu\cdot\nabla{ u^k_1}\right)^2+\left(\nu\cdot\nabla{u^k_2}\right)^2\right)\mathrm{d}S.
\end{aligned}
\end{equation*}
Applying the same argument of \cite[Lemma 4.1]{Lin-Zhang-2010}, we deduce the following estimate
\begin{equation}\label{CYL-Ch-2-Eq-8}
\nabla u^k_i(x)=-2\gamma_ix/|x|^2+o(1)~~\text{near the origin},
\end{equation}
which implies that
\begin{equation*}
\begin{aligned}
-\frac{\varepsilon}{2}\int_{\partial B_{\varepsilon}(0)}\left(|\nabla u^k_{1}|^2+|\nabla u^k_{2}|^2\right)\mathrm{d}S+\varepsilon\int_{\partial B_{\varepsilon}(0)}\left(\left(\nu\cdot\nabla{u^k_1}\right)^2+\left(\nu\cdot\nabla{u^k_2}\right)^2\right)\mathrm{d}S
=4\pi(\gamma_1^2+\gamma_2^2)+o(1).
\end{aligned}
\end{equation*}
Hence we can rewrite \eqref{CYL-Ch-2-Eq-7} as
\begin{equation}\label{CYL-Ch-2-Eq-9}
\begin{aligned}
&\frac{\theta}{2}\int_{\partial B_{\theta}(0)}(|\nabla u^k_{1}|^2+|\nabla u^k_{2}|^2)\mathrm{d}S-\theta\int_{\partial B_{\theta}(0)}\left(\left(\nu\cdot\nabla{u^k_1}\right)^2+\left(\nu\cdot\nabla{u^k_2}\right)^2\right)\mathrm{d}S+4\pi(\gamma_1^2+\gamma_2^2)\\
&=-4\pi\left(\sigma^k_1(\theta)+\sigma^k_2(\theta)+2\sigma^k_3(\theta)\right)+\theta\int_{\partial B_{\theta}(0)}(e^{u^k_1}+e^{u^k_2}+2e^{u^k_3})\mathrm{d}S+o(1).
\end{aligned}
\end{equation}

Since $u^k_1,u^k_2,u^k_3$ have fast decay on $\partial B(0,{r_k})$, i.e.,
\begin{equation*}
u^k_i(x)+2\log|x|\leq -N_{k},~|x|=r_k,\quad i=1,2,3
\end{equation*}
for some $N_k\to+\infty$, then by \cite[Lemma 2.4]{Liu-Wang-2021}, there exists $R^k_1\to+\infty$ such that
\begin{equation*}
R^k_1\mathrm{dist}(\Sigma_{k}^{\prime},\partial B(x_{k},r_{k}))=o(1)\mathrm{dist}(\Sigma_{k}\setminus\Sigma_{k}^{\prime},\partial B(x_{k},r_{k}))
\end{equation*}
and
\begin{equation*}
u^k_i(x)+2\log|x|\leq -N_{k},\quad r_k\leq|x|\leq R^k_1 r_k,
\end{equation*}
for some $N_k\to+\infty$. So we can choose $R^k\to+\infty$ with $R^k\leq R^k_1$ such that $u^k_1,u^k_2,u^k_3$ have fast decay on $\partial B(0,R^kr_k)$ and
\begin{equation*}
\int_{B(0,R^k r_{k})\setminus B(0,r_{k})}e^{u^k_{i}(x)}\mathrm{d}x\leq Ce^{-N_{k}}\log R_{k}=o(1),\quad i=1,2,3,
\end{equation*}
which implies that
\begin{equation*}
\sigma^k_i(R^k r_k)=\sigma^k_i(r_k)+o(1),\quad i=1,2,3.
\end{equation*}

Letting $\theta=\sqrt{R^k}r_k$, then we have
\begin{equation*}
-4\pi(\sigma^k_1(\theta)+\sigma^k_2(\theta)+2\sigma^k_3(\theta))=-4\pi(\sigma^k_1(r_k)+\sigma^k_2(r_k)
+2\sigma^k_3(r_k))+o(1)
\end{equation*}
and
\begin{equation*}
\theta\int_{\partial B_{\theta}(0)}\left(e^{u^k_1}+e^{u^k_2}+2e^{u^k_3}\right)\mathrm{d}S=o(1),
\end{equation*}
since $u^k_1,u^k_2,u^k_3$ have fast decay on $\partial B(0,\sqrt{R^k}r_k)$. Therefore, \eqref{CYL-Ch-2-Eq-9} can be rewritten as
\begin{equation}\label{CYL-Ch-2-Eq-10}
\begin{aligned}
&-\frac{\sqrt{R^k}r_{k}}{2}\int_{\partial B_{\sqrt{R^k}r_{k}}(0)}\left(|\nabla u^k_{1}|^2+|\nabla u^k_{2}|^2\right)\mathrm{d}S+\sqrt{R^k}r_{k}\int_{\partial B_{\sqrt{R^k}r_{k}}(0)}\left(\left(\nu\cdot\nabla{u^k_1}\right)^2+\left(\nu\cdot\nabla{u^k_2}\right)^2\right)\mathrm{d}S\\
&=4\pi\left(\sigma^k_1(r_k)+\sigma^k_2(r_k)+2\sigma^k_3(r_k)\right)+4\pi\left(\gamma_1^2+\gamma_2^2\right)+o(1).
\end{aligned}
\end{equation}
Similar to the estimate \eqref{CYL-Ch-2-Eq-8} (by \cite[Lemma 4.1]{Lin-Zhang-2010} and a scaling argument), we obtain
\begin{equation*}
  \nabla u^k_1(x)=-\frac{x}{|x|^2}(\sigma^k_1(r_k)-\sigma^k_3(r_k)-2\gamma_1)+\frac{o(1)}{|x|},~x\in\partial B_{\sqrt{R^k}r_k}(0),
\end{equation*}
and
\begin{equation*}
  \nabla u^k_2(x)=-\frac{x}{|x|^2}(\sigma^k_2(r_k)-\sigma^k_3(r_k)-2\gamma_2)+\frac{o(1)}{|x|},~x\in\partial B_{\sqrt{R^k}r_k}(0).
\end{equation*}
Therefore, we conclude from \eqref{CYL-Ch-2-Eq-10} that
\begin{equation*}
\left(\sigma^{k}_1(r_{k})
-\sigma^{k}_3(r_{k})-2\gamma_{1}\right)^2+\left(\sigma^{k}_2(r_{k})-\sigma^{k}_3(r_{k})-2\gamma_{2}\right)^2=
4\left(\sigma^{k}_1(r_{k})+\sigma^{k}_2(r_{k})+2\sigma^{k}_3(r_{k})\right)+4\left(\gamma_{1}^2+\gamma_{2}^2\right)+o(1).
\end{equation*}

Finally, taking $\mu_i=1+\gamma_i$ for $i=1,2,3$ and letting $k\to+\infty$, we can rewrite \eqref{CYL-Ch-2-Eq-6} as \eqref{CYL-Ch-2-Eq-Pohozaev-identity}. This completes the proof of Proposition \ref{CYL-Ch-2-Proposition-5}.
\end{proof}

\begin{rmk}\label{CYL-Ch-2-Remark-6}
Let $\Sigma_{k}^{\prime}\subseteq\Sigma_{k}$ be a subset of $\Sigma_{k}$ with $0\notin\Sigma_{k}^{\prime}\subseteq B(x_{k},r_{k})\subseteq B_{1}(0)$, and there holds
\begin{equation*}
\mathrm{dist}(\Sigma_{k}^{\prime},\partial B(x_{k},r_{k}))=o(1)\mathrm{dist}(\Sigma_{k}\setminus\Sigma_{k}^{\prime},\partial B(x_{k},r_{k})).
\end{equation*}
Suppose $u^{k}_{1},u^{k}_{2},u^{k}_{3}$ have fast decay on $\partial B(x_{k},r_{k})$. Then we have the following Pohozaev identity
\begin{equation*}
\begin{aligned}
\left(\sigma^{k}_1(r_{k},x_{k})
-\sigma^{k}_3(r_{k},x_{k})\right)^2
+\left(\sigma^{k}_2(r_{k},x_{k})-\sigma^{k}_3(r_{k},x_{k})\right)^2
=4\left(\sigma^{k}_1(r_{k},x_{k})+\sigma^{k}_2(r_{k},x_{k})+2\sigma^{k}_3(r_{k},x_{k})\right)+o(1).
\end{aligned}
\end{equation*}
In other words,
\begin{equation*}
(\sigma_1-\sigma_3)^2+(\sigma_2-\sigma_3)^2=4(\sigma_1+\sigma_2+2\sigma_3).
\end{equation*}
\end{rmk}

\section{The structure of $\Gamma(\mu_1,\mu_2,\mu_3)$ and affine Weyl group of $\mathbf{B}_{2}^{(1)}$}\label{CYL-Section-3}
\setcounter{equation}{0}
In this section, we shall compute all the possible local mass values for the blow up solutions to \eqref{1.1} by studying the set $\Gamma(\bm{\mu})$, which is introduced in the introduction (Section \ref{CYL-Section-1}). First of all, we prove that each element of $\Gamma(\bm{\mu})$ verifies the Pohozaev identity \eqref{CYL-Ch-2-Eq-Pohozaev-identity}.

\begin{proposition}\label{CYL-Ch-3-Proposition-1}
For each element $\bm{\sigma}\in\Gamma(\bm{\mu})$, $\bm{\sigma}$ satisfies the Pohozaev identity \eqref{CYL-Ch-2-Eq-Pohozaev-identity}.
\end{proposition}
\begin{proof}
One can get it easily by Vieta's theorem and Pohozaev identity \eqref{CYL-Ch-2-Eq-Pohozaev-identity}.
\end{proof}

As indicated by the setting of $G$, there might have several ways defining the same elements in $\Gamma(\bm{\mu})$. In order to represent each element in the cheapest way we  introduce the following concept
\begin{definition}\label{CYL-Ch-3-Definition-2}
We say $\widetilde{\mathfrak{R}}=\mathfrak{R}_{i_{n}}\mathfrak{R}_{i_{n-1}}\cdots\mathfrak{R}_{i_{1}}$ ($i_{a}\in\{1,2,3\},~1\leq a\leq n$) is a \textbf{simplest chain} with length $n$ if $\widetilde{\mathfrak{R}}$ can not be reduced to a sub-chain $\mathfrak{R}_{j_{m}}\mathfrak{R}_{j_{m-1}}\cdots\mathfrak{R}_{j_{1}}$ for any $m<n$, $j_{a}\in\{1,2,3\},~1\leq a\leq m$.
\end{definition}

\begin{rmk}\label{CYL-Ch-3-Remark-4}
For example, $\mathfrak{R}_3\mathfrak{R}_1\mathfrak{R}_2\mathfrak{R}_1$ is not a simplest chain, while $\mathfrak{R}_3\mathfrak{R}_1\mathfrak{R}_2\mathfrak{R}_3$ is a simplest chain with length $4$.
\end{rmk}

\begin{proposition}\label{CYL-Ch-3-Proposition-3}
For any $\bm{\sigma}=(\sigma_1,\sigma_2,\sigma_3)\in\Gamma(\bm{\mu})$, $\sigma_i=4\sum_{j=1}^{3}n_{ij}\mu_j$ for some $n_{ij}\in\mathbb{N}\cup\{0\}$.
\end{proposition}
\begin{proof}
By induction, one can easily verify that $\sigma_i=4\sum_{j=1}^{3}n_{ij}\mu_j$ is a polynomial of $\mu_j$ with $n_{ij}\in\mathbb{N}\cup\{0\}$. In fact, direct computation shows that the first $3$ levels of $\Gamma(\bm{\mu})$ are given as
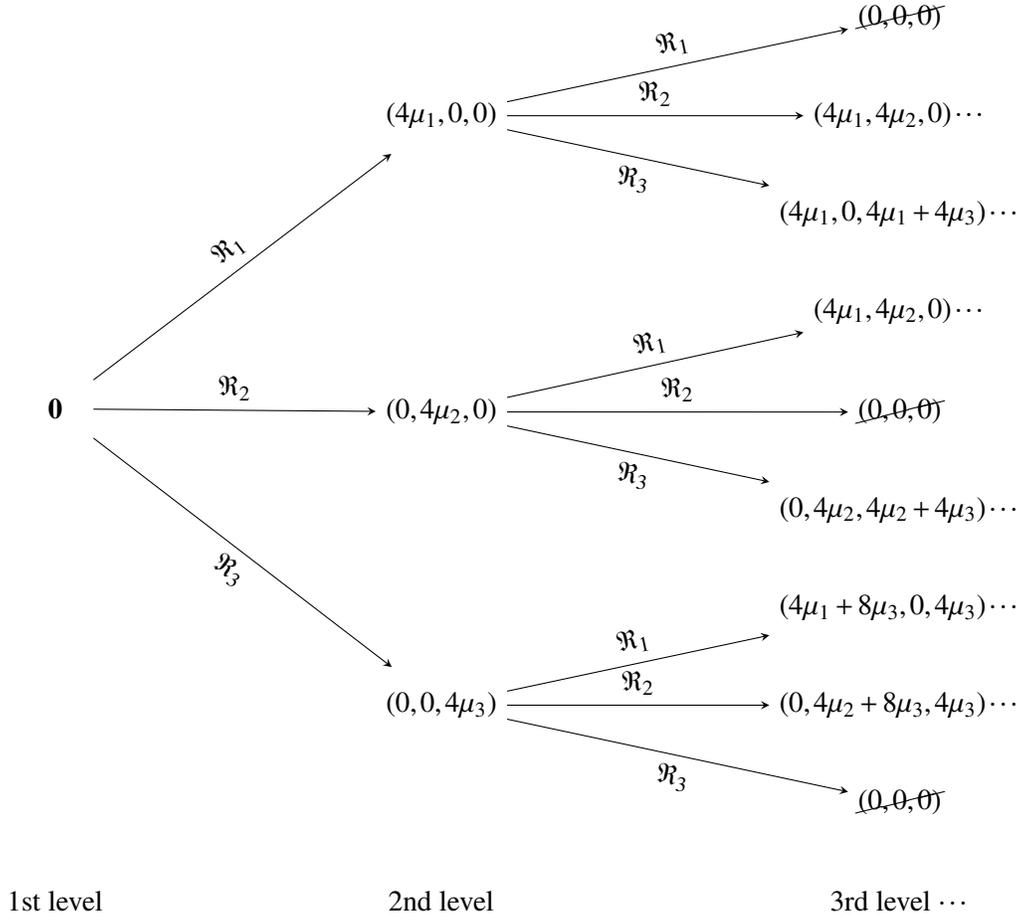
\begin{center}
\begin{tikzpicture}[>=stealth,sloped]
\matrix(tree)[matrix of nodes,minimum size=1cm,column sep=3.5cm,row sep=0.3cm,]
{
          &                & $\cancel{(0,0,0)}$               \\
          & $(4\mu_1,0,0)$ & $(4\mu_1,4\mu_2,0)\cdots$      \\
          &                & $(4\mu_1,0,4\mu_1+4\mu_3)\cdots$ \\
          &                & $(4\mu_1,4\mu_2,0)\cdots$  \\
 $\bm{0}$ & $(0,4\mu_2,0)$ & $\cancel{(0,0,0)}$\\
          &                & $(0,4\mu_2,4\mu_2+4\mu_3)\cdots$\\
          &                & $(4\mu_1+8\mu_3,0,4\mu_3)\cdots$\\
          & $(0,0,4\mu_3)$ & $(0,4\mu_2+8\mu_3,4\mu_3)\cdots$\\
          &                & $\cancel{(0,0,0)}$\\
$1$st level & $2$nd level &     $3$rd level~$\cdots$ \\
};
\draw[->](tree-5-1)--(tree-2-2)node[midway,above]{$\mathfrak{R}_1$};
\draw[->](tree-5-1)--(tree-5-2)node[midway,above]{$\mathfrak{R}_2$};
\draw[->](tree-5-1)--(tree-8-2)node[midway,below]{$\mathfrak{R}_3$};
\draw[->](tree-2-2)--(tree-1-3)node[midway,above]{$\mathfrak{R}_1$};
\draw[->](tree-2-2)--(tree-2-3)node[midway,above]{$\mathfrak{R}_2$};
\draw[->](tree-2-2)--(tree-3-3)node[midway,below]{$\mathfrak{R}_3$};
\draw[->](tree-5-2)--(tree-4-3)node[midway,above]{$\mathfrak{R}_1$};
\draw[->](tree-5-2)--(tree-5-3)node[midway,above]{$\mathfrak{R}_2$};
\draw[->](tree-5-2)--(tree-6-3)node[midway,below]{$\mathfrak{R}_3$};
\draw[->](tree-8-2)--(tree-7-3)node[midway,above]{$\mathfrak{R}_1$};
\draw[->](tree-8-2)--(tree-8-3)node[midway,above]{$\mathfrak{R}_2$};
\draw[->](tree-8-2)--(tree-9-3)node[midway,below]{$\mathfrak{R}_3$};
\end{tikzpicture}
\end{center}
For any $\bm{\sigma}\in\Gamma(\bm{\mu})$, there exists a simplest chain $\widetilde{\mathfrak{R}}=\mathfrak{R}_{i_{n}}\mathfrak{R}_{i_{n-1}}\cdots\mathfrak{R}_{i_{1}}$ with length $n$ such that $\bm{\sigma}=\widetilde{\mathfrak{R}}(\bm{0})$. In this case, we say that $\bm{\sigma}$ is in the $n$-th level. Now the conclusion is true for those $\bm{\sigma}$ in the $\mathcal{L}$-th level, $1\leq \mathcal{L}\leq 3$. By induction, we suppose that the conclusion holds for all $\bm{\sigma}^{(L)}$ in the $L$-th level, then we shall prove it also holds in the $(L+1)$-th level. Let $\bm{\sigma}^{(L)}=(\sigma_1^{(L)},\sigma_2^{(L)},\sigma_3^{(L)})$, where $\sigma_i^{(L)}=4\sum_{j=1}^{3}n_{ij}^{(L)}\mu_j$, $n_{ij}^{(L)}\in\mathbb{N}\cup\{0\}$. Then
\begin{equation*}
\left(\mathfrak{R}_1\bm{\sigma}^{(L)}\right)_{1}=\left(4-4n_{11}^{(L)}+8n_{31}^{(L)}\right)\mu_1
+\left(-4n_{12}^{(L)}+8n_{32}^{(L)}\right)\mu_2+\left(-4n_{13}^{(L)}+8n_{33}^{(L)}\right)\mu_3,
\end{equation*}
\begin{equation*}
\left(\mathfrak{R}_2\bm{\sigma}^{(L)}\right)_{2}=
\left(-4n_{21}^{(L)}+8n_{31}^{(L)}\right)\mu_1+\left(4-4n_{22}^{(L)}
+8n_{32}^{(L)}\right)\mu_2+\left(-4n_{23}^{(L)}+8n_{33}^{(L)}\right)\mu_3,
\end{equation*}
and
\begin{equation*}
\left(\mathfrak{R}_3\bm{\sigma}^{(L)}\right)_{3}=
\left(4n_{11}^{(L)}+4n_{21}^{(L)}-4n_{31}^{(L)}\right)\mu_1
+\left(4n_{12}^{(L)}+4n_{22}^{(L)}-4n_{32}^{(L)}\right)\mu_2
+\left(4+4n_{13}^{(L)}+4n_{23}^{(L)}-4n_{33}^{(L)}\right)\mu_3.
\end{equation*}
Obviously, the coefficients of $\mu_i$ in the above components are multiples of $4$. Thus we shall prove that these coefficients are nonnegative. We divide the proof into two steps as follows.
\medskip

\noindent\textbf{Step 1.} We claim
\begin{equation}\label{CYL-Ch-3-Eq-6}
4-4n_{11}^{(L)}+8n_{31}^{(L)}\geq 0, \quad -4n_{12}^{(L)}+8n_{32}^{(L)}\geq 0, \quad -4n_{13}^{(L)}+8n_{33}^{(L)}\geq 0,
\end{equation}
and
\begin{equation}\label{CYL-Ch-3-Eq-23}
-4n_{21}^{(L)}+8n_{31}^{(L)}\geq 0, \quad 4-4n_{22}^{(L)}+8n_{32}^{(L)}\geq 0, \quad -4n_{23}^{(L)}+8n_{33}^{(L)}\geq 0.
\end{equation}
Since the proof for \eqref{CYL-Ch-3-Eq-6} and \eqref{CYL-Ch-3-Eq-23} are almost the same, we shall only give details for the first one. Using \eqref{CYL-Ch-2-Eq-Pohozaev-identity} and substituting the expression of $\bm{\sigma}^{(L)}$ (or $\mathfrak{R}_1\bm{\sigma}^{(L)}$) into \eqref{CYL-Ch-2-Eq-Pohozaev-identity}, we can obtain the following three equations by comparing the coefficients of $\mu_{j}^2,~j=1,2,3$,
\begin{align}
\left(n_{11}^{(L)}-n_{31}^{(L)}\right)^2+\left(n_{21}^{(L)}-n_{31}^{(L)}\right)^{2}&=n_{11}^{(L)},
\label{CYL-Ch-3-Eq-7}\\
\left(n_{12}^{(L)}-n_{32}^{(L)}\right)^2+\left(n_{22}^{(L)}-n_{32}^{(L)}\right)^{2}&=n_{22}^{(L)},
\label{CYL-Ch-3-Eq-8}\\
\left(n_{13}^{(L)}-n_{33}^{(L)}\right)^2+\left(n_{23}^{(L)}-n_{33}^{(L)}\right)^{2}&=2n_{33}^{(L)}.
\label{CYL-Ch-3-Eq-9}
\end{align}
By \eqref{CYL-Ch-3-Eq-7}, we can regard $n_{11}^{(L)}$ as a solution of the following quadratic equation
\begin{equation}\label{CYL-Ch-3-Eq-10}
x^2-\left(2n_{31}^{(L)}+1\right)x+\left(n_{31}^{(L)}\right)^2+\left(n_{21}^{(L)}-n_{31}^{(L)}\right)^{2}=0.
\end{equation}
Since $\left(n_{31}^{(L)}\right)^2+\left(n_{21}^{(L)}-n_{31}^{(L)}\right)^{2}\geq 0$, then the other solution of \eqref{CYL-Ch-3-Eq-10} satisfies
\begin{equation*}
2n_{31}^{(L)}+1-n_{11}^{(L)}\geq 0~~\Longleftrightarrow~~4-4n_{11}^{(L)}+8n_{31}^{(L)}\geq 0.
\end{equation*}
In a similar manner, by \eqref{CYL-Ch-3-Eq-8} we can view $n_{12}^{(L)}$ as a solution of the following quadratic equation
\begin{equation}\label{CYL-Ch-3-Eq-11}
x^2-2n_{32}^{(L)}x+\left(n_{32}^{(L)}\right)^2+\left(n_{22}^{(L)}-n_{32}^{(L)}\right)^{2}-n_{22}^{(L)}=0.
\end{equation}
It is not difficult to check that
\begin{equation*}
\left(n_{32}^{(L)}\right)^2+\left(n_{22}^{(L)}-n_{32}^{(L)}\right)^{2}-n_{22}^{(L)}\geq 0.
\end{equation*}
Hence, the other solution of \eqref{CYL-Ch-3-Eq-11} satisfies
\begin{equation*}
2n_{32}^{(L)}-n_{12}^{(L)}\geq 0~~\Longleftrightarrow~~-4n_{12}^{(L)}+8n_{32}^{(L)}\geq 0.
\end{equation*}
As $n_{11}^{(L)}$ and $n_{12}^{(L)}$, by \eqref{CYL-Ch-3-Eq-9} we can treat $n_{13}^{(L)}$ as a solution of the following quadratic equation
\begin{equation}\label{CYL-Ch-3-Eq-12}
x^2-2n_{33}^{(L)}x+\left(n_{33}^{(L)}\right)^2+\left(n_{23}^{(L)}-n_{33}^{(L)}\right)^{2}-2n_{33}^{(L)}=0.
\end{equation}
Notice that if
\begin{equation}\label{CYL-Ch-3-Eq-24}
\left(n_{33}^{(L)}\right)^2+\left(n_{23}^{(L)}-n_{33}^{(L)}\right)^{2}-2n_{33}^{(L)}\geq 0,
\end{equation}
then the other solution of \eqref{CYL-Ch-3-Eq-12} satisfies
\begin{equation*}
2n_{33}^{(L)}-n_{13}^{(L)}\geq 0~~\Longleftrightarrow~~-4n_{13}^{(L)}+8n_{33}^{(L)}\geq 0.
\end{equation*}
Therefore, \eqref{CYL-Ch-3-Eq-6} is proved. So it remains to verify \eqref{CYL-Ch-3-Eq-24}. For the cases of $n_{33}^{(L)}\geq 2$ and $n_{33}^{(L)}=0$, it is trivial. While if $n_{33}^{(L)}=1$, we must have $\left(n_{23}^{(L)}-1\right)^{2}\geq 1$. Indeed, comparing the coefficients of $\mu_{2}\mu_{3}$ in the Pohozaev identity, we find $n_{23}^{(L)}$ is even. Hence \eqref{CYL-Ch-3-Eq-6} is proved.
\medskip

\noindent\textbf{Step 2.} We claim that
\begin{equation}\label{CYL-Ch-3-Eq-13}
4n_{11}^{(L)}+4n_{21}^{(L)}-4n_{31}^{(L)}\geq 0,~4n_{12}^{(L)}+4n_{22}^{(L)}-4n_{32}^{(L)}\geq 0,~4+4n_{13}^{(L)}+4n_{23}^{(L)}-4n_{33}^{(L)}\geq 0.
\end{equation}
As {\bf Step 1}, we can apply \eqref{CYL-Ch-2-Eq-Pohozaev-identity} to get that $n_{3i}^{(L)},~i=1,2,3$ are the solution to the following three equations respectively
\begin{equation*}
\left\{
\begin{array}{ll}
2x^2-2\left(n_{11}^{(L)}+n_{21}^{(L)}\right)x+\left(n_{11}^{(L)}\right)^2
+\left(n_{21}^{(L)}\right)^{2}-n_{11}^{(L)}=0,\\
2x^2-2\left(n_{12}^{(L)}+n_{22}^{(L)}\right)x+\left(n_{12}^{(L)}\right)^2
+\left(n_{22}^{(L)}\right)^{2}-n_{22}^{(L)}=0,\\
2x^2-2\left(n_{13}^{(L)}+n_{23}^{(L)}+1\right)x+\left(n_{13}^{(L)}\right)^2
+\left(n_{23}^{(L)}\right)^{2}=0.
\end{array}\right.
\end{equation*}
Notice that
\begin{equation*}
\left(n_{11}^{(L)}\right)^2+\left(n_{21}^{(L)}\right)^{2}-n_{11}^{(L)}\geq 0,~\left(n_{12}^{(L)}\right)^2+\left(n_{22}^{(L)}\right)^{2}-n_{22}^{(L)}\geq 0,~\left(n_{13}^{(L)}\right)^2+\left(n_{23}^{(L)}\right)^{2}\geq 0,
\end{equation*}
then by the same argument of \textbf{Step 1}, we conclude that
\begin{equation*}
n_{11}^{(L)}+n_{21}^{(L)}-n_{31}^{(L)}\geq 0,~n_{12}^{(L)}+n_{22}^{(L)}-n_{32}^{(L)}\geq 0,~1+n_{13}^{(L)}+n_{23}^{(L)}-n_{33}^{(L)}\geq 0.
\end{equation*}
Hence \eqref{CYL-Ch-3-Eq-13} is proved and we finish the proof.
\end{proof}

\begin{proposition}\label{CYL-Ch-3-Proposition-5}
Let $\Gamma_{N}(\bm{\mu})$ be defined as
\begin{equation}\label{CYL-Ch-3-Eq-5}
\Gamma_{N}(\bm{\mu})=\left\{\bm{\sigma} ~\big|~ \bm{\sigma}~\text{satisfies the Pohozaev identity~ \eqref{CYL-Ch-2-Eq-Pohozaev-identity}},
~\sigma_{i}=4\sum\limits_{j=1}^{3}n_{ij}\mu_{j},~n_{ij}\in\mathbb{N}\cup\{0\},~i=1,2,3\right\}.
\end{equation}
Then
\begin{equation*}
\Gamma_{N}(\bm{\mu})=\Gamma(\bm{\mu}).
\end{equation*}
\end{proposition}
\begin{proof}
It is easy to see that $\Gamma(\bm{\mu})\subseteq\Gamma_{N}(\bm{\mu})$ by Proposition \ref{CYL-Ch-3-Proposition-1} and Proposition \ref{CYL-Ch-3-Proposition-3}. Thus it remains to show $\Gamma_{N}(\bm{\mu})\subseteq\Gamma(\bm{\mu})$. For any $\bm{\sigma}=(\sigma_1,\sigma_2,\sigma_3)\in\Gamma_{N}(\bm{\mu})$, by Proposition \ref{CYL-Ch-3-Proposition-1}, $\mathfrak{R}_i\bm{\sigma}$ satisfies the Pohozaev identity \eqref{CYL-Ch-2-Eq-Pohozaev-identity}. So we shall prove
\begin{equation}\label{CYL-Ch-3-Eq-25}
\text{if}~\bm{\sigma}\in\Gamma_{N}(\bm{\mu}),~\text{then}~\mathfrak{R}_i\bm{\sigma}\in\Gamma_{N}(\bm{\mu}),\quad i=1,2,3.
\end{equation}
Let $\bm{\sigma}\in\Gamma_{N}(\bm{\mu})$, ${\sigma}_i=4\sum_{j=1}^{3}n_{ij}\mu_{j},~i=1,2,3$, where $n_{ij}\in\mathbb{N}\cup\{0\}$. Then by the same arguments  of Proposition \ref{CYL-Ch-3-Proposition-3}, we obtain
\begin{equation*}
\left\{
\begin{aligned}
4-4n_{11}+8n_{31}\geq 0,\\
-4n_{12}+8n_{32}\geq 0,\\
-4n_{13}+8n_{33}\geq 0,
\end{aligned}
\right.
\quad
\left\{
\begin{aligned}
-4n_{21}+8n_{31}\geq 0,\\
4-4n_{22}+8n_{32}\geq 0,\\
-4n_{23}+8n_{33}\geq 0,
\end{aligned}
\right.
~\text{and}~
\left\{
\begin{aligned}
4n_{11}+4n_{21}-4n_{31}\geq 0,\\
4n_{12}+4n_{22}-4n_{32}\geq 0,\\
4+4n_{13}+4n_{23}-4n_{33}\geq 0.
\end{aligned}
\right.
\end{equation*}
This completes the proof of \eqref{CYL-Ch-3-Eq-25}.

Next we define a partial order $\preceq$ in $\Gamma_{N}(\bm{\mu})$, we say
\begin{equation*}
\bm{\sigma}_{1}\preceq\bm{\sigma}_{2}~\text{provided}~(\bm{\sigma}_{1})_{i}\leq(\bm{\sigma}_{2})_{i},\quad i=1,2,3.
\end{equation*}
Then there must hold
\begin{equation*}
\text{either}~~~\bm{\sigma}\preceq\mathfrak{R}_{i}\bm{\sigma}~~~\text{or}~
~~\mathfrak{R}_{i}\bm{\sigma}\preceq\bm{\sigma},\quad i=1,2,3.
\end{equation*}
For any $\bm{\sigma}\in\Gamma_{N}(\bm{\mu})$, let
\begin{equation*}
\Gamma_{\bm{\sigma}}:=\left\{\mathfrak{R}_{i_{1}}\mathfrak{R}_{i_{2}}\cdots\mathfrak{R}_{i_{n}}\bm{\sigma}\mid n\in\mathbb{N}\cup\{0\},~i_{k}\in\{1,2,3\},~1\leq k\leq n\right\}.
\end{equation*}
Thus for any $\bm{\sigma}_{1},\bm{\sigma}_{2}\in\Gamma_{N}(\bm{\mu})$, we have either
\begin{equation}\label{CYL-Ch-3-Eq-14}
\Gamma_{\bm{\sigma}_{1}}=\Gamma_{\bm{\sigma}_{2}}~~~\text{or}~
~~\Gamma_{\bm{\sigma}_{1}}\cap\Gamma_{\bm{\sigma}_{2}}=\emptyset.
\end{equation}

Now we can prove that $\bm{0}=(0,0,0)\in \Gamma_{\bm{\sigma}}$ for any $\bm{\sigma}\in\Gamma_{N}(\bm{\mu})$. An element $\hat{\bm{\sigma}}\in\Gamma_{\bm{\sigma}}$ is minimal if $\widetilde{\bm{\sigma}}\preceq\hat{\bm{\sigma}}$ for some $\widetilde{\bm{\sigma}}\in\Gamma_{\bm{\sigma}}$, then $\widetilde{\bm{\sigma}}=\hat{\bm{\sigma}}$. By well-ordering Principle, we conclude that $\Gamma_{\bm{\sigma}}$ has a local minimal element $\bm{\sigma}_{0}=(\sigma_{1,0},\sigma_{2,0},\sigma_{3,0})$, i.e.,
\begin{equation*}
\bm{\sigma}_{0}\preceq\mathfrak{R}_{i}\bm{\sigma}_{0},\quad i=1,2,3.
\end{equation*}
Therefore, we have
\begin{equation*}
4\mu_{i}-2\sum\limits_{j=1}^{3}a_{ij}\sigma_{j,0}\geq 0,\quad i=1,2,3.
\end{equation*}
On the other hand, since $\bm{\sigma}_{0}$ satisfies the Pohozaev identity \eqref{CYL-Ch-2-Eq-Pohozaev-identity}, we get
\begin{equation*}
0\leq\sigma_{1,0}\left(2\sum\limits_{j=1}^{3}a_{1j}\sigma_{j,0}-4\mu_{1}\right)
+\sigma_{2,0}\left(2\sum\limits_{j=1}^{3}a_{2j}\sigma_{j,0}-4\mu_{2}\right)
+2\sigma_{3,0}\left(2\sum\limits_{j=1}^{3}a_{3j}\sigma_{j,0}-4\mu_{3}\right)\leq 0,
\end{equation*}
which implies that $\sigma_{i,0}=0,~i=1,2,3$. Otherwise, we must have
\begin{align*}
2\left(\sigma_{1,0}-\sigma_{3,0}\right)=4\mu_1,~2\left(\sigma_{2,0}-\sigma_{3,0}\right)=4\mu_2,
~2\left(-\frac{1}{2}\sigma_{1,0}-\frac{1}{2}\sigma_{2,0}+\sigma_{3,0}\right)=4\mu_3,
\end{align*}
then $2\mu_1+2\mu_2+4\mu_3=0$, a contradiction. Hence $\bm{0}\in\Gamma_{\bm{\sigma}}$. By \eqref{CYL-Ch-3-Eq-14}, we obtain $\Gamma_{\bm{\sigma}}=\Gamma_{\bm{0}}$ for any $\bm{\sigma}\in\Gamma_{N}(\bm{\mu})$. This amounts to say that
\begin{equation*}
\Gamma_{N}(\bm{\mu})=\Gamma(\bm{\mu}).
\end{equation*}
The proof of Proposition \ref{CYL-Ch-3-Proposition-5} is complete.
\end{proof}

In the following of this section we are focused on giving a  \textbf{precise expression} for any element in $\Gamma(\bm{\mu})$.

\begin{lemma}\label{CYL-Ch-3-Lemma-8}
Let $\gamma_i=0$, $i=1,2,3$. For any $\bm{\sigma}=(\sigma_1,\sigma_2,\sigma_3)\in\Gamma(1,1,1)$ with $\sigma_i=4n_i$ and $n_i\in\mathbb{N}\cup\{0\}$, $i=1,2,3$, we set
\begin{equation*}
n_1-n_3\equiv m_1^{\prime}~(\mathrm{mod}~4),\quad n_2-n_3\equiv m_2^{\prime}~(\mathrm{mod}~4)\quad
\mbox{for some}~m_1',m_2'\in\{0,1,2,3\}.
\end{equation*}
Then $(m_1^{\prime},m_2^{\prime})$ admits all the following $\mathbf{8}$ types:
\begin{equation*}
\Big\{(0,0),~(0,1),~(1,0),~(1,1),~(2,2),~(2,3),~(3,2),~(3,3)\Big\}.
\end{equation*}
In this way, we call that $\bm{\sigma}$ is of type $(m_1^{\prime},m_2^{\prime})$.
\end{lemma}
\begin{proof}
Applying the Pohozaev identity \eqref{CYL-Ch-2-Eq-Pohozaev-identity} with $\mu_i=1$, $i=1,2,3$, we obtain
\begin{equation}\label{CYL-Ch-3-Eq-20}
(n_1-n_3)^2+(n_2-n_3)^2=n_1+n_2+2n_3.
\end{equation}
Notice that
\begin{equation}\label{CYL-Ch-3-Eq-21}
n_1-n_3=4k_1+m_1^{\prime},~~n_2-n_3=4k_2+m_2^{\prime}
\end{equation}
for some $k_1, k_2\in\mathbb{Z}$ and $m_1^{\prime},m_2^{\prime}=0,1,2,3$. Substituting \eqref{CYL-Ch-3-Eq-21} into \eqref{CYL-Ch-3-Eq-20}, we conclude that
\begin{equation*}
{m_1^{\prime}}^2+{m_2^{\prime}}^2\equiv m_1^{\prime}+m_2^{\prime}~~(\mathrm{mod}~4),
\end{equation*}
which is solvable if and only if
\begin{equation*}
(m_1^{\prime},m_2^{\prime})\in \Big\{(0,0),~(0,1),~(1,0),~(1,1),~(2,2),~(2,3),~(3,2),~(3,3)\Big\}.
\end{equation*}
We notice that the following $8$ elements satisfy the Pohozaev identity \eqref{CYL-Ch-2-Eq-Pohozaev-identity} with $\mu_i=1$, $i=1,2,3$, and
\begin{equation*}
\begin{aligned}
(0,0,0)~\text{is of type}~(0,0),\quad (4,0,0)~\text{is of type}~(1,0),\\
(0,4,0)~\text{is of type}~(0,1),\quad (4,4,0)~\text{is of type}~(1,1),\\
(4,4,12)~\text{is of type}~(2,2),\quad (4,0,8)~\text{is of type}~(3,2),\\
(0,4,8)~\text{is of type}~(2,3),\quad (0,0,4)~\text{is of type}~(3,3).
\end{aligned}
\end{equation*}
Therefore, all the $8$ types could happen.
\end{proof}

\begin{rmk}\label{CYL-Ch-3-Remark-9}
(1) For any $\bm{\sigma}=(\sigma_1,\sigma_2,\sigma_3)\in\Gamma(1,1,1)$ which is of type $(m_1^{\prime},m_2^{\prime})$, one can easily check that $\mathfrak{R}_i\bm{\sigma}$ is of type $(n_1^{\prime},n_2^{\prime})$, where
\begin{equation*}
\begin{aligned}
\begin{cases}
n_1'\equiv-m_1'+1~(\mathrm{mod}~4),\quad n_2'=m_2',	
\qquad &\mathrm{if}~i=1,\\	
n_1'=m_1',\quad n_2'\equiv-m_2'+1~(\mathrm{mod}~4),\qquad &\mathrm{if}~i=2,\\
n_1'\equiv-m_2'-1~(\mathrm{mod}~4),\quad n_2'\equiv-m_1'-1~(\mathrm{mod}~4),\qquad &\mathrm{if}~i=3.
\end{cases}	
\end{aligned}	
\end{equation*}
Precisely, we conclude that
\begin{equation*}
\text{if}~\bm{\sigma}~\text{is of type}~(0,0),~\mathfrak{R}_1\bm{\sigma}~\text{is of type}~(1,0),~\mathfrak{R}_2\bm{\sigma}~\text{is of type}~(0,1),~\mathfrak{R}_3\bm{\sigma}~\text{is of type}~(3,3);
\end{equation*}
\begin{equation*}
\text{if}~\bm{\sigma}~\text{is of type}~(0,1),~\mathfrak{R}_1\bm{\sigma}~\text{is of type}~(1,1),~\mathfrak{R}_2\bm{\sigma}~\text{is of type}~(0,0),~\mathfrak{R}_3\bm{\sigma}~\text{is of type}~(2,3);
\end{equation*}
\begin{equation*}
\text{if}~\bm{\sigma}~\text{is of type}~(1,0),~\mathfrak{R}_1\bm{\sigma}~\text{is of type}~(0,0),~\mathfrak{R}_2\bm{\sigma}~\text{is of type}~(1,1),~\mathfrak{R}_3\bm{\sigma}~\text{is of type}~(3,2);
\end{equation*}
\begin{equation*}
\text{if}~\bm{\sigma}~\text{is of type}~(1,1),~\mathfrak{R}_1\bm{\sigma}~\text{is of type}~(0,1),~\mathfrak{R}_2\bm{\sigma}~\text{is of type}~(1,0),~\mathfrak{R}_3\bm{\sigma}~\text{is of type}~(2,2);
\end{equation*}
\begin{equation*}
\text{if}~\bm{\sigma}~\text{is of type}~(2,2),~\mathfrak{R}_1\bm{\sigma}~\text{is of type}~(3,2),~\mathfrak{R}_2\bm{\sigma}~\text{is of type}~(2,3),~\mathfrak{R}_3\bm{\sigma}~\text{is of type}~(1,1);
\end{equation*}
\begin{equation*}
\text{if}~\bm{\sigma}~\text{is of type}~(2,3),~\mathfrak{R}_1\bm{\sigma}~\text{is of type}~(3,3),~\mathfrak{R}_2\bm{\sigma}~\text{is of type}~(2,2),~\mathfrak{R}_3\bm{\sigma}~\text{is of type}~(0,1);
\end{equation*}
\begin{equation*}
\text{if}~\bm{\sigma}~\text{is of type}~(3,2),~\mathfrak{R}_1\bm{\sigma}~\text{is of type}~(2,2),~\mathfrak{R}_2\bm{\sigma}~\text{is of type}~(3,3),~\mathfrak{R}_3\bm{\sigma}~\text{is of type}~(1,0);
\end{equation*}
\begin{equation*}
\text{if}~\bm{\sigma}~\text{is of type}~(3,3),~\mathfrak{R}_1\bm{\sigma}~\text{is of type}~(2,3),~\mathfrak{R}_2\bm{\sigma}~\text{is of type}~(3,2),~\mathfrak{R}_3\bm{\sigma}~\text{is of type}~(0,0).
\end{equation*}

\noindent (2) For any general $\bm{\sigma}\in\Gamma(\bm{\mu})$, $\sigma_{i}=4\sum_{j=1}^{3}n_{ij}\mu_{j},~n_{ij}\in\mathbb{N}\cup\{0\}$. Let $N({\sigma}_i)=4\sum_{j=1}^{3}n_{ij}$ and denote
\begin{equation*}
\frac{1}{4}\left(N({\sigma}_1)-N({\sigma}_3)\right)\equiv m_1^{\prime}~(\mathrm{mod}~4),\quad \frac{1}{4}\left(N({\sigma}_2)-N({\sigma}_3)\right)\equiv m_2^{\prime}~(\mathrm{mod}~4)\quad \mbox{for some}~m_1',m_2'\in\{0,1,2,3\}.
\end{equation*}
Then $(m_1^{\prime},m_2^{\prime})$ also admits all the following $\mathbf{8}$ types:
\begin{equation*}
\Big\{(0,0),~(0,1),~(1,0),~(1,1),~(2,2),~(2,3),~(3,2),~(3,3)\Big\}.
\end{equation*}
We will prove this result in Lemma \ref{CYL-Ch-3-Lemma-10}. In this way, we say that $\bm{\sigma}$ is of type $(m_1^{\prime},m_2^{\prime})$ as well.
\end{rmk}

\begin{lemma}\label{CYL-Ch-3-Lemma-10}
For any $\bm{\sigma}\in\Gamma(\bm{\mu})$, suppose that $\bm{\sigma}$ is of type $(m_1^{\prime},m_2^{\prime})$ for some $m_1^{\prime},m_2^{\prime}=0,1,2,3$. Then $(m_1^{\prime},m_2^{\prime})$ admits all the following $\mathbf{8}$ types:
\begin{equation}\label{CYL-Ch-3-Eq-28}
\Big\{(0,0),~(0,1),~(1,0),~(1,1),~(2,2),~(2,3),~(3,2),~(3,3)\Big\}.
\end{equation}
\end{lemma}
\begin{proof}
We prove this result by induction. Recall the binary tree diagram given in the proof of Proposition \ref{CYL-Ch-3-Proposition-3} and we can directly check that
\begin{equation*}
\begin{aligned}
(0,0,0)~\text{is of type}~(0,0),\quad &(4\mu_1,0,0)~\text{is of type}~(1,0),\\
(0,4\mu_2,0)~\text{is of type}~(0,1),\quad &(0,0,4\mu_3)~\text{is of type}~(3,3),\\
(4\mu_1,4\mu_2,0)~\text{is of type}~(1,1),\quad &(4\mu_1,0,4\mu_1+4\mu_3)~\text{is of type}~(3,2),\\
(0,4\mu_2,4\mu_2+4\mu_3)~\text{is of type}~(2,3),\quad &(4\mu_1+8\mu_3,0,4\mu_3)~\text{is of type}~(2,3),\\
(0,4\mu_2+8\mu_3,4\mu_3)~\text{is of type}~(3,2),\quad &(4\mu_1,4\mu_2,4\mu_1+4\mu_2+4\mu_3)~\text{is of type}~(2,2).\\
\end{aligned}
\end{equation*}
The first $9$ elements are in the $1$-st, $2$-nd and $3$-rd levels, the last one is in the $4$-th level. Obviously, these elements are of type $(i,j)$ described as in \eqref{CYL-Ch-3-Eq-28} and all $8$ types could happen. We suppose that each element in the $N$-th level is of type $(i,j)$, which is one of \eqref{CYL-Ch-3-Eq-28}. Then we shall prove each element in the $(N+1)$-th level is of type $(i^{\prime},j^{\prime})$, where $(i^{\prime},j^{\prime})$ admits a form belonging to \eqref{CYL-Ch-3-Eq-28}.

Suppose $\bm{\sigma}=(\sigma_1,\sigma_2,\sigma_3)\in\Gamma(\bm{\mu})$ is an element in the $(N+1)$-level, where $\sigma_{i}=4\sum_{j=1}^{3}n_{ij}\mu_{j},~n_{ij}\in\mathbb{N}\cup\{0\}$. Without loss of generality, we assume that $\bm{\sigma}$ is of type $(0,0)$. Then one can easily check that
$\mathfrak{R}_i\bm{\sigma}$ is of type $(n_1^{\prime},n_2^{\prime})$, where for $i=1,2,3$, respectively,
\begin{equation*}
\begin{cases}
n_1^{\prime}\equiv-m_1^{\prime}+1~(\mathrm{mod}~4),\\
n_2^{\prime}=m_2^{\prime},
\end{cases}\qquad
\begin{cases}
n_1^{\prime}=m_1^{\prime},\\
n_2^{\prime}\equiv-m_2^{\prime}+1~(\mathrm{mod}~4),\\
\end{cases}\qquad
\begin{cases}
n_1^{\prime}\equiv-m_2^{\prime}-1~(\mathrm{mod}~4),\\
n_2^{\prime}\equiv-m_1^{\prime}-1~(\mathrm{mod}~4).\\
\end{cases}
\end{equation*}
Then we get the same conclusion holds by Remark \ref{CYL-Ch-3-Remark-9}-(1). Thus, we finish the proof.
\end{proof}

\begin{theorem}\label{CYL-Ch-3-Theorem-6}
For each element $\bm{\sigma}\in\Gamma(\bm{\mu})$, we can express $\bm{\sigma}$ as $\bm{\sigma}^{t}=\mathbb{F}\bm{\mu}^{t}$, where $\mathbb{F}=(f_{ij})_{3\times 3}$ is a matrix, $f_{ij}$ are multiples of $4$ by Proposition \ref{CYL-Ch-3-Proposition-5}, and $t$ denotes the transpose operator, i.e.,
\begin{equation*}
\begin{aligned}
\left(
\begin{array}{ccc}
\sigma_1\\
\sigma_2\\
\sigma_3
\end{array}
\right)
\end{aligned}
=\mathbb{F}\bm{\mu}^{t}=
\begin{aligned}
\left(
\begin{array}{ccc}
f_{11} & f_{12} & f_{13}\\
f_{21} & f_{22} & f_{23}\\
f_{31} & f_{32} & f_{33}
\end{array}
\right)
\end{aligned}
\begin{aligned}
\left(
\begin{array}{ccc}
\mu_1\\
\mu_2\\
\mu_3
\end{array}
\right).
\end{aligned}
\end{equation*}
Let
\begin{equation*}
\Gamma_{\mathbb{F}}(\bm{\mu})=\left\{\bm{\sigma} ~|~ \bm{\sigma}^{t}=\mathbb{F}\bm{\mu}^{t},~\text{where}~\mathbb{F}~\text{admits}~\mathbb{F}^{(\ell)},~\ell=1,\cdots,8\right\}.
\end{equation*}
Then $\Gamma_{\mathbb{F}}(\bm{\mu})=\Gamma_{N}(\bm{\mu})$. Here $\mathbb{F}^{(\ell)}$ are defined as the following:
\\

\noindent \textbf{Type~$(0,0)$.} For any $m_1,m_2\in \mathbb{Z}$ with $(m_1,m_2)\equiv(0,0)~(\mathrm{mod}~4)$, we have
\begin{equation*}
\mathbb{F}^{(1)}=\begin{aligned}
\left(
\begin{array}{lll}
\frac{1}{4}m_1^2+\frac{1}{4}m_2^2 &
\frac{1}{4}m_1^2+m_1+\frac{1}{4}m_2^2-m_2 &
\frac{1}{2}m_1^2+2m_1+\frac{1}{2}m_2^2\\
\frac{1}{4}m_1^2-m_1+\frac{1}{4}m_2^2+m_2 &
\frac{1}{4}m_1^2+\frac{1}{4}m_2^2 &
\frac{1}{2}m_1^2+\frac{1}{2}m_2^2+2m_2\\
\frac{1}{4}m_1^2-m_1+\frac{1}{4}m_2^2 &
\frac{1}{4}m_1^2+\frac{1}{4}m_2^2-m_2 &
\frac{1}{2}m_1^2+\frac{1}{2}m_2^2
\end{array}
\right).
\end{aligned}
\end{equation*}

\noindent \textbf{Type~$(0,1)$.} For any $m_1,m_2\in \mathbb{Z}$ with $(m_1,m_2)\equiv(0,1)~(\mathrm{mod}~4)$, we have
\begin{equation*}
\mathbb{F}^{(2)}=\begin{aligned}
\left(
\begin{array}{lll}
\frac{1}{4}m_1^2+\frac{1}{4}m_2^2-\frac{1}{2}m_2+\frac{1}{4} &
\frac{1}{4}m_1^2+m_1+\frac{1}{4}m_2^2+\frac{1}{2}m_2-\frac{3}{4} &
\frac{1}{2}m_1^2+2m_1+\frac{1}{2}m_2^2-m_2+\frac{1}{2} \\
\frac{1}{4}m_1^2-m_1+\frac{1}{4}m_2^2+\frac{1}{2}m_2-\frac{3}{4} &
\frac{1}{4}m_1^2+\frac{1}{4}m_2^2+\frac{3}{2}m_2+\frac{9}{4} &
\frac{1}{2}m_1^2+\frac{1}{2}m_2^2+m_2-\frac{3}{2}\\
\frac{1}{4}m_1^2-m_1+\frac{1}{4}m_2^2-\frac{1}{2}m_2+\frac{1}{4} &
\frac{1}{4}m_1^2+\frac{1}{4}m_2^2+\frac{1}{2}m_2-\frac{3}{4} &
\frac{1}{2}m_1^2+\frac{1}{2}m_2^2-m_2+\frac{1}{2}
  \end{array}
\right).
\end{aligned}
\end{equation*}

\noindent \textbf{Type~$(1,0)$.} For any $m_1,m_2\in \mathbb{Z}$ with $(m_1,m_2)\equiv(1,0)~(\mathrm{mod}~4)$, we have
\begin{equation*}
\mathbb{F}^{(3)}=\begin{aligned}
\left(
\begin{array}{lll}
\frac{1}{4}m_1^2+\frac{3}{2}m_1+\frac{1}{4}m_2^2+\frac{9}{4} &
\frac{1}{4}m_1^2+\frac{1}{2}m_1+\frac{1}{4}m_2^2-m_2-\frac{3}{4} &
\frac{1}{2}m_1^2+m_1+\frac{1}{2}m_2^2-\frac{3}{2} \\
\frac{1}{4}m_1^2+\frac{1}{2}m_1+\frac{1}{4}m_2^2+m_2-\frac{3}{4} &
\frac{1}{4}m_1^2-\frac{1}{2}m_1+\frac{1}{4}m_2^2+\frac{1}{4} &
\frac{1}{2}m_1^2-m_1+\frac{1}{2}m_2^2+2m_2+\frac{1}{2}\\
\frac{1}{4}m_1^2+\frac{1}{2}m_1+\frac{1}{4}m_2^2-\frac{3}{4} &
\frac{1}{4}m_1^2-\frac{1}{2}m_1+\frac{1}{4}m_2^2-m_2+\frac{1}{4} &
\frac{1}{2}m_1^2-m_1+\frac{1}{2}m_2^2+\frac{1}{2}
  \end{array}
\right).
\end{aligned}
\end{equation*}

\noindent \textbf{Type~$(1,1)$.} For any $m_1,m_2\in \mathbb{Z}$ with $(m_1,m_2)\equiv(1,1)~(\mathrm{mod}~4)$, we have
\begin{equation*}
\mathbb{F}^{(4)}=\begin{aligned}
\left(
\begin{array}{lll}
\frac{1}{4}m_1^2+\frac{3}{2}m_1+\frac{1}{4}m_2^2-\frac{1}{2}m_2+\frac{5}{2} &
\frac{1}{4}m_1^2+\frac{1}{2}m_1+\frac{1}{4}m_2^2+\frac{1}{2}m_2-\frac{3}{2} &
\frac{1}{2}m_1^2+m_1+\frac{1}{2}m_2^2-m_2-1\\
\frac{1}{4}m_1^2+\frac{1}{2}m_1+\frac{1}{4}m_2^2+\frac{1}{2}m_2-\frac{3}{2} &
\frac{1}{4}m_1^2-\frac{1}{2}m_1+\frac{1}{4}m_2^2+\frac{3}{2}m_2+\frac{5}{2} &
\frac{1}{2}m_1^2-m_1+\frac{1}{2}m_2^2+m_2-1\\
\frac{1}{4}m_1^2+\frac{1}{2}m_1+\frac{1}{4}m_2^2-\frac{1}{2}m_2-\frac{1}{2} &
\frac{1}{4}m_1^2-\frac{1}{2}m_1+\frac{1}{4}m_2^2+\frac{1}{2}m_2-\frac{1}{2} &
\frac{1}{2}m_1^2-m_1+\frac{1}{2}m_2^2-m_2+1
  \end{array}
\right).
\end{aligned}
\end{equation*}

\noindent \textbf{Type~$(2,2)$.} For any $m_1,m_2\in \mathbb{Z}$ with $(m_1,m_2)\equiv(2,2)~(\mathrm{mod}~4)$, we have
\begin{equation*}
\mathbb{F}^{(5)}=\begin{aligned}
\left(
\begin{array}{lll}
\frac{1}{4}m_1^2+m_1+\frac{1}{4}m_2^2-m_2+2 &
\frac{1}{4}m_1^2+\frac{1}{4}m_2^2-2 &
\frac{1}{2}m_1^2+2m_1+\frac{1}{2}m_2^2\\
\frac{1}{4}m_1^2+\frac{1}{4}m_2^2-2 &
\frac{1}{4}m_1^2-m_1+\frac{1}{4}m_2^2+m_2+2 &
\frac{1}{2}m_1^2+\frac{1}{2}m_2^2+2m_2\\
\frac{1}{4}m_1^2+\frac{1}{4}m_2^2-m_2 &
\frac{1}{4}m_1^2-m_1+\frac{1}{4}m_2^2 &
\frac{1}{2}m_1^2+\frac{1}{2}m_2^2
\end{array}
\right).
\end{aligned}
\end{equation*}

\noindent \textbf{Type~$(2,3)$.} For any $m_1,m_2\in \mathbb{Z}$ with $(m_1,m_2)\equiv(2,3)~(\mathrm{mod}~4)$, we have
\begin{equation*}
\mathbb{F}^{(6)}=\begin{aligned}
\left(
\begin{array}{lll}
\frac{1}{4}m_1^2+m_1+\frac{1}{4}m_2^2+\frac{1}{2}m_2+\frac{5}{4} &
\frac{1}{4}m_1^2+\frac{1}{4}m_2^2-\frac{1}{2}m_2-\frac{7}{4} &
\frac{1}{2}m_1^2+2m_1+\frac{1}{2}m_2^2-m_2+\frac{1}{2} \\
\frac{1}{4}m_1^2+\frac{1}{4}m_2^2+\frac{3}{2}m_2+\frac{1}{4} &
\frac{1}{4}m_1^2-m_1+\frac{1}{4}m_2^2+\frac{1}{2}m_2+\frac{5}{4} &
\frac{1}{2}m_1^2+\frac{1}{2}m_2^2+m_2-\frac{3}{2}\\
\frac{1}{4}m_1^2+\frac{1}{4}m_2^2+\frac{1}{2}m_2-\frac{3}{4} &
\frac{1}{4}m_1^2-m_1+\frac{1}{4}m_2^2-\frac{1}{2}m_2+\frac{1}{4} &
\frac{1}{2}m_1^2+\frac{1}{2}m_2^2-m_2+\frac{1}{2}
\end{array}
\right).
\end{aligned}
\end{equation*}

\noindent \textbf{Type~$(3,2)$.} For any $m_1,m_2\in \mathbb{Z}$ with $(m_1,m_2)\equiv(3,2)~(\mathrm{mod}~4)$, we have
\begin{equation*}
\mathbb{F}^{(7)}=\begin{aligned}
\left(
\begin{array}{lll}
\frac{1}{4}m_1^2+\frac{1}{2}m_1+\frac{1}{4}m_2^2-m_2+\frac{5}{4} &
\frac{1}{4}m_1^2+\frac{3}{2}m_1+\frac{1}{4}m_2^2+\frac{1}{4} &
\frac{1}{2}m_1^2+m_1+\frac{1}{2}m_2^2-\frac{3}{2} \\
\frac{1}{4}m_1^2-\frac{1}{2}m_1+\frac{1}{4}m_2^2-\frac{7}{4} &
\frac{1}{4}m_1^2+\frac{1}{2}m_1+\frac{1}{4}m_2^2+m_2+\frac{5}{4} &
\frac{1}{2}m_1^2-m_1+\frac{1}{2}m_2^2+2m_2+\frac{1}{2}\\
\frac{1}{4}m_1^2-\frac{1}{2}m_1+\frac{1}{4}m_2^2-m_2+\frac{1}{4} &
\frac{1}{4}m_1^2+\frac{1}{2}m_1+\frac{1}{4}m_2^2-\frac{3}{4} &
\frac{1}{2}m_1^2-m_1+\frac{1}{2}m_2^2+\frac{1}{2}
  \end{array}
\right).
\end{aligned}
\end{equation*}

\noindent \textbf{Type~$(3,3)$.} For any $m_1,m_2\in \mathbb{Z}$ with $(m_1,m_2)\equiv(3,3)~(\mathrm{mod}~4)$, we have
\begin{equation*}
\mathbb{F}^{(8)}=\begin{aligned}
\left(
\begin{array}{lll}
\frac{1}{4}m_1^2+\frac{1}{2}m_1+\frac{1}{4}m_2^2+\frac{1}{2}m_2+\frac{1}{2} & \frac{1}{4}m_1^2+\frac{3}{2}m_1+\frac{1}{4}m_2^2-\frac{1}{2}m_2+\frac{1}{2} &
\frac{1}{2}m_1^2+m_1+\frac{1}{2}m_2^2-m_2-1\\
\frac{1}{4}m_1^2-\frac{1}{2}m_1+\frac{1}{4}m_2^2+\frac{3}{2}m_2+\frac{1}{2} &
\frac{1}{4}m_1^2+\frac{1}{2}m_1+\frac{1}{4}m_2^2+\frac{1}{2}m_2+\frac{1}{2} &
\frac{1}{2}m_1^2-m_1+\frac{1}{2}m_2^2+m_2-1\\
\frac{1}{4}m_1^2-\frac{1}{2}m_1+\frac{1}{4}m_2^2+\frac{1}{2}m_2-\frac{1}{2} &
\frac{1}{4}m_1^2+\frac{1}{2}m_1+\frac{1}{4}m_2^2-\frac{1}{2}m_2-\frac{1}{2} &
\frac{1}{2}m_1^2-m_1+\frac{1}{2}m_2^2-m_2+1
  \end{array}
\right).
\end{aligned}
\end{equation*}
\end{theorem}
\begin{proof}
Let
\begin{equation*}
\Gamma_{\mathbb{F}}(\bm{\mu})=\left\{\bm{\sigma} ~|~ \bm{\sigma}^{t}=\mathbb{F}\bm{\mu}^{t}~\text{where}~\mathbb{F}~\text{admits}~\mathbb{F}^{(\ell)},~\ell=1,\cdots,8\right\}.
\end{equation*}
Our aim is to show $\Gamma_{\mathbb{F}}(\bm{\mu})=\Gamma_{N}(\bm{\mu})$.

On one hand, for any $\bm{\sigma}\in\Gamma_{\mathbb{F}}(\bm{\mu})$ with $\bm{\sigma}^{t}=\mathbb{F}^{(\ell)}\bm{\mu}^{t}$ for some $\ell\in\{1,\cdots,8\}$, one can directly check that $\bm{\sigma}$ satisfies the Pohozaev identity \eqref{CYL-Ch-2-Eq-Pohozaev-identity} and $f_{ij}^{(\ell)}$ are multiples of $4$. Hence $\Gamma_{\mathbb{F}}(\bm{\mu})\subseteq\Gamma_{N}(\bm{\mu})$ by the definition of $\Gamma_{N}(\bm{\mu})$ in \eqref{CYL-Ch-3-Eq-5}.

On the other hand, we can prove $\Gamma_{N}(\bm{\mu})\subseteq\Gamma_{\mathbb{F}}(\bm{\mu})$ by induction. For any $\bm{\sigma}\in\Gamma_{N}(\bm{\mu})$, $\sigma_{i}=4\sum_{j=1}^{3}n_{ij}\mu_{j}$, $n_{ij}\in\mathbb{N}\cup\{0\}$. Let $N({\sigma}_i)=4\sum_{j=1}^{3}n_{ij}$ and denote
\begin{equation}\label{CYL-Ch-3-Eq-26}
\frac{1}{4}\left(N({\sigma}_1)-N({\sigma}_3)\right)=m_1,\quad \frac{1}{4}\left(N({\sigma}_2)-N({\sigma}_3)\right)=m_2.
\end{equation}
Suppose $\bm{\sigma}$ is of type $(m_1^{\prime},m_2^{\prime})$, where $(m_1,m_2)\equiv (m_1^{\prime},m_2^{\prime})$ (mod $4$). We recall the binary tree diagram given in the proof of Proposition \ref{CYL-Ch-3-Proposition-3} and directly check that
\begin{equation*}
\begin{aligned}
\bm{0}=(0,0,0)~\text{is of type}~(0,0),&\quad\bm{0}^{t}=\mathbb{F}^{(1)}\bm{\mu}^{t}~\text{for}~(m_1,m_2)=(0,0);\\
\bm{\sigma}=(4\mu_1,0,0)~\text{is of type}~(1,0),&\quad\bm{\sigma}^{t}=\mathbb{F}^{(3)}\bm{\mu}^{t}~\text{for}~(m_1,m_2)=(1,0);\\
\bm{\sigma}=(0,4\mu_2,0)~\text{is of type}~(0,1),&\quad\bm{\sigma}^{t}=\mathbb{F}^{(2)}\bm{\mu}^{t}~\text{for}~(m_1,m_2)=(0,1);\\
\bm{\sigma}=(0,0,4\mu_3)~\text{is of type}~(3,3),&\quad\bm{\sigma}^{t}=\mathbb{F}^{(8)}\bm{\mu}^{t}~\text{for}~(m_1,m_2)=(-1,-1);\\
\bm{\sigma}=(4\mu_1,4\mu_2,0)~\text{is of type}~(1,1),&\quad\bm{\sigma}^{t}=\mathbb{F}^{(4)}\bm{\mu}^{t}~\text{for}~(m_1,m_2)=(1,1);\\
\bm{\sigma}=(4\mu_1,0,4\mu_1+4\mu_3)~\text{is of type}~(3,2),&\quad\bm{\sigma}^{t}=\mathbb{F}^{(7)}\bm{\mu}^{t}~\text{for}~(m_1,m_2)=(-1,-2);\\
\bm{\sigma}=(0,4\mu_2,4\mu_2+4\mu_3)~\text{is of type}~(2,3),&\quad\bm{\sigma}^{t}=\mathbb{F}^{(6)}\bm{\mu}^{t}~\text{for}~(m_1,m_2)=(-2,-1);\\
\bm{\sigma}=(4\mu_1+8\mu_3,0,4\mu_3)~\text{is of type}~(2,3),&\quad\bm{\sigma}^{t}=\mathbb{F}^{(6)}\bm{\mu}^{t}~\text{for}~(m_1,m_2)=(2,-1);\\
\bm{\sigma}=(0,4\mu_2+8\mu_3,4\mu_3)~\text{is of type}~(3,2),&\quad\bm{\sigma}^{t}=\mathbb{F}^{(7)}\bm{\mu}^{t}~\text{for}~(m_1,m_2)=(-1,2);\\
\bm{\sigma}=(4\mu_1,4\mu_2,4\mu_1+4\mu_2+4\mu_3)~\text{is of type}~(2,2),&\quad\bm{\sigma}^{t}=\mathbb{F}^{(5)}\bm{\mu}^{t}~\text{for}~(m_1,m_2)=(-2,-2).
\end{aligned}
\end{equation*}
This implies that $\bm{\sigma}\in\Gamma_{\mathbb{F}}(\bm{\mu})$ for those $\bm{\sigma}$ in the $\mathcal{L}$-th level with $\bm{\sigma}\in\Gamma_{N}(\bm{\mu})$, $1\leq\mathcal{L}\leq 3$. In particular, all possibilities $\bm{\sigma}^{t}=\mathbb{F}^{(\ell)}\bm{\mu}^{t}$ could happen, $\ell=1,\cdots,8$. Suppose that $\bm{\sigma}\in\Gamma_{\mathbb{F}}(\bm{\mu})$ for those $\bm{\sigma}$ in the $\mathcal{L}$-th level with $\bm{\sigma}\in\Gamma_{N}(\bm{\mu})$, $1\leq\mathcal{L}\leq M$. We need to verify $\mathfrak{R}_i\bm{\sigma}\in\Gamma_{\mathbb{F}}(\bm{\mu})$ for any $\bm{\sigma}\in\Gamma_{N}(\bm{\mu})$ in $M$-th level and any $i=1,2,3$.

Now, suppose that $\bm{\sigma}$ is an element in the $M$-th level of $\Gamma_{N}(\bm{\mu})$, by assumption we have  $\bm{\sigma}\in\Gamma_{\mathbb{F}}(\bm{\mu})$ with $\bm{\sigma}^{t}=\mathbb{F}^{(\ell)}\bm{\mu}^{t}$ for some $\ell\in\{1,\cdots,8\}$. Denote
\begin{equation*}
\sigma_i=\sum\limits_{j=1}^{3}f^{(\ell)}_{ij}(m_1,m_2)\mu_j,~\text{where}~f^{(\ell)}_{ij}(m_1,m_2)~\text{is the element of}~\mathbb{F}^{(\ell)}.
\end{equation*}
Assume that $\bm{\sigma}$ is of type $(m_1^{\prime},m_2^{\prime})$, where $(m_1,m_2)\equiv (m_1^{\prime},m_2^{\prime})$ ($\mathrm{mod}$ $4$), and $\mathfrak{R}_i\bm{\sigma}$ is of type $(n_1^{\prime},n_2^{\prime})$, where $(n_1,n_2)\equiv (n_1^{\prime},n_2^{\prime})$ ($\mathrm{mod}$ $4$). Here $(m_1,m_2)$ and $(n_1,n_2)$ are defined as in \eqref{CYL-Ch-3-Eq-26} related to $\bm{\sigma}$ and $\mathfrak{R}_i\bm{\sigma}$, respectively. The corresponding matrices $\mathbb{F}$ of type $(m_1^{\prime},m_2^{\prime})$ and $(n_1^{\prime},n_2^{\prime})$ are $\mathbb{F}^{(\ell)}$ and $\mathbb{F}^{(\ell^{\prime})}$ respectively. One can directly check that
\begin{equation}\label{CYL-Ch-3-Eq-27}
\left(\mathfrak{R}_i\bm{\sigma}\right)^{t}=\left\{
\begin{aligned}
&\mathbb{F}^{(\ell_1)}(1-m_1,m_2)\bm{\mu}^{t},\quad&\text{if}~i=1,\\
&\mathbb{F}^{(\ell_2)}(m_1,1-m_2)\bm{\mu}^{t},\quad&\text{if}~i=2,\\
&\mathbb{F}^{(\ell_3)}(-m_2-1,-m_1-1)\bm{\mu}^{t},\quad&\text{if}~i=3,
\end{aligned}
\right.
\end{equation}
for some indices $\ell_i\in\{1,\cdots,8\},~i=1,2,3.$

In the following, we just provide the details for the case $l=1$. By straightforward computation we have
\begin{equation*}
\left(\mathfrak{R}_1\bm{\sigma}\right)^{t}=\mathbb{F}^{(3)}(1-m_1,m_2)\bm{\mu}^{t},\quad
~\left(\mathfrak{R}_2\bm{\sigma}\right)^{t}=\mathbb{F}^{(2)}(m_1,1-m_2)\bm{\mu}^{t},\quad
~\left(\mathfrak{R}_3\bm{\sigma}\right)^{t}=\mathbb{F}^{(8)}(-m_2-1,-m_1-1)\bm{\mu}^{t}.
\end{equation*}
Denote $\bm{\sigma}=(\sigma_1,\sigma_2,\sigma_3)$, where
\begin{equation*}
\sigma_i=\sum\limits_{j=1}^{3}f^{(1)}_{ij}(m_1,m_2)\mu_j,~\text{where}~f^{(1)}_{ij}(m_1,m_2)~\text{is the element of}~\mathbb{F}^{(1)}.
\end{equation*}
Then $\mathfrak{R}_1\bm{\sigma}$, $\mathfrak{R}_2\bm{\sigma}$ and $\mathfrak{R}_3\bm{\sigma}$ is of type $(1,0)$, $(0,1)$ and $(3,3)$ respectively by Remark \ref{CYL-Ch-3-Remark-9}, and hence ${\ell}_1=3$, ${\ell}_2=2$ and ${\ell}_3=8$. On one hand, direct computation shows $\left(\mathfrak{R}_1\bm{\sigma}\right)^{t}=\mathbb{A}\bm{\mu}^{t}$ with
\begin{equation*}
\mathbb{A}=\begin{aligned}
\left(
\begin{array}{lll}
\frac{1}{4}\left(m_1-4\right)^2+\frac{1}{4}m_2^2&
\frac{1}{4}\left(m_1-2\right)^2+\frac{1}{4}\left(m_2-2\right)^2-2 &
\frac{1}{2}\left(m_1-2\right)^2+\frac{1}{2}m_2^2-2 \\
\frac{1}{4}\left(m_1-2\right)^2+\frac{1}{4}\left(m_2+2\right)^2-2 &
\frac{1}{4}m_1^2+\frac{1}{4}m_2^2 &
\frac{1}{2}m_1^2+\frac{1}{2}\left(m_2+2\right)^2-2\\
\frac{1}{4}\left(m_1-2\right)^2+\frac{1}{4}m_1^2-1 &
\frac{1}{4}m_1^2+\frac{1}{4}\left(m_2-2\right)^2-1 &
\frac{1}{2}m_1^2+\frac{1}{2}m_2^2
  \end{array}
\right).
\end{aligned}
\end{equation*}
On the other hand, we can rewrite $\mathbb{F}^{(3)}$ as
\begin{equation*}
\mathbb{F}^{(3)}=\begin{aligned}
\left(
\begin{array}{lll}
\frac{1}{4}\left(m_1+3\right)^2+\frac{1}{4}m_2^2 &
\frac{1}{4}\left(m_1+1\right)^2+\frac{1}{4}\left(m_2-2\right)^2-2 &
\frac{1}{2}\left(m_1+1\right)^2+\frac{1}{2}m_2^2-2 \\
\frac{1}{4}\left(m_1+1\right)^2+\frac{1}{4}\left(m_2+2\right)^2-2 &
\frac{1}{4}\left(m_1-1\right)^2+\frac{1}{4}m_2^2 &
\frac{1}{2}\left(m_1-1\right)^2+\frac{1}{2}\left(m_2+2\right)^2-2\\
\frac{1}{4}\left(m_1+1\right)^2+\frac{1}{4}m_2^2-1 &
\frac{1}{4}\left(m_1-1\right)^2+\frac{1}{4}\left(m_2-2\right)^2-1 &
\frac{1}{2}\left(m_1-1\right)^2+\frac{1}{2}m_2^2
  \end{array}
\right).
\end{aligned}
\end{equation*}
Comparing the above two matrices, it is easy to observe that
\begin{equation*}
\left(\mathfrak{R}_1\bm{\sigma}\right)^{t}=\mathbb{F}^{(3)}(1-m_1,m_2)\bm{\mu}^{t}.
\end{equation*}
Similarly, we obtain
\begin{equation*}
\left(\mathfrak{R}_2\bm{\sigma}\right)^{t}=\mathbb{F}^{(2)}(m_1,1-m_2)\bm{\mu}^{t}
~\text{and}~\left(\mathfrak{R}_3\bm{\sigma}\right)^{t}=\mathbb{F}^{(8)}(-m_2-1,-m_1-1)\bm{\mu}^{t}.
\end{equation*}
Then we finish the discussion for the case of $\ell=1$. For the other cases, by direct computation one can get:
\begin{enumerate}
\item [(1).] If $\ell=1$,
\begin{equation*}
\left(\mathfrak{R}_1\bm{\sigma}\right)^{t}=\mathbb{F}^{(3)}(1-m_1,m_2)\bm{\mu}^{t},
~\left(\mathfrak{R}_2\bm{\sigma}\right)^{t}=\mathbb{F}^{(2)}(m_1,1-m_2)\bm{\mu}^{t},
~\left(\mathfrak{R}_3\bm{\sigma}\right)^{t}=\mathbb{F}^{(8)}(-1-m_2,-1-m_1)\bm{\mu}^{t};
\end{equation*}
\item [(2).] If $\ell=2$,
\begin{equation*}
\left(\mathfrak{R}_1\bm{\sigma}\right)^{t}=\mathbb{F}^{(4)}(1-m_1,m_2)\bm{\mu}^{t}
,~\left(\mathfrak{R}_2\bm{\sigma}\right)^{t}=\mathbb{F}^{(1)}(m_1,1-m_2)\bm{\mu}^{t}
,~\left(\mathfrak{R}_3\bm{\sigma}\right)^{t}=\mathbb{F}^{(6)}(-1-m_2,-1-m_1)\bm{\mu}^{t};
\end{equation*}
\item [(3).] If $\ell=3$,
\begin{equation*}
\left(\mathfrak{R}_1\bm{\sigma}\right)^{t}=\mathbb{F}^{(1)}(1-m_1,m_2)\bm{\mu}^{t}
,~\left(\mathfrak{R}_2\bm{\sigma}\right)^{t}=\mathbb{F}^{(4)}(m_1,1-m_2)\bm{\mu}^{t}
,~\left(\mathfrak{R}_3\bm{\sigma}\right)^{t}=\mathbb{F}^{(7)}(-1-m_2,-1-m_1)\bm{\mu}^{t};
\end{equation*}
\item [(4).] If $\ell=4$,
\begin{equation*}
\left(\mathfrak{R}_1\bm{\sigma}\right)^{t}=\mathbb{F}^{(2)}(1-m_1,m_2)\bm{\mu}^{t}
,~\left(\mathfrak{R}_2\bm{\sigma}\right)^{t}=\mathbb{F}^{(3)}(m_1,1-m_2)\bm{\mu}^{t}
,~\left(\mathfrak{R}_3\bm{\sigma}\right)^{t}=\mathbb{F}^{(5)}(-1-m_2,-1-m_1)\bm{\mu}^{t};
\end{equation*}
\item [(5).] If $\ell=5$,
\begin{equation*}
\left(\mathfrak{R}_1\bm{\sigma}\right)^{t}=\mathbb{F}^{(7)}(1-m_1,m_2)\bm{\mu}^{t}
,~\left(\mathfrak{R}_2\bm{\sigma}\right)^{t}=\mathbb{F}^{(6)}(m_1,1-m_2)\bm{\mu}^{t}
,~\left(\mathfrak{R}_3\bm{\sigma}\right)^{t}=\mathbb{F}^{(4)}(-1-m_2,-1-m_1)\bm{\mu}^{t};
\end{equation*}
\item [(6).] If $\ell=6$,
\begin{equation*}
\left(\mathfrak{R}_1\bm{\sigma}\right)^{t}=\mathbb{F}^{(8)}(1-m_1,m_2)\bm{\mu}^{t}
,~\left(\mathfrak{R}_2\bm{\sigma}\right)^{t}=\mathbb{F}^{(5)}(m_1,1-m_2)\bm{\mu}^{t}
,~\left(\mathfrak{R}_3\bm{\sigma}\right)^{t}=\mathbb{F}^{(2)}(-1-m_2,-1-m_1)\bm{\mu}^{t};
\end{equation*}
\item [(7).] If $\ell=7$,
\begin{equation*}
\left(\mathfrak{R}_1\bm{\sigma}\right)^{t}=\mathbb{F}^{(5)}(1-m_1,m_2)\bm{\mu}^{t}
,~\left(\mathfrak{R}_2\bm{\sigma}\right)^{t}=\mathbb{F}^{(8)}(m_1,1-m_2)\bm{\mu}^{t}
,~\left(\mathfrak{R}_3\bm{\sigma}\right)^{t}=\mathbb{F}^{(3)}(-1-m_2,-1-m_1)\bm{\mu}^{t};
\end{equation*}
\item [(8).] If $\ell=8$,
\begin{equation*}
\left(\mathfrak{R}_1\bm{\sigma}\right)^{t}=\mathbb{F}^{(6)}(1-m_1,m_2)\bm{\mu}^{t}
,~\left(\mathfrak{R}_2\bm{\sigma}\right)^{t}=\mathbb{F}^{(7)}(m_1,1-m_2)\bm{\mu}^{t}
,~\left(\mathfrak{R}_3\bm{\sigma}\right)^{t}=\mathbb{F}^{(1)}(-1-m_2,-1-m_1)\bm{\mu}^{t}.
\end{equation*}
\end{enumerate}
Therefore \eqref{CYL-Ch-3-Eq-27} is proved. This implies that any element $\bm{\sigma}\in\Gamma_{N}(\bm{\mu})$ in the $(M+1)$-level also belongs to $\Gamma_{\mathbb{F}}(\bm{\mu})$. Thus, $\Gamma_{N}(\bm{\mu})\subseteq\Gamma_{\mathbb{F}}(\bm{\mu})$ and the proof of Theorem \ref{CYL-Ch-3-Theorem-6} is complete.
\end{proof}

\begin{rmk}\label{CYL-Ch-3-Remark-7}
For the case of $\gamma_1=\gamma_2=0$, one can easily verify that for any $\bm{\sigma}=(\sigma_1,\sigma_2,\sigma_3)\in \Gamma(\bm{\mu})$,
\begin{equation*}
\left\{
\begin{aligned}
\sigma_{1}=m_1(m_1+3)+m_2(m_2-1),\\
\sigma_{2}=m_1(m_1-1)+m_2(m_2+3),\\
\sigma_{3}=m_1(m_1-1)+m_2(m_2-1),
\end{aligned}
\right.
~\text{where}~(m_1,m_2)\in\mathbb{Z}^2\quad\text{and}\quad
\left\{
\begin{aligned}
\text{either}~&m_1,m_2\equiv0,1~(\mathrm{mod}~4),\\
\text{or}~&m_1,m_2\equiv2,3~(\mathrm{mod}~4).
\end{aligned}
\right.
\end{equation*}
This expression coincides with that of \cite[Theorem 1.1]{Liu-Wang-2021}.
\end{rmk}

\section{Some technical lemmas}\label{CYL-Section-4}
\setcounter{equation}{0}
In this section, we present some important results for preparation. Consider the following system
\begin{equation}\label{CYL-Ch-4-Eq-1}
\Delta u^k_i(x)+\sum\limits_{j=1}^{3}a_{ij}e^{u^k_j(x)}=4\pi\gamma_i\delta_0\quad\text{in}\quad B_1(0),\quad i=1,2,3,
\end{equation}
where the coefficient matrix $A:=(a_{ij})_{3\times3}$ is defined as in \eqref{1.2}. By a little abuse of notations, in the sequel we denote for any sequence $(\mathbf{x},\mathbf{r})=\{(x^k,r_k)\}$,
\begin{equation*}
\sigma^k_i(B(x^k,r_k)):=\frac{1}{2\pi}\int_{B(x^k,r_k)}e^{u^k_i(x)}\mathrm{d}x,\quad i=1,2,3.
\end{equation*}

In Lemma \ref{CYL-Ch-4-Lemma-1}, we assume that
\begin{enumerate}[(i)]
\item The Harnack-type inequality holds
\begin{equation}\label{CYL-Ch-4-Eq-25}
u^k_i(x)+2\log|x|\leq C\quad\text{for}\quad\frac{1}{2}l_k\leq|x|\leq 2s_k,\quad i=1,2,3.
\end{equation}
\item All components of $\mathbf{u}^k$ have fast decay on $\partial B(0,l_k)$ and
\begin{equation*}
\lim\limits_{r\to 0}\lim\limits_{k\to+\infty}\sigma^k_{i}(B(0,rl_{k}))=\lim\limits_{k\to+\infty}\sigma^k_{i}(B(0,l_{k})).
\end{equation*}
\end{enumerate}
For simplicity, we denote $\sigma_i:=\lim\limits_{k\to+\infty}\sigma^k_{i}(B(0,l_{k}))$ in Lemma \ref{CYL-Ch-4-Lemma-1}.

\begin{lemma}\label{CYL-Ch-4-Lemma-1}
Let $\mu_i=1+\gamma_i$. Assume that (i) and (ii) hold.
\begin{enumerate}[(a)]
  \item If $u^k_i$ has slow decay on $\partial B(0,s_k)$ for some $i\in\{1,2,3\}$, then
  \begin{equation*}
    2\mu_i-\sum\limits_{j=1}^{3}a_{ij}\sigma_{j}>0.
  \end{equation*}
  \item At least one component of $\mathbf{u}^k=(u^k_1,u^k_2,u^k_3)$ has fast decay on $\partial B(0,s_k)$.
\end{enumerate}
\end{lemma}
\begin{proof}
(a) Performing the following scaling for the system \eqref{CYL-Ch-4-Eq-1}
\begin{equation*}
v^k_i(y)=u^k_i(s_ky)+2\log{s_k}\quad\text{for}\quad y\in B_{2}(0),\quad i=1,2,3,
\end{equation*}
gives
\begin{equation*}
\begin{aligned}
\Delta v^k_{i}(y)+\sum\limits_{j=1}^{3}a_{ij}e^{v^k_{j}(y)}&=4\pi\gamma_{i}\delta_{0}\quad\text{in}\quad B_2(0),\quad i=1,2,3.
\end{aligned}
\end{equation*}
If there is at least one component of $\mathbf{u}^k$ has slow decay on $\partial B(0,s_k)$, we denote $J$ the set of consecutive indices such that $u^k_i$ $(i\in J)$ has slow decay on $\partial B(0,s_k)$, then $v^k_i$ $(i\in J)$ converges to the solution $v_i$ $(i\in J)$ of a Toda system (or Liouville equation). Precisely, one of the following alternatives holds:
\medskip

\noindent\textbf{Case (1).}  If {$J=\{i\}$ for $i\in\{1,2,3\}$}. Then $v^k_j(y)\to -\infty$ in $L^{\infty}_{\mathrm{loc}}(B_2(0)\setminus \{0\})$ for $j\in\{1,2,3\}\setminus J$, and $v^k_i(y)\to v(y)$ in $C^2_{\mathrm{loc}}(B_2(0)\setminus \{0\})$, where $v(y)$ satisfies
\begin{equation*}
-\Delta v=e^v \quad \text{in}\quad B_2(0)\setminus\{0\}.
\end{equation*}

\noindent\textbf{Case (2).} If {$J=\{1,2\}$}. Then $v^k_3(y)\to -\infty$ in $L^{\infty}_{\mathrm{loc}}(B_2(0)\setminus \{0\})$, and $(v^k_1,v^k_2)\to (v_1,v_2)$ in $C_{\mathrm{loc}}^2(B_{2}(0)\setminus\{0\})$, where $(v_1,v_2)$ satisfies
\begin{align*}
\left\{
\begin{array}{rr}
\Delta v_1+e^{v_1}=0 \\
\Delta v_2+e^{v_2}=0
\end{array}\right.\quad
\text{in}\quad B_{2}(0)\setminus\{0\}.
\end{align*}

\noindent\textbf{Case (3).}  If {$J=\{i,3\}$ for $i=1$ or $2$}. Then $v^k_j(y)\to -\infty$ in $L^{\infty}_{\mathrm{loc}}(B_2(0)\setminus \{0\})$ for $j\in\{1,2,3\}\setminus J$, and $(v^k_i,v^k_3)\to (v_i,v_3)$ in $C_{\mathrm{loc}}^2(B_{2}(0)\setminus\{0\})$, where $(v_i,v_3)$ satisfies
\begin{align*}
\left\{
\begin{array}{rr}
\Delta v_i+e^{v_i}-e^{v_3}=0 \\
\Delta v_3-\frac{1}{2}e^{v_i}+e^{v_3}=0
\end{array}\right.
\quad\text{in}\quad B_{2}(0)\setminus\{0\}.
\end{align*}

\noindent\textbf{Case (4).}  If {$J=\{1,2,3\}$}. Then $(v^k_1,v^k_2,v^k_3)\to (v_1,v_2,v_3)$ in $C_{\mathrm{loc}}^2(B_{2}(0)\setminus\{0\})$, where $(v_1,v_2,v_3)$ satisfies
\begin{equation}\label{CYL-Ch-4-Eq-2}
\left\{
\begin{array}{ll}
-\Delta v_1=e^{v_1}-e^{v_3} \\
-\Delta v_2=e^{v_2}-e^{v_3} \\
-\Delta v_3=-\frac{1}{2}e^{v_1}-\frac{1}{2}e^{v_2}+e^{v_3}
\end{array}\right.\quad
\text{in}\quad B_{2}(0)\setminus\{0\}.
\end{equation}

We shall prove the point (a) by concerning the \textbf{Case (4)}, the proofs for the other cases are similar and simpler. The strength of the Dirac measure at $0$ for \eqref{CYL-Ch-4-Eq-2} can be expressed by
\begin{equation*}
\begin{aligned}
\lim\limits_{r\to 0}\int_{\partial B(0,r)}\frac{\partial v_1(y)}{\partial\nu}\mathrm{d}S
&=\lim\limits_{r\to 0}\lim\limits_{k\to\infty}\left(4\pi\gamma_1-\int_{B(0,r)}e^{v^k_1(y)}\mathrm{d}y
+\int_{B(0,r)}e^{v^k_3(y)}\mathrm{d}y\right)\\
&=4\pi\gamma_1-2\pi(\sigma_1-\sigma_3)=:4\pi\beta_1,
\end{aligned}
\end{equation*}
\begin{equation*}
\begin{aligned}
\lim\limits_{r\to 0}\int_{\partial B(0,r)}\frac{\partial v_2(y)}{\partial\nu}\mathrm{d}S
&=\lim\limits_{r\to 0}\lim\limits_{k\to\infty}\left(4\pi\gamma_2-\int_{B(0,r)}e^{v^k_2(y)}\mathrm{d}y
+\int_{B(0,r)}e^{v^k_3(y)}\mathrm{d}y\right)\\
&=4\pi\gamma_2-2\pi(\sigma_2-\sigma_3)=:4\pi\beta_2,
\end{aligned}
\end{equation*}
and
\begin{equation*}
\begin{aligned}
\lim\limits_{r\to 0}\int_{\partial B(0,r)}\frac{\partial v_3(y)}{\partial\nu}\mathrm{d}S
&=\lim\limits_{r\to 0}\lim\limits_{k\to\infty}\left(4\pi\gamma_3+\frac{1}{2}\int_{B(0,r)}e^{v^k_1(y)}\mathrm{d}y
+\frac{1}{2}\int_{B(0,r)}e^{v^k_2(y)}\mathrm{d}y-\int_{B(0,r)}e^{v^k_3(y)}\mathrm{d}y\right)\\
&=4\pi\gamma_3-2\pi\left(-\frac{1}{2}\sigma_1-\frac{1}{2}\sigma_2+\sigma_3\right)=:4\pi\beta_3.
\end{aligned}
\end{equation*}
This implies that there is an extra term written as $4\pi\beta_{i}\delta_{0}$ on the R.H.S. of \eqref{CYL-Ch-4-Eq-2} for each $i=1,2,3$. It is known that if $\beta_i\leq -1$ for some $i$, then \eqref{CYL-Ch-4-Eq-2} has no solutions. Hence
\begin{equation*}
\beta_1:=\gamma_1-\frac{1}{2}(\sigma_1-\sigma_3)>-1,~\beta_2:=\gamma_2-\frac{1}{2}(\sigma_2-\sigma_3)>-1,
~\beta_3:=\gamma_{3}-\frac{1}{2}\left(-\frac{1}{2}\sigma_1-\frac{1}{2}\sigma_2+\sigma_3\right)>-1,
\end{equation*}
which implies that
\begin{equation}\label{CYL-Ch-4-Eq-27}
2\mu_i-\sum\limits_{j=1}^{3}a_{ij}\sigma_{j}>0,\quad i=1,2,3.
\end{equation}
\medskip

(b) Since $\mathbf{u}^k$ has fast decay on $\partial B(0,l_k)$, $\boldsymbol{\sigma}$ satisfies the Pohozaev identity \eqref{CYL-Ch-2-Eq-Pohozaev-identity}, i.e.,
\begin{equation*}
(\sigma_1-\sigma_3)^2+(\sigma_2-\sigma_3)^2=4(\mu_1\sigma_1+\mu_2\sigma_2+2\mu_3\sigma_3),
\end{equation*}
which is equivalent to
\begin{align*} \sigma_1\left(\sum\limits_{j=1}^{3}a_{1j}\sigma_{j}-2\mu_1\right)
+\sigma_2\left(\sum\limits_{j=1}^{3}a_{2j}\sigma_{j}-2\mu_2\right)
+2\sigma_3\left(\sum\limits_{j=1}^{3}a_{3j}\sigma_{j}-2\mu_3\right)=2\mu_1\sigma_1+2\mu_2\sigma_2+4\mu_3\sigma_3\geq 0.
\end{align*}
As a consequence, there exists $i_0\in\{1,2,3\}$ such that $\sum_{j=1}^{3}a_{i_{0}j}\sigma_{j}-2\mu_{i_0}\geq 0$. However, if all components of $\mathbf{u}^k$ have slow decay on $\partial B(0,s_k)$,  {then by the (a)-th conclusion we obtain} \eqref{CYL-Ch-4-Eq-27}. Contradiction arises. Hence, we conclude that at least one component of $\mathbf{u}^k$ has fast decay on $\partial B(0,s_k)$.
\end{proof}

\begin{lemma}\label{CYL-Ch-4-Lemma-2}
Suppose that the Harnack-type inequality \eqref{CYL-Ch-4-Eq-25} holds for all components of $\mathbf{u}^k$ over $r\in[\frac{l_k}{2},2s_k]$. If all components of $\mathbf{u}^k$ have fast decay on all $r\in[l_{k},s_k]$, then
\begin{equation*}
\sigma^k_i(B(0,s_k))=\sigma^k_i(B(0,l_k))+o(1),\quad i=1,2,3.
\end{equation*}
\end{lemma}

\begin{proof}
The conclusion is trivial if $s_k/l_k\leq C$. In the following we assume that  $s_k/l_k\to+\infty$. We shall prove the lemma by contradiction. Suppose there exists $l\in\{1,2,3\}$ such that
\begin{equation}\label{CYL-Ch-4-Eq-4}
\sigma^k_l(B(0,s_k))>\sigma^k_l(B(0,l_k))+\delta_1
\end{equation}
for some $\delta_1>0$. We define
\begin{equation*}
\hat{\sigma}_i=\lim\limits_{k\to+\infty}\sigma^k_i(B(0,l_k)),\quad i=1,2,3,
\end{equation*}
and
\begin{equation*}
\begin{aligned}
I_{1~(\mathrm{resp.}2,3)}=\left\{i\in\{1,2,3\}~\Big{|}
~2\mu_i-\sum_{j=1}^{3}a_{ij}\hat{\sigma}_j<~(\mathrm{resp.}>,=)~0\right\}.
\end{aligned}
\end{equation*}
Then we claim
\begin{equation}\label{CYL-Ch-4-Eq-5}
\text{\eqref{CYL-Ch-4-Eq-4} is impossible for}~i\in I_{1}\cup I_{2}\cup I_{3}.
\end{equation}
In fact, we have $I_1\cup  I_2$ is not empty. Otherwise, we have
\begin{equation*}
\hat\sigma_1-\hat\sigma_3=2\mu_1,\quad
\hat\sigma_2-\hat\sigma_3=2\mu_2,\quad
-\hat\sigma_1-\hat\sigma_2+2\hat\sigma_3=4\mu_3.
\end{equation*}
Adding the above three equations we get that $\mu_1+\mu_2+2\mu_3=0$ and this is impossible due to that $\mu_i>0$ for each $i$. Thus $I_1\cup  I_2$ is not empty. We set
\begin{equation}\label{CYL-Ch-4-Eq-21}
\delta_{2}=\frac{1}{10000}\min\left\{\min\limits_{i\in   I_1\cup  I_2}\Big|2\mu_i-\sum\limits_{j=1}^{3}a_{ij}\hat{\sigma}_j\Big|,\delta_1,1\right\},
\end{equation}
and choose $\tilde{l}_k\in(l_k,s_k)$ such that
\begin{equation}\label{CYL-Ch-4-Eq-6}
\max\limits_{i=1,2,3}\left\{\sigma^k_i(\tilde{l}_k)-\sigma^k_i(l_k)\right\}=\delta_2.
\end{equation}
Obviously, $\delta_2>0$. Using Harnack-type inequality \eqref{CYL-Ch-4-Eq-25} we have $u^k_i(x)=\overline{u}^k_i(|x|)+O(1)$ for $\frac{1}{2}l_k\leq|x|\leq 2s_k$, see \eqref{CYL-Ch-2-Eq-5}. By direct computations, we deduce from \eqref{CYL-Ch-4-Eq-1} that
\begin{equation}\label{CYL-Ch-4-Eq-22}
\frac{\mathrm{d}}{\mathrm{d}r}\left(\overline{u}^k_i(r)+2\log{r}\right)
=\frac{2\mu_i-\sum_{j=1}^{3}a_{ij}\sigma^k_j(r)}{r},\quad l_k\leq r\leq s_k,\quad i=1,2,3,
\end{equation}
where $\sigma^k_i(r)=\sigma^k_i(B(0,r)),~i=1,2,3$.

Concerning the indices of $I_1$, we could get from \eqref{CYL-Ch-4-Eq-21} that
\begin{equation*}
\frac{\mathrm{d}}{\mathrm{d}r}\left(\overline{u}^k_i(r)+2\log{r}\right)\leq
-\frac{\delta_2}{r},\quad\text{for}~r\in[l_k,\tilde{l}_k],\quad i\in {I}_1.
\end{equation*}
By integrating the above equation from $l_k$ up to $r\in[l_k,\tilde{l}_k]$, we have
\begin{equation*}
\overline{u}^k_i(r)+2\log{r}\leq\overline{u}^k_i(l_k)+2\log{l_k}+\delta_2\log{\frac{l_k}{r}},\quad i\in {I}_1,
\end{equation*}
which implies that for $|x|=r$,
\begin{equation}\label{CYL-Ch-4-Eq-7}
e^{u^k_i(x)}\leq O(1)e^{\overline{u}^k_i(r)}\leq C e^{-N_k}l_k^{\delta_2}r^{-2-\delta_2},\quad i\in {I}_1,
\end{equation}
where we have used $\overline{u}^k_i(l_k)+2\log{l_k}\leq -N_k$ by the assumption of fast-decay. By integrating \eqref{CYL-Ch-4-Eq-7} over $l_k\leq |x|\leq \tilde{l}_k$, we conclude that
\begin{equation*}
\int_{l_k\leq|x|\leq\tilde{l}_k}e^{u^k_i(x)}\mathrm{d}x\leq 2C\pi e^{-N_k}l_k^{\delta_2}\int_{l_k}^{\tilde{l}_k}r^{-1-\delta_2}\mathrm{d}r
\leq2C\pi\frac{e^{-N_k}}{\delta_2}\to 0,\quad i\in {I}_1,
\end{equation*}
as $k\to+\infty$. Hence
\begin{equation}\label{CYL-Ch-4-Eq-8}
\sigma^k_i(\tilde{l}_k)=\sigma^k_i(l_k)+o(1),\quad i\in {I}_1.
\end{equation}

Similarly, for those $i\in {I}_2$, we deduce from \eqref{CYL-Ch-4-Eq-21} that
\begin{equation*}
\frac{\mathrm{d}}{\mathrm{d}r}\left(\overline{u}^k_i(r)+2\log{r}\right)\geq\frac{\delta_2}{r},
\quad\text{for}~r\in[l_k,\tilde{l}_k],\quad i\in {I}_2.
\end{equation*}
By integrating the above equation from $r\in[l_k,\tilde{l}_k]$ up to $\tilde{l}_k$, we get that
\begin{equation*}
\overline{u}^k_i(r)+2\log{r}\leq\overline{u}^k_i(\tilde{l}_k)+2\log{\tilde{l}_k}+\delta_2\log{\frac{r}{\tilde{l}_k}},
\quad\text{for}~r\in[l_k,\tilde{l}_k],\quad i\in {I}_2.
\end{equation*}
This implies that for $|x|=r\in[l_k,\tilde{l}_k]$, there holds
\begin{equation*}
e^{u^k_i(x)}\leq O(1)e^{\overline{u}^k_i(r)}\leq C e^{-N_k}{\tilde{l}_k}^{-\delta_2}r^{-2+\delta_2},\quad i\in  {I}_2,
\end{equation*}
where we have used $\overline{u}^k_i(\tilde{l}_k)+2\log{\tilde{l}_k}\leq -N_k$ by the assumption of fast-decay. Thus
\begin{equation*}
\int_{l_k\leq|x|\leq\tilde{l}_k}e^{u^k_i(x)}\mathrm{d}x\leq 2\pi C e^{-N_k}{\tilde{l}_k}^{-\delta_2}\int_{l_k}^{\tilde{l}_k}r^{-1+\delta_2}\mathrm{d}r
\leq2\pi C\frac{e^{-N_k}}{\delta_2}\to 0,\quad i\in  {I}_2,
\end{equation*}
as $k\to+\infty$, which implies that
\begin{equation}\label{CYL-Ch-4-Eq-24}
\sigma^k_i(\tilde{l}_k)=\sigma^k_i(l_k)+o(1),\quad i\in {I}_2.
\end{equation}
Hence, combining \eqref{CYL-Ch-4-Eq-8} with \eqref{CYL-Ch-4-Eq-24} we have
\begin{equation*}
\sigma^k_i(\tilde{l}_k)=\sigma^k_i(l_k)+o(1),\quad i\in I_{1}\cup I_{2}.
\end{equation*}

It remains to consider the indices $i\in I_3$. We set $\varepsilon_i=\lim\limits_{k\to+\infty}\sigma_i^k(\tilde{l}_k)-\hat\sigma_i$. For $i\in I_1\cup I_2$,  {$\varepsilon_i=0$}. For $i\in I_3$, we have $2\mu_i-\sum_{j=1}^3a_{ij}\hat\sigma_j=0$.
Substituting $\hat\sigma_i+\varepsilon_i,~i=1,2,3$ into the Pohozaev identity \eqref{CYL-Ch-2-Eq-Pohozaev-identity}
we derive that
\begin{equation*}
\varepsilon_1^2+\varepsilon_2^2+2\varepsilon_3^2-2\varepsilon_1\varepsilon_3-2\varepsilon_2\varepsilon_3=0,
\end{equation*}
and it leads to
\begin{equation*}
\varepsilon_1=\varepsilon_2=\varepsilon_3=0.
\end{equation*}
Contradiction arises. Thus \eqref{CYL-Ch-4-Eq-6} does not hold for $i\in I_3$ and it proves \eqref{CYL-Ch-4-Eq-5}.
\end{proof}

To state the last lemma in this section, we introduce the following definition. Let the Harnack-type inequality \eqref{CYL-Ch-4-Eq-25} hold for all components of $\mathbf{u}^k$ over all $r\in[\frac{1}{2}l_k,2\tau_k]$. For a sequence $s_k\leq\tau_k$, we define
\begin{equation}\label{CYL-Ch-4-Eq-14}
\hat{\sigma}_i(B(\mathbf{x},\mathbf{s}))=
\begin{cases}
\lim\limits_{k\to+\infty}\sigma^k_i(B(x^k,s_k)),~&\text{if}~u^k_i~\text{has fast decay on}~\partial B(x^k,s_k),\\
\lim\limits_{r\to 0}\lim\limits_{k\to+\infty}\sigma^k_i(B(x^k,rs_k)),~&\text{if}~u^k_i~\text{has slow decay on}~\partial B(x^k,s_k),
\end{cases}
\end{equation}
where $(\mathbf{x},\mathbf{s})$ stands for the sequence of the pair $\{(x^k,s_k)\}$. When $x^k=0$, we simply denote $\hat{\sigma}_i(B(\mathbf{x},\mathbf{s}))$ by $\hat{\sigma}_i(\mathbf{s})$.

\begin{lemma}\label{CYL-Ch-4-Lemma-3}
Let $\hat{\sigma}_i(\mathbf{s})$ be defined as in \eqref{CYL-Ch-4-Eq-14}. There is a sequence $r_k\in[l_k,\tau_k]$ satisfying the following conditions
\begin{enumerate}[(1)]
  \item $\mathbf{u}^k$ has fast decay on $\partial B(0,r_k)$,
  \item $\exists~i\in\{1,2,3\}$, such that $\hat{\sigma}_i(\mathbf{r})\neq\hat{\sigma}_i(\bm{\tau})$ ($\mathbf{r}$ and $\bm{\tau}$ stand for the sequence $\{r_k\}$ and $\{\tau_k\}$).
\end{enumerate}
Then there exists $s_k\in(r_k,\tau_k)$ such that
\begin{enumerate}[(i)]
  \item $s_k/r_k\to+\infty$, there is at least one component $u^k_i$ of $\mathbf{u}^k$ such that $u^k_i$ has slow decay on $\partial B(0,{s_k})$.
  \item $\hat{\sigma}_i(\mathbf{s})=\hat{\sigma}_i(\mathbf{r}),~i=1,2,3$ ($\mathbf{s}$ stands for the sequence $\{s_k\}$).
\end{enumerate}
\end{lemma}
\begin{proof}
Since $\mathbf{u}^k$ has fast decay on $\partial B(0,{r_k})$, by Proposition \ref{CYL-Ch-2-Proposition-5}, $(\hat{\sigma}_1(\mathbf{r}),\hat{\sigma}_2(\mathbf{r}),\hat{\sigma}_3(\mathbf{r}))$ satisfies the Pohozaev identity \eqref{CYL-Ch-2-Eq-Pohozaev-identity}. Set $\delta=\max\limits_{i=1,2,3}\{\hat{\sigma}_i(\bm{\tau})-\hat{\sigma}_i(\mathbf{r})\}$, then $\delta>0$ by the assumption (2). We decompose $\{1,2,3\}:=I_1\cup I_2\cup I_3$, where
\begin{equation*}
\begin{aligned}
I_{1~(\mathrm{resp.}2,3)}=\left\{i\in\{1,2,3\}~\Big|~ 2\mu_i-\sum_{j=1}^{3}a_{ij}\hat{\sigma}_j(\mathbf{r})<~(\mathrm{resp.}>,=)~0\right\}.
\end{aligned}
\end{equation*}
Set
\begin{equation}\label{CYL-Ch-4-Eq-15}
\delta_0=\frac{1}{10000}\min\left\{\min\limits_{i\in I_1\cup I_2}\Big|2\mu_i-\sum\limits_{j=1}^{3}a_{ij}\hat{\sigma}_j(\mathbf{r})\Big|,\delta,1\right\}.
\end{equation}
We choose $\ell_k\in[r_k,\tau_k]$ such that
\begin{equation}\label{CYL-Ch-4-Eq-16}
\max\limits_{i=1,2,3}\{\sigma^k_i(\ell_k)-\sigma^k_i(r_k)\}=\delta_0.
\end{equation}
Lemma \ref{CYL-Ch-4-Lemma-2} and \eqref{CYL-Ch-4-Eq-16} implies that $\mathbf{u}^k$ can not have fast decay on $[r_k,\ell_k]$. So there is a sequence $r_k\ll s_k\ll \ell_k$ such that some component of $\mathbf{u}^k$ has slow decay on $\partial B(0,s_k)$. So it remains to show
\begin{equation}\label{CYL-Ch-4-Eq-17}
\hat{\sigma}_i(\mathbf{s})=\hat{\sigma}_i(\mathbf{r})\quad\text{for}\quad i=1,2,3.
\end{equation}

We prove \eqref{CYL-Ch-4-Eq-17} by contradiction. Suppose it is not true, then
\begin{equation}\label{CYL-Ch-4-Eq-26}
\varepsilon_0=\max\limits_{i=1,2,3}\{\hat{\sigma}_i(\mathbf{s})-\hat{\sigma}_i({\mathbf{r}})\}>0.
\end{equation}
By Lemma \ref{CYL-Ch-4-Lemma-2} we conclude that there exists a sequence $r_k\ll\hat{s}_k\ll s_k$ such that
\begin{enumerate}[(i)]
\item some components of $\mathbf{u}^k$ have slow decay on $\partial B(0,\hat{s}_k)$,
\item $\max\limits_{i=1,2,3}\{\hat{\sigma}_i(\hat{\mathbf{s}})
    -\hat{\sigma}_i({\mathbf{r}})\}\in[\frac{1}{2}\delta_0,\delta_0]$, where $\hat{\mathbf{s}}$ stands for the sequence $\{\hat{s}_k\}$.
\end{enumerate}
We scale $u^k_i$ by
\begin{equation*}
v^k_i(y)=u^k_i(\hat{s}_ky)+2\log{\hat{s}_k}.
\end{equation*}
Using (i), we have some components of $\mathbf{v}^k$ converge and there is a sequence $R_k\to+\infty$ such that
\begin{enumerate}[(i)]
  \item $R_k\hat{s}_k\ll s_k$, $\mathbf{u}^k$ has fast decay on $\partial B(0,R_k\hat{s}_k)$,
  \item Set $\sigma_i=\lim\limits_{k\to+\infty}\sigma^k_i(B(0,R_k\hat{s}_k))$, then $\bm{\sigma}=(\sigma_1,\sigma_2,\sigma_3)$ satisfies the Pohozaev identity \eqref{CYL-Ch-2-Eq-Pohozaev-identity}.
\end{enumerate}

Let $\varepsilon_i=\sigma_i-\hat{\sigma}_i(\mathbf{r})$ and then $\max\limits_{i=1,2,3}\varepsilon_i\geq\frac{1}{2}\delta_0$. We claim that
\begin{equation}\label{CYL-Ch-4-Eq-9}
\varepsilon_i=0\quad \text{for}\quad i\in I_1\cup I_2.
\end{equation}
First, by the same arguments of Lemma \ref{CYL-Ch-4-Lemma-2} we have $I_1\cup I_2\neq\emptyset$.  Next, we shall prove that $u_i^k$ has fast decay on $\partial B(0,\hat s_k)$ for $i\in I_1\cup I_2$. For $i\in I_1$, using Lemma \ref{CYL-Ch-4-Lemma-1}-(a) we can verify it. While for $i\in I_2$, we shall prove it by contradiction. Suppose it is not true then we could get some $i_0\in I_2\cap J$, where $J$ is the maximal set with consecutive indices such that $v^k_i~(i\in J)$ converges to {a solution of} a Toda system (or Liouville equation)
\begin{equation*}
\Delta v_i+\sum\limits_{j\in J}a_{ij}e^{v_j}=4\pi\left(\gamma_i-\frac{1}{2}\sum\limits_{j=1}^3 a_{ij}\hat{\sigma}_{j}(\hat{\mathbf{s}})\right)\delta_0~~\text{in}~\mathbb{R}^2, \quad i\in J,
\end{equation*}
where $\gamma_i-\frac{1}{2}\sum_{j=1}^3 a_{ij}\hat{\sigma}_{j}(\hat{\mathbf{s}})>-1,~i\in J$. Applying the classification result of \cite{Prajapat-Tarantello-2001}, we conclude that
\begin{equation*}
\begin{aligned}
\sigma^k_{i_0}(R_k\hat{s}_k)&=\hat{\sigma}_{i_0}(\hat{\mathbf{s}})+\frac{1}{2\pi}\int_{\mathbb{R}^2}e^{v_{i_0}}
+o(1)\\
&\geq\hat{\sigma}_{i_0}(\hat{\mathbf{s}})+\mu_{i_0}-\frac{1}{2}\sum\limits_{j=1}^{3}a_{i_0 j}\hat{\sigma}_{j}(\hat{\mathbf{s}})+o(1)\\
&\geq\hat{\sigma}_{i_0}(\hat{\mathbf{s}})+2\delta_0+o(1),
\end{aligned}
\end{equation*}
which contradicts to \eqref{CYL-Ch-4-Eq-15} and \eqref{CYL-Ch-4-Eq-16}. Hence, each $u^k_i~(i\in I_2)$ has fast decay on $\partial B(0,\hat{s}_k)$, as required. After that, we can apply the same arguments \eqref{CYL-Ch-4-Eq-21}--\eqref{CYL-Ch-4-Eq-24} in the proof of Lemma \ref{CYL-Ch-4-Lemma-2} to obtain
\begin{equation*}
\hat{\sigma}_i(\mathbf{R}\hat{\mathbf{s}})=\hat{\sigma}_i(\hat{\mathbf{s}})=\hat{\sigma}_i(\mathbf{r}),~i\in I_1\cup I_2.
\end{equation*}
Thus \eqref{CYL-Ch-4-Eq-9} is proved.

Since both $\hat{\bm{\sigma}}(\mathbf{r})$ and $\bm{\sigma}$ satisfy the Pohozaev identity \eqref{CYL-Ch-2-Eq-Pohozaev-identity}, substituting $\hat\sigma_i(\mathbf{r})+\varepsilon_i$, $i=1,2,3$ into the Pohozaev identity \eqref{CYL-Ch-2-Eq-Pohozaev-identity}
we derive that
\begin{equation*}
\varepsilon_1^2+\varepsilon_2^2+2\varepsilon_3^2-2\varepsilon_1\varepsilon_3-2\varepsilon_2\varepsilon_3=0,
\end{equation*}
which implies that
\begin{equation*}
\varepsilon_1=\varepsilon_2=\varepsilon_3=0,
\end{equation*}
a contradiction to \eqref{CYL-Ch-4-Eq-26}. This finishes the proof of Lemma \ref{CYL-Ch-4-Lemma-3}.
\end{proof}

\section{Local masses on blow-up areas away from the origin}\label{CYL-Section-5}
\setcounter{equation}{0}
\subsection{Local masses on the bubbling disk centered at $x^k_j\neq 0$}
In this subsection, we study the local behavior of $\mathbf{u}^k$ in the ball $B(x^k_t,\tau^k_t)$ where $x^k_t\neq 0$. Precisely, let $x^k_t\in\Sigma_k\setminus\{0\}$ and set
\begin{equation*}
\tau^k_t=\frac{1}{2}\mathrm{dist}(x^k_t,\Sigma_k\setminus\{x^k_t\}).
\end{equation*}
By Proposition \ref{CYL-Ch-2-Proposition-1}, $l^k_t\ll\tau^k_t$. Notice that the Harnack-type inequality \eqref{CYL-Ch-2-Eq-2} holds for $B(x^k_t,\tau^k_t)\setminus\{x^k_t\}$, i.e.,
\begin{equation*}
u^k_i(x)+2\log{|x-x^k_t|}\leq C,~\text{for}~x\in B(x^k_t,\tau^k_t),\quad i=1,2,3.
\end{equation*}
The local mass of the $i$-th component is given by
\begin{equation*}
\sigma^k_{i,t}(r)=\frac{1}{2\pi}\int_{B(x^k_t,r)}e^{u^k_i(x)}\mathrm{d}x,\quad i=1,2,3.
\end{equation*}
Since $x^k_t\neq 0$ and $0\notin B(x^k_t,\tau^k_t)$, the system \eqref{1.1} is reduced to
\begin{equation}\label{CYL-Ch-5-Eq-1}
\left\{
\begin{aligned}
\Delta u^k_i(x)+\sum\limits_{j=1}^{3}a_{ij}e^{u^k_j(x)}&=0 ,\quad i=1,2,3,\\
u^k_{1}+u^k_{2}+2u^k_{3}&=0,
\end{aligned}
\right.
\end{equation}
in $B(x^k_t,\tau^k_t)$, where the coefficient matrix $A=(a_{ij})_{3\times3}$ is defined as in \eqref{1.2}.

For convenience, we fix $t$ and simplify the notations by dropping the index $t$ throughout this subsection. Recall that all the components of $\mathbf{u}^k$ have fast decay on $\partial B(x^k,{l^k})$. In \eqref{CYL-Ch-5-Eq-1}, $\gamma_i=0$, we have $\mu_i=1$. In next proposition we shall assume $\bm{\mu}=(\mu_1,\mu_2,\mu_3)=(1,1,1)$. For a sequence $s_k\leq\tau^k$, we recall that $\hat{\bm{\sigma}}(\mathbf{s})=(\hat{\sigma}_1(\mathbf{s}),\hat{\sigma}_2(\mathbf{s}),\hat{\sigma}_3(\mathbf{s}))$ is defined as in \eqref{CYL-Ch-4-Eq-14}.

\begin{proposition}\label{CYL-Ch-5-Proposition-1}
Let $\mathbf{u}^k=(u^k_1,u^k_2,u^k_3)$ be a sequence of solutions of \eqref{CYL-Ch-5-Eq-1} and $\hat{\bm{\sigma}}(\boldsymbol{\tau})$ be defined as in \eqref{CYL-Ch-4-Eq-14}, then there holds
\begin{enumerate}[(1)]
  \item At least one component of $\mathbf{u}^k$ has fast decay on $\partial B(x^k,\tau^k)$.
  \item $\hat{\bm{\sigma}}(\bm{\tau})
      =(\hat{\sigma}_1(\bm{\tau}),\hat{\sigma}_2(\bm{\tau}),\hat{\sigma}_3(\bm{\tau}))\in\Gamma(1,1,1)$, where $\bm{\tau}$ stands for the sequence $\{\tau^k\}$.
\end{enumerate}
\end{proposition}
\begin{proof}
(1) has been proved in Lemma \ref{CYL-Ch-4-Lemma-1}. To prove (2), we divide the argument into several steps.
\medskip

\noindent\textbf{Step 1.} We prove that $(\hat{\sigma}_1(\boldsymbol{l}),\hat{\sigma}_2(\bm{l}),\hat{\sigma}_3(\bm{l}))\in\Gamma(1,1,1)$, where $\bm{l}$ stands for the sequence $\{l^k\}$. Set
\begin{equation*}
\varepsilon^k=e^{-\frac{1}{2}\max\limits_{i=1,2,3}u^k_i(x^k)},
\end{equation*}
and let
\begin{equation*}
v^k_i(y)=u^k_i(x^k+\varepsilon^k y)+2\log{\varepsilon^k},\quad i=1,2,3.
\end{equation*}
{Suppose that there exists $J\subsetneqq\{1,2,3\}$ such that $u^k_i$ has slow decay on $\partial B(x^k,\tau^k)$ for $i\in J$, while the left ones have fast decay on $\partial B(x^k,\tau^k)$, i.e., $v^k_i(y)\to-\infty$ in $L_{\mathrm{loc}}^{\infty}(\mathbb{R}^2)$ for $i\in\{1,2,3\}\setminus J$ and $v^k_i(y)$ converges to $v_i(y)$ in $C_{\mathrm{loc}}^2(\mathbb{R}^2)$ $\text{for}~i\in J$, where $v_i(y)$ satisfies}
\begin{equation*}
\Delta v_i(y)+\sum\limits_{j\in J}a_{ij}e^{v_j(y)}=0\quad\text{in}\quad\mathbb{R}^2,\quad i\in J.
\end{equation*}
Furthermore, there exists a sequence $R_k$ with $R_k\to +\infty$ as $k\to+\infty$ such that $l^k:=R_k\varepsilon^k\leq\tau^k$ and satisfies
\begin{enumerate}[(1).]
\item $\int_{B(0,{R_k})}e^{v_{i}}\mathrm{d}y=\int_{\mathbb{R}^2}e^{v_{i}}\mathrm{d}y+o(1)$, $i\in J$,
\item $v^k_i(y)+2\log|y|\leq-N_k$ for $|y|=R_{k}$, for some $N_k\to+\infty$, $i=1,2,3$,
\item $\mathbf{u}^k$ has fast decay on $\partial B(x^k,l^k)$.
\end{enumerate}
Then one of the following alternatives holds:
\begin{enumerate}[(a).]
\item If $J=\emptyset$, then $(\hat{\sigma}_1(\bm{l}),\hat{\sigma}_2(\bm{l}),\hat{\sigma}_3(\bm{l}))=(0,0,0)$.

\item If $J=\{i\}$ for some $i\in\{1,2,3\}$, then $(\hat{\sigma}_1(\bm{l}),\hat{\sigma}_2(\bm{l}),\hat{\sigma}_3(\bm{l}))=(4,0,0)$ or $(0,4,0)$ or $(0,0,4)$.

\item If $J=\{1,2\}$, then $(\hat{\sigma}_1(\bm{l}),\hat{\sigma}_2(\bm{l}),\hat{\sigma}_3(\bm{l}))=(4,4,0)$.

\item If $J=\{1,3\}$ or $\{2,3\}$, then by the classification result Theorem \ref{CYL-Ch7-Theorem-Classification}-(a), we get $(\hat{\sigma}_1(\bm{l}),\hat{\sigma}_2(\bm{l}),\hat{\sigma}_3(\bm{l}))=(16,0,12)$ or $(0,16,12)$.
\end{enumerate}

Therefore, in any case, we always have
\begin{equation*}
(\hat{\sigma}_1(\bm{l}),\hat{\sigma}_2(\bm{l}),\hat{\sigma}_3(\bm{l}))\in\Gamma(1,1,1).
\end{equation*}
\medskip

\noindent\textbf{Step 2.} If $\hat{\sigma}_i(\bm{l})=\hat{\sigma}_i(\bm{\tau})$ for $i=1,2,3$, then Proposition \ref{CYL-Ch-5-Proposition-1} is proved. Otherwise, there exists $j\in\{1,2,3\}$ such that $\hat{\sigma}_j(\bm{l})\neq\hat{\sigma}_j(\bm{\tau})$, then we can apply Lemma \ref{CYL-Ch-4-Lemma-3} to find a sequence $s_k$ such that $l^k\ll s_k\ll\tau^k$, some components of $\mathbf{u}^k$ have slow decay on $\partial B(x^k,{s_k})$ and
\begin{equation*}
\hat{\sigma}_i(\bm{l})=\hat{\sigma}_i(\mathbf{s})\quad\text{for}\quad i=1,2,3,
\end{equation*}
where $\mathbf{s}$ stands for the sequence $\{s_k\}$. Let
\begin{equation*}
v^k_i(y)=u^k_i(x^k+s_ky)+2\log{s_k}.
\end{equation*}
Suppose that there exists $J\subsetneqq\{1,2,3\}$ such that $u^k_i$ has slow decay on $\partial B(x^k,s_k)$ for $i\in J$, while the left ones have fast decay on $\partial B(x^k,s_k)$, i.e., $v^k_i(y)\to-\infty$ in $L_{\mathrm{loc}}^{\infty}(\mathbb{R}^2)$ for $i\in\{1,2,3\}\setminus J$ and $v^k_i(y)$ converges to $v_i(y)$ in $C_{\mathrm{loc}}^2(\mathbb{R}^2)$ $\text{for}~i\in J$, where $v_i(y)$ satisfies
\begin{equation*}
\Delta v_i(y)+\sum\limits_{j\in J}a_{ij}e^{v_j(y)}=4\pi\gamma_i^{*}\delta_0\quad\text{in}\quad\mathbb{R}^2,\quad i\in J,
\end{equation*}
and $\gamma_i^{*}=-\frac{1}{2}\sum_{j=1}^{3}a_{ij}\hat{\sigma}_{j}(\bm{l})>-1$. Furthermore, there exists a sequence $N_k^{*}$ with $N_{k}^{*}\to +\infty$ as $k\to+\infty$ such that $N_k^{*}s_k\leq\tau^k$ and satisfies
\begin{enumerate}[(1).]
\item $\int_{B(0,{N_k^{*}})}e^{v_{i}}\mathrm{d}y=\int_{\mathbb{R}^2}e^{v_{i}}\mathrm{d}y+o(1)$, $i\in J$,
\item $v^k_i(y)+2\log|y|\leq-N_k$ for $|y|=N^{*}_{k}$, for some $N_k\to+\infty$, $i=1,2,3$,
\item $\mathbf{u}^k$ has fast decay on $\partial B(x^k,N^{*}_{k} s_k)$.
\end{enumerate}
Then one of the following alternatives holds:
\medskip

\noindent\textbf{Case~(a).} If $J=\{i_0\}$ for some $i_0\in\{1,2,3\}$. Notice that $\hat{\sigma}_{i}(\mathbf{N^{*}s})=\hat{\sigma}_i(\bm{l})$ for $i\in\{1,2,3\}\setminus J$, where $\mathbf{N^{*}s}$ stands for the sequence $\{N_k^* s_k\}$, and both $\hat{\bm{{\sigma}}}(\mathbf{N^{*}s})$ and $\hat{\bm{\sigma}}(\bm{l})$ satisfy the Pohozaev identity \eqref{CYL-Ch-2-Eq-Pohozaev-identity}. Then $\hat{\sigma}_{{i_0}}(\mathbf{N^{*}s})$ and $\hat{\sigma}_{i_0}(\bm{l})$ are two roots of a quadratic polynomial in $\sigma_{i_0}$. From which we can directly solve $\hat{\bm{\sigma}}(\mathbf{N^{*}s})=\mathfrak{R}_{{i_0}}(\hat{\bm{{\sigma}}}(\bm{l}))$. We can also deduce this result in another way by applying classification result of singular Liouville equation \cite{Prajapat-Tarantello-2001}:
\begin{equation*}
\begin{aligned}
\hat{\sigma}_{{i_0}}(\mathbf{N^{*}s})&=\frac{1}{2\pi}\int_{\mathbb{R}^2}e^{v_{{i_0}}}\mathrm{d}x
+\hat{\sigma}_{i_0}(\bm{l})
=4\left(1-\frac{1}{2}\sum\limits_{j=1}^3 a_{{i_0}j}\hat{\sigma}_{j}(\bm{l})\right)+\hat{\sigma}_{i_0}(\bm{l})
=\left(\mathfrak{R}_{{i_0}}(\hat{\bm{\sigma}}(\bm{l}))\right)_{{i_0}}.
\end{aligned}
\end{equation*}
Then $\hat{\bm{\sigma}}(\mathbf{N^{*}s})=\mathfrak{R}_{{i_0}}(\hat{\bm{{\sigma}}}(\bm{l}))\in\Gamma(1,1,1)$ if $\hat{\bm{\sigma}}(\bm{l})\in\Gamma(1,1,1)$.
\smallskip

\noindent\textbf{Case~(b).} If $J=\{1,2\}$. From the view of the classification result of \cite{Prajapat-Tarantello-2001}, we directly deduce that
\begin{equation*}
\begin{aligned}
\hat{\sigma}_{m}(\mathbf{N^{*}s})&=\frac{1}{2\pi}\int_{\mathbb{R}^2}e^{v_{m}}\mathrm{d}x+\hat{\sigma}_m(\bm{l})
=4\left(1-\frac{1}{2}\sum\limits_{j=1}^3 a_{mj}\hat{\sigma}_{j}(\bm{l})\right)+\hat{\sigma}_m(\bm{l})
=\left(\mathfrak{R}_{12}(\hat{\bm{\sigma}}(\bm{l}))\right)_{m},~\text{for}~m=1,2.
\end{aligned}
\end{equation*}
We also notice that $\hat{\sigma}_{3}(\mathbf{N^{*}s})=\hat{\sigma}_3(\bm{l})$, hence $\hat{\bm{\sigma}}(\mathbf{N^{*}s})=\mathfrak{R}_{12}(\hat{\bm{\sigma}}(\bm{l}))\in\Gamma(1,1,1)$ provided $\hat{\bm{\sigma}}(\bm{l})\in\Gamma(1,1,1)$.
\smallskip

\noindent\textbf{Case~(c).} If $J=\{m,3\}$ for $m=1$ or $2$. From the classification result Theorem \ref{CYL-Ch7-Theorem-Classification}-(a), we directly deduce that
\begin{equation*}
\begin{aligned}
\hat{\sigma}_{m}(\mathbf{N^{*}s})&=\frac{1}{2\pi}\int_{\mathbb{R}^2}e^{v_{m}}\mathrm{d}x+\hat{\sigma}_m(\bm{l})
=8\left(1-\frac{1}{2}\sum\limits_{j=1}^3 a_{mj}\hat{\sigma}_{j}(\bm{l})\right)+8\left(1-\frac{1}{2}\sum\limits_{j=1}^3 a_{3j}\hat{\sigma}_{j}(\bm{l})\right)+\hat{\sigma}_m(\bm{l})\\
&=\left(\mathfrak{R}_{m3}\left(\hat{\bm{\sigma}}(\bm{l})\right)\right)_{m},~\text{for}~m=1~\text{or}~2,
\end{aligned}
\end{equation*}
and
\begin{equation*}
\begin{aligned}
\hat{\sigma}_{3}(\mathbf{N^{*}s})&=\frac{1}{2\pi}\int_{\mathbb{R}^2}e^{v_{3}}\mathrm{d}x+\hat{\sigma}_3(\bm{l})
=4\left(1-\frac{1}{2}\sum\limits_{j=1}^3 a_{mj}\hat{\sigma}_{j}(\bm{l})\right)+8\left(1-\frac{1}{2}\sum\limits_{j=1}^3 a_{3j}\hat{\sigma}_{j}(\bm{l})\right)+\hat{\sigma}_3(\bm{l})\\
&=\left(\mathfrak{R}_{m3}\left(\hat{\bm{\sigma}}(\bm{l})\right)\right)_{3},~\text{for}~m=1~\text{or}~2.
\end{aligned}
\end{equation*}
We also notice that $\hat{\sigma}_{j}(\mathbf{N^{*}s})=\hat{\sigma}_j(\bm{l})$ for $j\in\{1,2,3\}\setminus J$, then $\hat{\bm{\sigma}}(\mathbf{N^{*}s})=\mathfrak{R}_{m3}(\hat{\bm{\sigma}}(\bm{l}))\in\Gamma(1,1,1)$ since $\hat{\bm{\sigma}}(\bm{l})\in\Gamma(1,1,1)$.
\medskip

\noindent\textbf{Step 3.} Let $s_{k,1}=N_{k}^{*}s_k$. If
\begin{equation*}
\hat{\sigma}_i(\mathbf{s}_1)=\hat{\sigma}_i(\bm{\tau})\quad\text{for}\quad i=1,2,3,
\end{equation*}
then Proposition \ref{CYL-Ch-5-Proposition-1} is proved. Otherwise, we could repeat the argument of \textbf{Step 2} to find $s_{k,1}\ll s_{k,j}\ll s_{k,j+1}$ such that
$\hat{\bm{\sigma}}(\mathbf{s}_{j+1})\in\Gamma(1,1,1)$, where $\mathbf{s}_{j+1}$ stands for the sequence $\{s_{k,j+1}\}$. Since the energy is finite by (iii) of \eqref{1.10} and the total gain for the local masses at each step has a lower bound
\begin{equation*}
\sum\limits_{i=1}^{3}\left(\hat{\sigma}_i(\mathbf{s}_{j+1})-\hat{\sigma}_i(\mathbf{s}_j)\right)\geq4>0,
\end{equation*}
then after finitely many steps we have $\hat{\sigma}_i(\mathbf{s}_j)=\hat{\sigma}_i(\bm{\tau})~\text{for}~i=1,2,3$.
\end{proof}

From the proof of Proposition \ref{CYL-Ch-5-Proposition-1}, we can further deduce a more general result which will be used to prove Theorem \ref{CYL-Theorem-1} in Section \ref{CYL-Section-6}. Suppose that there is a subset $J\subsetneqq\{1,2,3\}$ such that each $u^k_i~(i\in J)$ has slow decay on $\partial B(x^k,s_k)$. Let
\begin{equation*}
v^k_i(y)=u^k_i(x^k+s_ky)+2\log{s_k},\quad i=1,2,3.
\end{equation*}
Then $v^k_i(y)\to-\infty$ for $i\in\{1,2,3\}\setminus J$ in $L^{\infty}_{\mathrm{loc}}(\mathbb{R}^2)$ and $v^k_i(y)$ converges to $v_i(y)$ in $C_{\mathrm{loc}}^2(\mathbb{R}^2)$ for $i\in J$, where $v_i(y)$ satisfies
\begin{equation}\label{CYL-Ch-5-Eq-2}
\Delta v_i(y)+\sum\limits_{j\in J}a_{ij}e^{v_j(y)}=4\pi\gamma_i^{*}\delta_0+4\pi\sum\limits_{l=1}^{N}m_{il}\delta_{q_l}\quad\text{in}\quad\mathbb{R}^2,\quad i\in J,
\end{equation}
where $0\neq q_{l}\in\mathbb{R}^2$, $m_{il}\in\mathbb{N}$ and $\gamma_i^{*}=\gamma_i-\frac{1}{2}\sum_{j=1}^{3}a_{ij}\sigma_j>-1$. Set
\begin{equation*}
\sigma^{*}_{j}=
\left\{
\begin{array}{ll}
\sigma_j,~&\text{if}~j\in \{1,2,3\}\setminus J,\\
\sigma_j+\frac{1}{2\pi}\int_{\mathbb{R}^2} e^{v_j}dx,~&\text{if}~j\in J.
\end{array}\right.
\end{equation*}
Then we have the following result

\begin{proposition}\label{CYL-Ch-5-Proposition-2}
If there is $\hat{\bm{\sigma}}\in\Gamma(\bm{\mu})$ such that $\sigma_{i}=\hat{\sigma}_i+4n_i,~n_i\in\mathbb{Z}$. Then $\sigma^*_i=\hat{\sigma}^*_i+4n^*_i$ with $\hat{\bm{\sigma}}^{*}\in\Gamma(\bm{\mu})$ and $n^*_i\in\mathbb{Z}$.
\end{proposition}
\begin{proof} We shall consider the following three cases.
\medskip

\noindent \textbf{Case (a).}  If {$J=\{l\}$} for some $l\in\{1,2,3\}$. Then by the classification result of \cite{Prajapat-Tarantello-2001}, we have
\begin{equation}\label{CYL-Ch-5-Eq-3}
\begin{aligned}
\frac{1}{2\pi}\int_{\mathbb{R}^2}e^{v_{l}}\mathrm{d}x
&=4(1+\gamma^*_l)+4N_l=4\mu^{*}_{l}+4N_l,
\end{aligned}
\end{equation}
where
\begin{equation*}
\mu^{*}_{l}=1+\gamma^*_l=1+\gamma_l-\frac{1}{2}\sum\limits_{j=1}^3 a_{lj}{\sigma}_{j}=\mu_l-\frac{1}{2}\sum\limits_{j=1}^3 a_{lj}{\sigma}_{j},
\end{equation*}
and $N_l\in\mathbb{Z}$. Since $\sigma_{i}=\hat{\sigma}_i+4n_i$ for some $n_i\in\mathbb{Z}$, we get
\begin{equation*}
\mu^{*}_{i}=\mu_i-\frac{1}{2}\sum\limits_{j=1}^3 a_{ij}{\sigma}_{j}=\mu_i-\frac{1}{2}\sum\limits_{j=1}^3 a_{ij}\hat{\sigma}_{j}+\overline{n}_i,\quad\overline{n}_i\in\mathbb{Z},\quad i=1,2,3.
\end{equation*}
Then \eqref{CYL-Ch-5-Eq-3} can be rewritten as
\begin{equation*}
\begin{aligned}
\frac{1}{2\pi}\int_{\mathbb{R}^2}e^{v_{l}}\mathrm{d}x=4\mu^{*}_{l}+4N_l=4\mu_l-\sum_{j=1}^3a_{lj}\hat\sigma_j
+4\widetilde N_l,\quad\widetilde N_l\in\mathbb{Z}.
\end{aligned}
\end{equation*}
Thus we have
\begin{equation*}
\begin{aligned}
\sigma^{*}_l=\sigma_l+\frac{1}{2\pi}\int_{\mathbb{R}^2}e^{v_{l}}\mathrm{d}x
=\hat{\sigma}_l+4n_l+4\mu_l-2\sum_{j=1}^3a_{lj}\hat\sigma_j+4\widetilde{N}_l
=:\hat{\sigma}_l^*+4n^{*}_l,\quad n_l^*\in\mathbb{Z}.
\end{aligned}
\end{equation*}
While for $j\neq l$, we set $\hat{\sigma}_j=\hat{\sigma}_j^*$, then
\begin{equation*}
\begin{aligned}
\sigma^{*}_j=\sigma_{j}=\hat{\sigma}_j+4n_{j}=:\hat{\sigma}_j^*+4n_{j}.
\end{aligned}
\end{equation*}
Therefore, we get
\begin{equation*}
\bm{\sigma}^{*}=\hat{\bm{\sigma}}^{*}+4\mathbb{Z}\quad\text{with}\quad \hat{\bm{\sigma}}^{*}=\mathfrak{R}_{l}(\hat{\bm{\sigma}})\in\Gamma(\bm{\mu}).
\end{equation*}
\medskip

\noindent \textbf{Case (b).} If {$J=\{1,2\}$}. Then by the classification result of \cite{Prajapat-Tarantello-2001}, we have
\begin{equation}\label{CYL-Ch-5-Eq-4}
\begin{aligned}
\frac{1}{2\pi}\int_{\mathbb{R}^2}e^{v_{i}}\mathrm{d}x
&=4(1+\gamma^*_i)+4N_i=4\mu^{*}_{i}+4N_i,\quad N_i\in\mathbb{Z},\quad i=1,2,
\end{aligned}
\end{equation}
where
\begin{equation*}
\mu^{*}_{i}=1+\gamma^*_i=1+\gamma_i-\frac{1}{2}\sum\limits_{j=1}^3 a_{ij}{\sigma}_{j}=\mu_i-\frac{1}{2}\sum\limits_{j=1}^3 a_{ij}{\sigma}_{j},\quad i=1,2.
\end{equation*}
Since $\sigma_{i}=\hat{\sigma}_i+4n_i$ for some $n_i\in\mathbb{Z}$, $i=1,2,3$, we get that
\begin{equation*}
\mu^{*}_{i}=\mu_i-\frac{1}{2}\sum\limits_{j=1}^3 a_{ij}{\sigma}_{j}=\mu_i-\frac{1}{2}\sum\limits_{j=1}^3 a_{ij}\hat{\sigma}_{j}+\overline{n}_i,\quad\overline{n}_i\in\mathbb{Z},\quad i=1,2,3.
\end{equation*}
Then \eqref{CYL-Ch-5-Eq-4} can be rewritten as
\begin{equation*}
\begin{aligned}
\frac{1}{2\pi}\int_{\mathbb{R}^2}e^{v_{i}}\mathrm{d}x=4\mu^{*}_{i}+4N_i
=4\left(\mu_i-\frac12\sum_{j=1}^3a_{ij}\hat\sigma_j\right)+4\hat N_i,\quad\hat N_i\in\mathbb{Z},\quad
i=1,2.
\end{aligned}
\end{equation*}
Hence, for $i=1,2$ we get that
\begin{equation*}
\begin{aligned}
\sigma^{*}_i=\sigma_i+\frac{1}{2\pi}\int_{\mathbb{R}^2}e^{v_{i}}\mathrm{d}x
=\hat{\sigma}_i+4\mu_i-2\sum_{j=1}^3a_{ij}\hat\sigma_j+
+4n_i+4\widetilde{N}_i=:\hat{\sigma}_i^*+4n^{*}_i,
\end{aligned}
\end{equation*}
with $\hat{\sigma}_i^*=4\mu_i-2\sum_{j=1}^3a_{ij}\hat\sigma_j+\hat{\sigma}_i
=\left(\mathfrak{R}_{12}(\hat{\bm{\sigma}})\right)_i$. While for the third component, we set $\hat{\sigma}_3^*=\hat{\sigma}_3$. Then we conclude that
\begin{equation*}
\bm{\sigma}^{*}=\hat{\bm{\sigma}}^{*}+4\mathbb{Z}\quad\text{with}\quad\hat{\bm{\sigma}}^{*}
=\mathfrak{R}_{12}(\hat{\bm{\sigma}})\in\Gamma(\bm{\mu}).
\end{equation*}
\medskip

\noindent \textbf{Case (c).} If {$J=\{i,3\}$} for $i=1$ or $2$. Then by the classification result Theorem \ref{CYL-Ch7-Theorem-Classification}-(c), we have
\begin{equation*}
\begin{aligned}
\frac{1}{2\pi}\int_{\mathbb{R}^2}e^{v_{i}}\mathrm{d}x=\bar\sigma_i+4N_i,\quad i=1~\text{or}~2,
\end{aligned}
\end{equation*}
and
\begin{equation*}
\begin{aligned}
\frac{1}{2\pi}\int_{\mathbb{R}^2}e^{v_{3}}\mathrm{d}x
&=\bar\sigma_3+4N_3,
\end{aligned}
\end{equation*}
where
\begin{equation}
\label{CYL-Ch-5-Eq-10}
\begin{aligned}
(\bar\sigma_i,\bar\sigma_3)\in\Big\{
&(0,0),~(4\mu_i^*,0),~(0,4\mu_3^*),~(4\mu_i^*,4(\mu_i^*+\mu_3^*)),~(4\mu_i^*+8\mu_3^*,4\mu_3^*),\\
&(4\mu_i^*+8\mu_3^*,4\mu_i^*+8\mu_3^*),~(8\mu_i^*+8\mu_3^*,4\mu_i^*+4\mu_3^*),
~(8\mu_i^*+8\mu_3^*,4\mu_i^*+8\mu_3^*)\Big\},
\end{aligned}
\end{equation}
and
\begin{equation*}
\mu_i^*=\gamma_i^*+1,~i=1~(\mbox{or}~2),3,\quad N_i\in\mathbb{Z}.
\end{equation*}
Then we get that
\begin{equation}
\label{CYL-Ch-5-Eq-9}
\sigma_i^*=\hat\sigma_i+\bar\sigma_i+4\bar N_i,~i=1~\mbox{or}~2,
\quad\mbox{and}\quad \sigma_3^*=\hat\sigma_3+\bar\sigma_3+4\bar N_3,
\end{equation}
where $\bar N_i\in\mathbb{Z},~i=1~(\text{or}~2), 3$. We set
\begin{equation*}
\hat\sigma_j^*=\hat\sigma_j+\bar\sigma_j,~j=i,3,\quad\mathrm{and}\quad \hat\sigma_j^*=\hat\sigma_j,~j\neq i,3.
\end{equation*}
By \eqref{CYL-Ch-5-Eq-10} and \eqref{CYL-Ch-5-Eq-9}, from direct computation we get that
\begin{equation*}
\hat{\bm{\sigma}}^{*}\in\Big\{
\hat{\bm{\sigma}},~\mathfrak{R}_i\hat{\bm{\sigma}},~\mathfrak{R}_3\hat{\bm{\sigma}},
~\mathfrak{R}_{i}\mathfrak{R}_{3}\hat{\bm{\sigma}},~\mathfrak{R}_{3}\mathfrak{R}_{i}\hat{\bm{\sigma}},
~\mathfrak{R}_{i}\mathfrak{R}_{3}\mathfrak{R}_{i}\hat{\bm{\sigma}},
~\mathfrak{R}_{3}\mathfrak{R}_{i}\mathfrak{R}_{3}\hat{\bm{\sigma}},~\mathfrak{R}_{i3}\hat{\bm{\sigma}}
\Big\}.
\end{equation*}
Since $\hat{\bm{\sigma}}\in\Gamma(\mu)$ we get that $\hat{\bm{\sigma}}^{*}\in\Gamma(\mu)$ as well by the definition of $\Gamma(\mu)$. Thus, we have
\begin{equation*}
{\bm{\sigma}}^*={\hat{\bm{\sigma}}}^*+4\mathbb{Z}\quad\mbox{with}\quad {\hat{\bm{\sigma}}}^*\in\Gamma(\mu).
\end{equation*}
This finishes the whole proof.
\end{proof}

\subsection{Local mass in a group that does not contain $0$}
In this subsection, we shall collect together some of the bubbling disks $B(x_i^k,l_i^k)$ into a larger one. We present the procedure as follows. For any $x_i^k\neq0\in\Sigma_k$, we define the subset $S^{k}_1$ of $\Sigma_k$ by the following way. Let $x_2^k\neq0\in\Sigma_k$ be the point of $\Sigma_k$ such that
\begin{equation*}
\mathrm{dist}(x_1^k,x_2^k)=\mathrm{dist}\left(x_1^k,\Sigma_k\setminus\{0,x_1^k\}\right),
\end{equation*}
where $\mathrm{dist}(\cdot,\cdot)$ is the distance function. We set
\begin{equation*}
S^{k}_1=\begin{cases}
\{x_1^k\},\quad &\mbox{if}\quad\frac{\mathrm{dist}(x_1^k,x_2^k)}{|x_1^k|}\quad\mbox{has a positive lower bound},\\
\left\{x_j^k~\mid~\mathrm{dist}(x_1^k,x_j^k)\leq C\mathrm{dist}(x_1^k,x_2^k)\right\},\quad &\mbox{if}\quad\frac{\mathrm{dist}(x_1^k,x_2^k)}{|x_1^k|}~\to0.
\end{cases}
\end{equation*}
Then we have
\begin{equation*}
\forall x,y\in S_1^k,\quad \mbox{the ratio}~\frac{\mathrm{dist}(x,y)}{|y|}\to0~\mbox{as}~k\to+\infty.
\end{equation*}
Obviously, one can see that $\Sigma_k$ can be decomposed into a disjoint union of $S_j^k:$
\begin{equation*}
\Sigma_k=\{0\}\cup S_1^k\cup\cdots \cup S_{m_0}^k,\quad m_0\leq m\quad\mbox{and}\quad m_0\in\mathbb{N}.
\end{equation*}
Next, we shall use $S_1^k$ as an example to illustrate the combination process for the points in $S_1^k$. We may assume
\begin{equation*}
S_1^k=\{x_{1,1}^k,\cdots,x_{1,m_1}^k\}\subseteq\Sigma_k.
\end{equation*}
Let {$\tau_{S^k_1}^k$} and $\tau_{l}^k$ be defined as follows:
\begin{equation*}
\tau_{S^k_1}^k=\frac12\mbox{dist}(x_{1,1}^k,\Sigma_k\setminus S_1^k)\quad \mbox{and}\quad\tau_{l}^k=\frac12\mbox{dist}(x_{1,l}^k,S_1^k\setminus\{x_{1,l}^k\})\quad\text{for}\quad l=1,\cdots,m_{1}.
\end{equation*}
Then it is not difficult to see that
\begin{equation*}
\bigcup_{l=1}^{m_1}B(x_{1,l}^k,\tau_{l}^k)\subseteq B(x_1^k,\tau_{S^k_1}^k)\quad \mbox{and}\quad \tau_{l}^k/\tau_{S^k_1}^k\to 0\quad\mbox{for}\quad l=1,\cdots,m_1.
\end{equation*}
By Proposition \ref{CYL-Ch-5-Proposition-1}, the local mass $\hat{\sigma}_{i}(B(\mathbf{x}_{1,l},\bm{\tau}_{l}))=4m_{l,i},~i=1,2,3,~l=1,\cdots,m_1$ satisfies
\begin{equation*}
(4m_{l,1},4m_{l,2},4m_{l,3})\in\Gamma(1,1,1)\quad \mbox{for}\quad l=1,\cdots,m_1,
\end{equation*}
where $(\mathbf{x}_{1,l},\bm{\tau}_{l})$ stands for the sequence of pairs $\{(x^k_{1,l},\tau^k_{l})\}$, $m_{l,i}\in\mathbb{N}\cup\{0\},~1\leq l\leq m_1$. In the following proposition we shall compute the possible values on the blow up mass in $B(x_{1,1}^k,\tau_{S^k_1}^k)$

\begin{proposition}\label{CYL-Ch-5-Proposition-3}
There holds:
\begin{enumerate}[(i)]
  \item At least one component of $\mathbf{u}^k$ has fast decay on $\partial B(x^k_{1,1},\tau^k_{1})$.
  \item Let $\sigma_i=\hat{\sigma}_i(B(\mathbf{x}_{1,1},\bm{\tau}_{S^k_1}))$, then $\sigma_i\in4\mathbb{N}\cup\{0\}$.
\end{enumerate}
\end{proposition}
\begin{proof}
By Proposition \ref{CYL-Ch-5-Proposition-1}, it is sufficient to consider the following two cases.
\medskip

\noindent \textbf{Case~(1).} $\mathbf{u}^k$ has fast decay on $\partial B(x^k_{1,1},\tau^{k}_{1})$. Let
\begin{equation*}
l^k(S^k_1)=2\max_{1\leq l\leq m_1}\mbox{dist}(x_{1,1}^k,x_{1,l}^k).
\end{equation*}
Then $\mathbf{u}^k$ has fast decay on $\partial B(x^k_{1,l},\tau_{l}^k)$ for any $l=1,\cdots,m_1.$ Therefore
      \begin{equation}\label{CYL-Ch-5-Eq-7}
      \begin{aligned}
      \sigma^k_i(B(x^k_{1,1},l^k(S^k_1)))
      &=\frac{1}{2\pi}\int_{B(x^k_{1,1},l^k(S^k_1))}e^{u^k_i(x)}\mathrm{d}x\\
      &=\frac{1}{2\pi}\int_{\bigcup_{l=1}^{m_1} B(x^k_{1,l},\tau^k_{l})}e^{u^k_i(x)}\mathrm{d}x
      +\frac{1}{2\pi}\int_{B(x^k_{1,1},l^k(S^k_1))\setminus \bigcup_{j=1}^{m_1} B(x^k_{1,j},\tau^k_{j})}e^{u^k_i(x)}\mathrm{d}x\\
      &=\frac{1}{2\pi}\int_{\bigcup_{l=1}^{m_1} B(x^k_{1,l},\tau^k_{l})}e^{u^k_i(x)}\mathrm{d}x+o(1)\\
      &=\sum\limits_{l=1}^{m_1}4m_{l,i}+o(1),
      \end{aligned}
      \end{equation}
      where $m_{l,i}\in\mathbb{N}\cup\{0\},~i=1,2,3,~1\leq l\leq m_1$. Here we have used the fast-decay of $\mathbf{u}^k$ outside of {$B(x^k_{1,l},\tau^k_{l})$}, and then the second integral of \eqref{CYL-Ch-5-Eq-7} is $o(1)$.

      If it holds that
      \begin{equation*}
        \hat{\sigma}_{i}(B(\mathbf{x}_{1,1},\mathbf{l}(S^k_1)))
        =\hat{\sigma}_{i}(B(\mathbf{x}_{1,1},\mathbf{\bm{\tau}}_{S^k_1})),\quad\forall i=1,2,3,
      \end{equation*}
      where $(\mathbf{x}_{1,1},\mathbf{l}(S^k_1))$ and $(\mathbf{x}_{1,1},\mathbf{\bm{\tau}}_{S^k_1})$ stand for the sequences of pairs $\{(x^k_{1,1},l^k(S^k_1))\}$ and $\{(x^k_{1,1},\tau^k_{S^k_1})\}$ respectively, then Proposition \ref{CYL-Ch-5-Proposition-3} is proved. Otherwise,
      \begin{equation*}
        \hat{\sigma}_{i}(B(\mathbf{x}_{1,1},\mathbf{\bm{\tau}}_{S^k_1}))>
        \hat{\sigma}_{i}(B(\mathbf{x}_{1,1},\mathbf{l}(S^k_1)))\quad\mbox{for some}~i\in\{1,2,3\}.
      \end{equation*}
      In this situation we can apply the same arguments of \textbf{Step 2} and \textbf{Step 3} in the proof of Proposition \ref{CYL-Ch-5-Proposition-1} to obtain that
      \begin{equation*}
        \hat{\sigma}_{i}(B(\mathbf{x}_{1,1},\mathbf{\bm{\tau}}_{S^k_1}))
        =\hat{\sigma}_{i}(B(\mathbf{x}_{1,1},\mathbf{s}_{\ell}))\in4\mathbb{N}\cup\{0\},~i=1,2,3,
      \end{equation*}
      for some $\mathbf{s}_{\ell}$, where $\mathbf{s}_{\ell}$ stands for the sequence $\{s^k_\ell\}$ with $l^k(S^k_1)\ll s^k_\ell\ll \tau^k_{S_1}$.
\medskip

\noindent \textbf{Case~(2).} Some components of $\mathbf{u}^k$ have slow decay on $\partial B(x^k_{1,1},\tau^{k}_{1})$. Since $l^k(S^k_1)\thicksim\tau^k_l$ for $1\leq l\leq m_1$, some components of $\mathbf{u}^k$ have slow decay on $\partial B(x^k_{1,1},l^k(S^k_1))$. Then one of the following alternatives (similar to that of \textbf{Step 2} in the proof of Proposition \ref{CYL-Ch-5-Proposition-1}) holds:
      \begin{enumerate}[(\text{2-}1)\text{.}]
        \item $u^k_1$ and $u^k_2$ have slow decay on $\partial B(x^k_{1,1},l^k(S^k_1))$.
        \item $u^k_{1}$ and $u^k_3$ (or $u^k_2$ and $u^k_3$) have slow decay on $\partial B(x^k_{1,1},l^k(S^k_1))$.
        \item One component of $\mathbf{u}^k$ has slow decay on $\partial B(x^k_{1,1},l^k(S^k_1))$.
      \end{enumerate}

      We set $J$ as the maximal set such that each $u^k_i$ ($i\in J$) has slow decay on $\partial B(x^k_{1,1},l^k(S^k_1))$. Let
      \begin{equation*}
        v^k_i(y)=u^k_i(x^k_{1,1}+l^k(S^k_1) y)+2\log{l^k(S^k_1)},\quad i=1,2,3,
      \end{equation*}
      then $v^k_i(y)\to -\infty$ in $L^{\infty}_{\mathrm{loc}}(\mathbb{R}^2)$ for $i\in \{1,2,3\}\setminus J$ and $v^k_i(y)$ converges to $v_i(y)$ in $C_{\mathrm{loc}}^2(\mathbb{R}^2)$ for $i\in J$, where $v_{i}(y)$ satisfies
      \begin{equation*}
      \Delta v_{i}(y)+\sum\limits_{j\in J}a_{ij}e^{v_j(y)}=4\pi\sum\limits_{l=1}^{m_1}n_{l,i}\delta_{q_l}\quad\text{in}\quad\mathbb{R}^2,\quad i\in J,
      \end{equation*}
      where $n_{l,i}=-2\sum_{j=1}^{3}a_{ij} m_{l,j} \in\mathbb{N}\cup\{0\}$. Applying the classification result Theorem \ref{CYL-Ch7-Theorem-Classification}-(b), there exists a sequence $R_k$ with $R_{k}\to +\infty$ as $k\to+\infty$ such that $R_k l^k(S^k_1)\leq \tau^k_{S^k_1}$ and
      \begin{enumerate}[(a)\text{.}]
      \item $\int_{B(0,{R_k})}e^{v_i}\mathrm{d}y=\int_{\mathbb{R}^2}e^{v_i}\mathrm{d}y+o(1)
          =4\widetilde{m}_{i}+o(1),~i\in J$, where $\widetilde{m}_{i}\in\mathbb{N}$,
      \item $v^k_i(y)+2\log|y|\leq-N_k$ for $|y|=R_{k}$, for some $N_k\to+\infty,~i=1,2,3$,
      \item $\mathbf{u}^k$ has fast decay on $\partial B(x^k_{1,1},R_k l^k{(S^k_1)})$.
      \end{enumerate}
      Hence the local mass
      \begin{equation*}
      \sigma^k_i(B(x^k_{1,1},R_k l^{k}(S^k_1)))=4\left(\sum\limits_{l=1}^{m_1}m_{l,i}+\widetilde{m}_{i}\right)+o(1),\quad i\in J,
      \end{equation*}
      and
      \begin{equation*}
      \sigma^k_i(B(x^k_{1,1},R_k l^{k}(S^k_1)))=4\sum\limits_{l=1}^{m_1}m_{l,i}+o(1),\quad i\in \{1,2,3\}\setminus J.
      \end{equation*}
      Notice that the total gain of local masses at each step is at least $4$. Following almost the same argument as {\bf Step 3} of the proof to Proposition \ref{CYL-Ch-5-Proposition-1}, we could get the conclusion \ref{CYL-Ch-5-Proposition-3} in finitely many steps.
\end{proof}

\section{Proof of Theorem \ref{CYL-Theorem-1} and Theorem \ref{CYL-Theorem-2}}\label{CYL-Section-6}
\setcounter{equation}{0}
In last section, we have decomposed
\begin{equation*}
	\Sigma_k=\{0\}\cup S_1^k\cup\cdots\cup S_{m_0}^k,
\end{equation*}
and completed the combination process for the points in each group $S_l$ with $1\leq l\leq m_0$. In this section, we shall regard $S_l^k$ as points and do the further collection. Precisely, we shall collect $\{0\}$ and several closer sets $S_l^k$ at first. Suppose $x_l^k\in S_l^k$. Without loss of generality, we may assume the sets $S_1^k,\cdots,S_{l_0}^k$ are the ones with
\begin{equation*}
	C^{-1}|x_1^k|\leq|x_l^k|\leq C|x_1^k|,~\mbox{for}~1\leq l\leq l_0,
\end{equation*}
and
\begin{equation*}
	|x_l^k|\gg |x_1^k|\quad \mbox{for}\quad l>l_0.
\end{equation*}
Setting $\tau_k=\frac12$ if $m_0=0$ and $\tau_k=\frac{1}{2}|x^k_1|$ if $m_0>0$. Let $\sigma_i=\hat{\sigma}_i(\bm{\tau})$, where $\hat{\sigma}_i(\bm{\tau})$ is defined as in \eqref{CYL-Ch-4-Eq-14}, then $\bm{\sigma}\in\Gamma(\bm{\mu})$. To start the second step of combination procedure, we compute the possible blow up mass in the bubbling disk centered at origin.

\begin{lemma}\label{CYL-Ch-6-Lemma-1}
$\bm{\sigma}=(\sigma_1,\sigma_2,\sigma_3)\in\Gamma(\bm{\mu})$.
\end{lemma}
\begin{proof}
We select $r_k\ll\tau_k$ such that $\max\limits_{i=1,2,3}\sigma^k_i(B(0,r_k))\leq1/k$ and $\mathbf{u}^k$ has fast decay on $\partial B(0,r_k)$. We can apply Lemma \ref{CYL-Ch-4-Lemma-3} to find a sequence $s_k$ such that at least one component of $\mathbf{u}^k$ has slow decay on $\partial B(0,s_k)$ and $\hat{\bm{\sigma}}(\mathbf{s})=0$. If $\tau_k/s_k\leq C$ or $\bm{\sigma}=\hat{\bm{\sigma}}(\mathbf{s})$, then Lemma \ref{CYL-Ch-6-Lemma-1} is proved, i.e., $\bm{\sigma}=\bm{0}\in\Gamma(\bm{\mu})$. If $\tau_k/s_k\to+\infty$, we can apply the same arguments of \textbf{Step 2} and \textbf{Step 3} in the proof of Proposition \ref{CYL-Ch-5-Proposition-1}. In fact, let
\begin{equation*}
v^k_i(y)=u^k_i({s}_ky)+2\log{{s}_k}.
\end{equation*}
{Suppose that there exists $J\subsetneqq\{1,2,3\}$ such that $u^k_i$ has slow decay on $\partial B(0,s_k)$ for $i\in J$, while the left ones have fast decay on $\partial B(0,s_k)$, i.e., $v^k_i(y)\to-\infty$ in $L_{\mathrm{loc}}^{\infty}(\mathbb{R}^2)$ for $i\in\{1,2,3\}\setminus J$ and $v^k_i(y)$ converges to $v_i(y)$ in $C_{\mathrm{loc}}^2(\mathbb{R}^2)$ $\text{for}~i\in J$, where $v_i(y)$ satisfies}
\begin{equation*}
\Delta v_i(y)+\sum\limits_{j\in J}a_{ij}e^{v_j(y)}=4\pi\gamma_i^{*}\delta_0\quad\text{in}\quad\mathbb{R}^2,\quad i\in J,
\end{equation*}
and $\gamma_i^{*}=\gamma_i-\frac{1}{2}\sum_{j=1}^{3}a_{ij}\hat{\sigma}_{j}(\mathbf{s})=\gamma_i>-1$. Furthermore, there exists a sequence $N_k^{*}$ with $N_{k}^{*}\to +\infty$ as $k\to+\infty$ such that $N_k^{*}s_k\leq\tau_k$ and satisfies
\begin{enumerate}[(1).]
\item $\int_{B(0,{N_k^{*}})}e^{v_{i}}\mathrm{d}y=\int_{\mathbb{R}^2}e^{v_{i}}\mathrm{d}y+o(1)$, $i\in J$,
\item $v^k_i(y)+2\log|y|\leq-N_k$ for $|y|=N^{*}_{k}$, for some $N_k\to+\infty$, $i=1,2,3$,
\item $\mathbf{u}^k$ has fast decay on $\partial B(x^k,N^{*}_{k} s_k)$.
\end{enumerate}
Then one of the following alternatives holds:
\medskip

\noindent\textbf{Case (a).} If $J$ only contains one element $l$ for some $l\in\{1,2,3\}.$ Applying the classification result of \cite{Prajapat-Tarantello-2001}, we have
\begin{equation*}
\begin{aligned}
\hat{\sigma}_{l}(\mathbf{N^{*}s})&=\frac{1}{2\pi}\int_{\mathbb{R}^2}e^{v_{l}}\mathrm{d}x
+\hat{\sigma}_{l}(\mathbf{s})
=4\left(1+\gamma_l\right)+0=4\mu_l=\left(\mathfrak{R}_{l}(\bm{0})\right)_{l}.
\end{aligned}
\end{equation*}
Notice that $\hat{\sigma}_{i}(\mathbf{N^{*}s})=\hat{\sigma}_{i}(\mathbf{s})=0$ for $i\in\{1,2,3\}\setminus\{l\}$, hence $\hat{\bm{\sigma}}(\mathbf{N^{*}s})=\mathfrak{R}_{l}(\bm{0})\in\Gamma(\bm{\mu})$.
\medskip

\noindent\textbf{Case (b).} If {$J=\{1,2\}$.} By the classification result of \cite{Prajapat-Tarantello-2001} again, we have
\begin{equation*}
\begin{aligned}
\hat{\sigma}_{l}(\mathbf{N^{*}s})&=\frac{1}{2\pi}\int_{\mathbb{R}^2}e^{v_{l}}\mathrm{d}x+\hat{\sigma}_l(\mathbf{s})
=4\mu_l+0=\left(\mathfrak{R}_{12}(\bm{0})\right)_{l},\quad l=1,2.
\end{aligned}
\end{equation*}
While $\hat{\sigma}_3(\mathbf{N^{*}s})=\hat{\sigma}_3(\mathbf{s})=0$, therefore $\hat{\bm{\sigma}}(\mathbf{N^{*}s})=\mathfrak{R}_{12}(\bm{0})\in\Gamma(\bm{\mu})$.
\medskip

\noindent\textbf{Case (c).} If {$J=\{l,3\}$ for $l=1$ or $2$.} Applying the classification result Theorem \ref{CYL-Ch7-Theorem-Classification}-(a), we have
\begin{equation*}
\begin{aligned}
\hat{\sigma}_{l}(\mathbf{N^{*}s})
&=\frac{1}{2\pi}\int_{\mathbb{R}^2}e^{v_{l}}\mathrm{d}x+\hat{\sigma}_{l}(\mathbf{s})=8\mu_l+8\mu_3+0
=\left(\mathfrak{R}_{l3}(\bm{0})\right)_{l},\quad l=1~\text{or}~2,
\end{aligned}
\end{equation*}
and
\begin{equation*}
\begin{aligned}
\hat{\sigma}_{3}(\mathbf{N^{*}s})
&=\frac{1}{2\pi}\int_{\mathbb{R}^2}e^{v_{3}}\mathrm{d}x+\hat{\sigma}_{3}(\mathbf{s})
=4\mu_l+8\mu_3+0
=\left(\mathfrak{R}_{l3}(\bm{0})\right)_{3}, \quad l=1~\text{or}~2.
\end{aligned}
\end{equation*}
For the left component  $j\in\{1,2,3\}\setminus\{l,3\}$, $\hat{\sigma}_{j}(\mathbf{N^{*}s})=\hat{\sigma}_{j}(\mathbf{s})=0$, then $\hat{\bm{\sigma}}(\mathbf{N^{*}s})=\mathfrak{R}_{l3}(\bm{0})\in\Gamma(\bm{\mu})$, $l=1$ or $2$.
\medskip

Next, let $s_{k,1}=N_{k}^{*}s_k$. If
\begin{equation*}
\hat{\sigma}_i(\mathbf{s}_1)=\hat{\sigma}_i(\bm{\tau})\quad\text{for}\quad i=1,2,3,
\end{equation*}
then Lemma \ref{CYL-Ch-6-Lemma-1} is proved. Otherwise, we could repeat the above argument to find $s_{k,1}\ll s_{k,j}\ll s_{k,j+1}$ such that
$\hat{\bm{\sigma}}(\mathbf{s_{j+1}})\in\Gamma(\bm{\mu})$. Since the energy is finite and the total gain for the local masses at each step has a lower bound
\begin{equation*}
\sum\limits_{i=1}^{3}\left(\hat{\sigma}_i(\mathbf{s}_{j+1})-\hat{\sigma}_i(\mathbf{s}_j)\right)
\geq\min\limits_{i=1,2,3}4\mu_i>0,
\end{equation*}
the process will stop after finitely many $j$ steps (by a little abuse of notations) and we have
\begin{equation*}
\hat{\sigma}_i(\mathbf{s}_j)=\hat{\sigma}_i(\bm{\tau})\quad\text{for}\quad i=1,2,3.
\end{equation*}
Hence
\begin{equation*}
(\sigma_1,\sigma_2,\sigma_3)=(\hat{\sigma}_1(\bm{\tau}),\hat{\sigma}_2(\bm{\tau}),\hat{\sigma}_3(\bm{\tau}))
\in\Gamma(\bm{\mu}),
\end{equation*}
The statement of at least one component of $\mathbf{u}^k$ has fast decay on $\partial B(0,\tau_k)$ follows by Lemma \ref{CYL-Ch-4-Lemma-1}-(b).
\end{proof}

Now we are able to provide the proof of Theorem  \ref{CYL-Theorem-1}
\begin{proof}[{\it Proof of Theorem \ref{CYL-Theorem-1}}.]
If $m_0=0$, we prove it immediately by Lemma \ref{CYL-Ch-6-Lemma-1}. So we assume that $m_0\neq 0$. We select $S_1^k,\cdots,S_{l_0}^k$ as the beginning of this section  and group together
\begin{equation*}
S=\{0\}\cup S_1^k\cup\cdots\cup S_{l_0}^k.
\end{equation*}
In other words, we view $\{0\}$ and $S_i^k$ ($1\leq i\leq l_0$) as points. Let
\begin{equation*}
{\hat{\tau}^k(S)=\frac{1}{2}\|S^k_{l_0+1}\|\gg l^k_{l_0}(S),}
\end{equation*}
where
\begin{equation*}
{\|S^k_{l_0+1}\|=\min\limits_{x^k\in \Sigma_k\setminus S}\mathrm{dist}(0,x^k)} \quad \text{and} \quad l^k_{l_0}(S):=2\max\limits_{1\leq l\leq l_0}\max\limits_{x^k\in S_l^k}\mathrm{dist}(0,x^k).
\end{equation*}
Similar to the proof of Proposition \ref{CYL-Ch-5-Proposition-3}, we shall encounter the following two cases.
\medskip

\noindent \textbf{Case (1).} $\mathbf{u}^k$ has fast decay on $\partial B(0,l^k_{l_0}(S))$. Similar to the argument of \eqref{CYL-Ch-5-Eq-7} in Proposition \ref{CYL-Ch-5-Proposition-3}, we deduce that the local mass
\begin{equation*}
\begin{aligned}
\hat{\sigma}_{i}(\mathbf{l}_{l_0}(S))=\hat{\sigma}_{i}({\bm{\tau}})+
\sum\limits_{l=1}^{l_0}\hat{\sigma}_{i}\left(B\left(\mathbf{x}_l,\bm{\tau}_{\mathbf{s}_l}\right)\right)
=\hat{\sigma}_{i}({\bm{\tau}})+4m_i,\quad i=1,2,3,
\end{aligned}
\end{equation*}
where $m_i\in\mathbb{N}\cup\{0\}$, $\mathbf{l}_{l_0}(S)$ and $\left(\mathbf{x}_\ell,\bm{\tau}_{S_\ell}\right)$ stand for the sequence of $\{l^k_{l_0}(S)\}$ and the pair $\{(x^k_l,\tau^k_{S_l})\}$, and $\hat{\bm{\sigma}}({\bm{\tau}})\in\Gamma(\bm{\mu})$ by Lemma \ref{CYL-Ch-6-Lemma-1}.
If
\begin{equation}
\label{eq6-condition1}
\hat{\sigma}_{i}(\hat{\bm{\tau}}(S))=\hat{\sigma}_{i}(\mathbf{l}_{l_0}(S)),\quad i=1,2,3,
\end{equation}
then we complete the process of grouping $\{0\}$ and $S_1^k,\cdots,S_{l_0}^k$, and the proof of Theorem \ref{CYL-Theorem-1} is complete by grouping till $S^k_{m_0}$. If \eqref{eq6-condition1} does not hold, by Lemma \ref{CYL-Ch-4-Lemma-3} we can find a sequence $s_k$ with $l^k_{l_0}(S)\ll s_k\ll\hat{\tau}^k(S)$ such that some components  $u^k_i~(i\in J)$ have slow decay on $\partial B(0,s_k)$ and
\begin{equation*}
\hat{\sigma}_{i}(\mathbf{l}_{l_0}(S))=\hat{\sigma}_i(\mathbf{s}),\quad i=1,2,3,
\end{equation*}
where $\mathbf{s}$ stands for the sequence $\{s_k\}$. Let
\begin{equation*}
v^k_i(y)=u^k_i(s_ky)+2\log{s_k}.
\end{equation*}
Then we can apply the same argument of \textbf{Step 2} and \textbf{Step 3} in the proof of Proposition \ref{CYL-Ch-5-Proposition-1} and the result of Proposition \ref{CYL-Ch-5-Proposition-2} to get that
\begin{equation*}
\hat{\sigma}_{i}(\hat{\bm{\tau}}(S))=\hat{\sigma}_{i}(\mathbf{s}_{\ell})\in\Gamma(\bm{\mu})+4\mathbb{Z},\quad i=1,2,3,
\end{equation*}
for some $\mathbf{s}_{\ell}$, where $\mathbf{s}_{\ell}$ stands for the sequence $\{s^k_\ell\}$ with $l_{l_0}^k(S)\ll s^k_\ell\ll \hat{\tau}^k(S)$. Indeed, we have $v^k_i(y)\to-\infty$ in $L^\infty_{\mathrm{loc}}(\mathbb{R}^2)$ for $i\in \{1,2,3\}\setminus J$ and $v^k_i(y)$ converges to $v_i(y)$ for $i\in J$ in $C^2_{\mathrm{loc}}(\mathbb{R}^2)$, where $v_i(y)$ satisfies \eqref{CYL-Ch-5-Eq-2}. All the possibilities have been discussed in Proposition \ref{CYL-Ch-5-Proposition-2}. Then there is a sequence $N_k^*\to+\infty$ as $k\to+\infty$ such that $\mathbf{u}^k$ has fast decay on $\partial B(0,N_k^* s_k)$. Using Proposition \ref{CYL-Ch-5-Proposition-2} we conclude that the new local mass
\begin{equation*}
\bm{\sigma}:=\lim\limits_{k\to+\infty}\left(\sigma^k_1(N_k^* s_k),\sigma^k_2(N_k^* s_k),\sigma^k_3(N_k^* s_k)\right)\in\Gamma(\bm{\mu})+4\mathbb{Z}.
\end{equation*}
We can repeat almost the same procedure as above to find $s_{k,\ell+1}\gg s_{k,\ell}$ such that $\hat{\bm{\sigma}}(\mathbf{s}_{\ell+1})\in\Gamma(\bm{\mu})+4\mathbb{Z}$ if necessary. Since the energy is finite and the total gain of local masses at each step has a lower bound, then the process will stop after finitely many steps and
\begin{equation*}
\bm{\sigma}
=(\hat{\sigma}_{1}(\hat{\bm{\tau}}(S)),\hat{\sigma}_{2}(\hat{\bm{\tau}}(S)),\hat{\sigma}_{3}(\hat{\bm{\tau}}(S)))
\in\Gamma(\bm{\mu})+4\mathbb{Z}.
\end{equation*}
This ends the process of grouping $\{0\}$ and $S^k_1,\cdots,S^k_{l_0}$. One can continue the procedure until  $S_{m_0}$ is included. Then Theorem \ref{CYL-Theorem-1} is proved.
\medskip

\noindent \textbf{Case (2).} Suppose that there is a subset $J\subsetneqq \{1,2,3\}$ such that $u^k_i~(i\in J)$ has slow decay on $\partial B(0,l^k_{l_0}(S))$. Let
\begin{equation*}
v^k_i(y)=u^k_i(l^k_{l_0}(S) y)+2\log{l^k_{l_0}(S)}.
\end{equation*}
Then we go back to the proof of \textbf{Case~(1)} by replacing $s_k$ by $l^k_{l_0}(S)$. The rest proof is similar and we skip the details.
\end{proof}

As a consequence, we are able to provide the proof of Theorem \ref{CYL-Theorem-2}.

\begin{proof}[{\it Proof of Theorem \ref{CYL-Theorem-2}.}]
Suppose that $\mathbf{\tilde u}^k$ is a sequence of solutions to system  \eqref{1.16}. Let
\begin{equation*}
\hat{u}^k_i(x)=\tilde u^k_i(x)+\log\frac{\rho^k_i}{\int_{M}h^k_ie^{u^k_i}\mathrm{d}V_{g}},\quad i=1,2,3.
\end{equation*}
Then the system \eqref{1.16} can be rewritten as follows
\begin{equation}
\label{6.final}
\begin{aligned}
\Delta \hat{u}^k_i+\sum\limits_{j=1}^{3}a_{ij}\tilde h^k_je^{\hat{u}^k_j}-\sum\limits_{j=1}^3\frac{a_{ij}\rho_j}{|M|}=0,\quad i=1,2,3.
\end{aligned}
\end{equation}
We notice that
$$\hat u_i^k\leq \tilde u_i^k+\log \rho_i^k-\int_Mu_i^k\mathrm{d}V_g-C\quad\mbox{for some universal constant}~C,$$
where we have used the Jensen's inequality and $h_i^k$ are smooth positive functions. Thus,
\begin{equation}
\label{6.final-1}
\hat u_1^k+\hat u_2^k+2\hat u_3^k= C_k,\quad C_k~\mbox{is uniformly bounded above with respect to}~k.
\end{equation}

Around each blow up $p$ we take $r$ such that $B\cap B_r(p)=\{p\}$, and set
\begin{equation*}
\bar u_i^k=\begin{cases}
\hat u_i^k,\quad &\mbox{if}\quad p\notin S,\\	
\hat u_i^k+4\pi\alpha_{p,i}G(x,p), &\mbox{if}\quad p\in S.
\end{cases}
\end{equation*}
Together with \eqref{1.15}, \eqref{6.final} and \eqref{6.final-1} we see that $\bar u_i^k$ satisfies
\begin{equation}
\label{6.final-3}
\begin{cases}
\Delta\bar u_i^k+\sum\limits_{j=1}^3a_{ij}\bar h_j^ke^{\bar u_j^k}-\sum\limits_{j=1}^3\frac{a_{ij}\rho_j}{|M|}+\frac{4\pi\beta_{p,i}}{|M|}=4\pi\beta_{p,i}\delta_p,\quad i=1,2,3,\\
\bar u_1^k+\bar u_2^k+2\bar u_3^k= C_k,\quad C_k~\mbox{is uniformly bounded above with respect to}~k.
\end{cases}
\end{equation}
where $\bar h_i^k=\tilde h_i^k\exp\left(-4\pi\beta_{p,i}G(x,p)\right),~i=1,2,3$ and
\begin{equation*}
\beta_{p,i}=\begin{cases}
0,\quad &\mbox{if}\quad p\notin S\\
\alpha_{p,i},\quad &\mbox{if}\quad p\in S.
\end{cases}
\end{equation*}
Then it is not difficult to see that $\bar h_i^k,~i=1,2,3$ are
smooth and positive in $B_r(p)$. Though there are some extra constant terms in \eqref{6.final-3} and $\Delta_g$ contains some other terms than the Euclidean Laplace operator in the local coordinate, they can be ignored when we perform the blow up analysis for \eqref{6.final-3} since we only need the local information for computing the local mass. By Theorem \ref{CYL-Theorem-1} and the Remark \ref{CYL-Ch-1-Remark-1} we have
\begin{equation*}
\frac{1}{2\pi}\lim\limits_{r\to 0}\lim\limits_{k\to+\infty}\int_{B(p,r)}\bar h^k_je^{\bar{u}^k_j}\mathrm{d}V_{g}=\sigma_{p,j}^*+4m_{p,j},~j=1,2,3,
\end{equation*}
where $(\sigma_{p,1}^*,\sigma_{p,2}^*,\sigma_{p,3}^*)\in\Gamma(\mu_{p,1},\mu_{p,2},\mu_{p,3})$ and $m_{p,j}\in\mathbb{Z}$. This amounts to say that
\begin{equation*}
\frac{1}{2\pi}\lim\limits_{r\to 0}\lim\limits_{k\to +\infty}\frac{\rho^k_{j}\int_{B(p,r)}h^k_je^{u^k_j}\mathrm{d}V_{g}}{\int_{M}h^k_je^{u^k_j}\mathrm{d}V_{g}}
=\sigma_{p,j}^*+4m_{p,j},~j=1,2,3.
\end{equation*}
Together with the fact there is  {at least} one component which possesses the concentration property, without loss of generality, we may assume the first component has such property. Then we get that
\begin{equation*}
\lim\limits_{k\to+\infty}\rho^k_1=2\pi\sum\limits_{p \in R}\sigma_{p,1}^*+8\pi m_{1}~\mbox{for some}~m_1\in\mathbb{Z}.
\end{equation*}
It obeys the choice of $\rho_1\notin\Gamma_1$. Therefore, we must have ${\mathbf{\tilde u}}^k$ is uniformly bounded and it finishes the proof.
\end{proof}

\section{Blow-up analysis for the sinh-Gordon equation}\label{CYL-Section-8}
\setcounter{equation}{0}
In this section, we consider the blow-up analysis for the sinh-Gordon equation
\begin{equation}\label{CYL-Ch-8-Eq-1}
\Delta u+e^{u}-e^{-u}=0.
\end{equation}
Let $u_1=u,~u_2=-u$, then \eqref{CYL-Ch-8-Eq-1} can be written as the following equation with $\gamma_1=\gamma_2=0$,
\begin{equation}\label{CYL-Ch-8-Eq-2}
\left\{
\begin{aligned}
\Delta u_{1}+e^{u_{1}}-e^{u_{2}}&=4\pi{\gamma_1}\delta_0,\\
\Delta u_{2}-e^{u_{1}}+e^{u_{2}}&=4\pi{\gamma_2}\delta_0,\\
u_{1}+u_{2}&=0.
\end{aligned}
\right.
\end{equation}
The local mass for the blow up solutions to equation \eqref{CYL-Ch-8-Eq-1} with $\gamma_1=\gamma_2=0$ has been derived by Jost-Wang-Ye-Zhou \cite{Jost-Wang-Ye-Zhou-2008} and Jevnikar-Wei-Yang \cite{Jevnikar-Wei-Yang-2018DIE}  by different methods. Now we shall consider a more general situation, where
\begin{equation}\label{CYL-Ch-8-Eq-9}
\gamma_i>-1,~i=1,2,\quad \gamma_1=-\gamma_2\neq0.
\end{equation}
Suppose that $\mathbf{u}^k=(u^k_1,u^k_2)$ is a sequence of blow-up solutions of system \eqref{CYL-Ch-8-Eq-2} in $B_1(0)$  satisfying
\begin{equation}
\label{CYL-Ch-8-Eq-5}
\left\{
\begin{aligned}
&\int_{B_{1}(0)}e^{u_{i}^{k}(x)}\mathrm{d}x\leq C,\quad i=1,2,\\
&|u^{k}_{i}(x)-u^{k}_{i}(y)|\leq C,~i=1,2,~\forall x,y\in\partial B_1(0),~\text{for some constant}~C>0,\\
&\max\limits_{i=1,2}\max\limits_{K\subset\subset\overline{B_{1}(0)}\setminus\{0\}}u^{k}_{i}(x)\leq C(K),~\max\limits_{i=1,2}\max\limits_{\overline{B_1(0)}}(u^k_i(x)+2\gamma_i\log|x|)\to+\infty.
\end{aligned}
\right.
\end{equation}
Denote the local mass $\bm{\sigma}$ corresponding to the solutions $\mathbf{u}^k$ by
\begin{equation*}
\begin{aligned}
\sigma_{i}:=\frac{1}{2\pi}\lim\limits_{r\to 0}\lim\limits_{k\to\infty}\int_{B_{r}(0)}e^{u^{k}_{i}(x)}\mathrm{d}x,~i=1,2.
\end{aligned}
\end{equation*}
Then applying the same arguments from Section \ref{CYL-Section-2} to Section \ref{CYL-Section-6} in the present article, we obtain the following result.

\begin{theorem}\label{CYL-Ch-8-Theorem-1}
Suppose that $\mathbf{u}^{k}=(u^{k}_{1},u^{k}_{2})$ is a sequence of solutions of system \eqref{CYL-Ch-8-Eq-2} satisfying \eqref{CYL-Ch-8-Eq-9}--\eqref{CYL-Ch-8-Eq-5}. Let $\mu_i=1+\gamma_i,~i=1,2$. Then there exists ${\bm{0}\neq\hat{\bm{\sigma}}}=(\hat{\sigma}_1,\hat{\sigma}_2)\in\Gamma(\bm{\mu})$ such that
\begin{equation*}
{\sigma}_i=\hat{\sigma}_i+4m_{i},~m_i\in\mathbb{Z},~i=1,2,
\end{equation*}
where
\begin{equation}\label{CYL-Ch-8-Eq-7}
\begin{aligned}
\Gamma(\bm{\mu})=&\bigg\{(\sigma_1,\sigma_2)~\big{|} ~ \sigma_i=4\sum\limits_{j=1}^{2}n_{ij}\mu_j,~n_{ij}\in\mathbb{N}\cup\{0\},~i=1,2,~(\sigma_1-\sigma_2)^2
=4(\mu_1\sigma_1+\mu_2\sigma_2)\bigg\}.
\end{aligned}
\end{equation}
Furthermore, one can easily check that \eqref{CYL-Ch-8-Eq-7} is equivalent to
\begin{equation*}
\Gamma(\bm{\mu})=\bigg\{(\sigma_1,\sigma_2)~\bigg{|}~
\begin{aligned}
&\sigma_1=(m+1)^2\mu_1+(m^2-1)\mu_2,~~\sigma_2=(m^2-1)\mu_1+(m-1)^2\mu_2,~m\equiv 1~(\mathrm{mod}~2)\\
&\sigma_1=m^2\mu_1+\left((m-1)^2-1\right)\mu_2,~\sigma_2=\left((m+1)^2-1\right)\mu_1+m^2\mu_2,~m\equiv 0~(\mathrm{mod}~2)
\end{aligned}
\bigg\}.
\end{equation*}
\end{theorem}
\begin{rmk}\label{CYL-Ch-8-Remark-2}
For the case of $\gamma_1=\gamma_2=0$, we obtain
\begin{equation*}
\Gamma(\bm{\mu})=\big\{(\sigma_1,\sigma_2)=\left(2m(m+1),2m(m-1)\right)~
\text{or}~\left(2m(m-1),2m(m+1)\right)\big\}.
\end{equation*}
This result coincides with \cite[Corollary 1.2]{Jost-Wang-Ye-Zhou-2008} and \cite[Theorem 1.1]{Jevnikar-Wei-Yang-2018DIE}.
\end{rmk}

\section{Appendix: classification results}\label{CYL-Section-7}
\setcounter{equation}{0}
In this appendix, we provide a quantization result which has been used frequently in the grouping process. We consider the following system
\begin{equation}\label{CYL-Ch-7-Eq-1}
\begin{cases}
\Delta u+e^u-\frac12e^v=4\pi\sum_{\ell=1}^N\alpha_{\ell,1}\delta_{p_\ell},\quad &\mathrm{in}\quad\mathbb{R}^2,\\
\Delta v-e^u+e^v=4\pi\sum_{\ell=1}^N\alpha_{\ell,2}\delta_{p_\ell},\quad &\mathrm{in}\quad\mathbb{R}^2,\\
\int_{\mathbb{R}^2}e^u\mathrm{d}x,~\int_{\mathbb{R}^2}e^v\mathrm{d}x<+\infty,
\end{cases}
\end{equation}
{where $p_{\ell}$ are distinct points in $\mathbb{R}^2$ and $\alpha_{\ell,i}>-1$ for $1\leq\ell\leq N$, $i=1,2$.}

Let
\begin{equation*}
\sigma_u=\frac{1}{2\pi}\int_{\mathbb{R}^2}e^u\mathrm{d}x,\quad
\sigma_v=\frac{1}{2\pi}\int_{\mathbb{R}^2}e^v\mathrm{d}x.
\end{equation*}
The main result of this section is the following

\begin{theorem}\label{CYL-Ch7-Theorem-Classification}
Suppose that $(u,v)$ is a solution to system \eqref{CYL-Ch-7-Eq-1}.
\begin{enumerate}[(a).]
\item If $N=1$, i.e., there is only one singular source on the right hand side of \eqref{CYL-Ch-7-Eq-1}, then there holds
\begin{equation*}
(\sigma_u,\sigma_v)=(8\alpha_{1,1}+4\alpha_{1,2}+12, 8\alpha_{1,1}+8\alpha_{1,2}+16).
\end{equation*}
	
\item If $\alpha_{\ell,i}\in\mathbb{N}\cup\{0\},~i=1,2,~1\le\ell\le N$, then there holds $(\sigma_u,\sigma_v)=(4\tilde N_1,4\tilde N_2)$ for some {positive} integers $\tilde N_1$ and $\tilde N_2$.
	
\item If $\alpha_{1,i}>-1,~i=1,2$ and $\alpha_{\ell,i}\in\mathbb{N}\cup\{0\},~i=1,2,~2\le\ell\le N$, then there holds
\begin{equation*}
(\sigma_u,\sigma_v)=(\sigma_u^*+4N_1^*,\sigma_v^*+4N_2^*),
\end{equation*}
where {$N_i^*\in\mathbb{Z}$}, $i=1,2$ and
\begin{equation*}
\begin{aligned}
(\sigma_u^*,\sigma_v^*)\in\Big\{&(0,0),(4\alpha_{1,1},0),(0,4\alpha_{1,2}),
(4\alpha_{1,1},8\alpha_{1,1}+4\alpha_{1,2}),
(4\alpha_{1,1}+4\alpha_{1,2},4\alpha_{1,2}),\\
&(4\alpha_{1,1}+4\alpha_{1,2},8\alpha_{1,1}+8\alpha_{1,2}),
(8\alpha_{1,1}+4\alpha_{1,2},8\alpha_{1,1}+4\alpha_{1,2}), (8\alpha_{1,1}+4\alpha_{1,2},8\alpha_{1,1}+8\alpha_{1,2})\Big\}.
\end{aligned}
\end{equation*}
\end{enumerate}
\end{theorem}
\begin{proof}
Let
$$u=u_1+\log2,\quad v=u_2+\log2.$$
Then we have
\begin{equation}\label{CYL-Ch-7-Eq-2}
(\sigma_{u_1},\sigma_{u_2})=\frac12(\sigma_u,\sigma_v),~\quad {where}\quad \sigma_{u_i}=\frac{1}{2\pi}\int_{\mathbb{R}^2}e^{u_i}\mathrm{d}x,
\end{equation}
and $(u_1,u_2)$ satisfies
\begin{equation}\label{CYL-Ch-7-Eq-3}
\begin{cases}
\Delta u_1+2e^{u_1}-e^{u_2}=4\pi\sum_{\ell=1}^N\alpha_{\ell,1}\delta_{p_\ell}, \\
\Delta u_2-2 e^{u_1}+2e^{u_2}=4\pi\sum_{\ell=1}^N\alpha_{\ell,2}\delta_{p_\ell},
\end{cases}\quad  \mathrm{in}\quad\mathbb{R}^2.
\end{equation}
We notice that equation \eqref{CYL-Ch-7-Eq-3} is exactly the standard $\mathbf{C}_2$ Toda system, which can be obtained from ${\bf A}_3$ Toda system with some symmetry condition. Precisely, we consider the following system
\begin{equation}\label{CYL-Ch-7-Eq-4}
\begin{cases}
\Delta \tilde u_1+2e^{\tilde u_1}-e^{\tilde u_2}=4\pi\sum_{\ell=1}^N\alpha_{\ell,1}\delta_{p_\ell}, \\
\Delta \tilde u_2- e^{\tilde u_1}+2e^{\tilde u_2}-e^{\tilde u_3}=4\pi\sum_{\ell=1}^N\alpha_{\ell,2}\delta_{p_\ell}, \\
\Delta \tilde u_3- e^{\tilde u_2}+2e^{\tilde u_3} =4\pi\sum_{\ell=1}^N\alpha_{\ell,3}\delta_{p_\ell},
\end{cases}\quad  \mathrm{in}\quad\mathbb{R}^2.
\end{equation}
Setting
\begin{equation}
\label{CYL-Ch-7-Eq-5}
\tilde u_1=\tilde u_3=u_1, \quad \tilde u_2=u_2, \quad \alpha_{\ell,1}=\alpha_{\ell,3},~1\leq \ell\leq N.
\end{equation}
Then \eqref{CYL-Ch-7-Eq-4} can be reduced to \eqref{CYL-Ch-7-Eq-3}.

When there is only one singularity on the right hand side of \eqref{CYL-Ch-7-Eq-4} we get from \cite[Theorem 1.3]{Lin-Wei-Ye-2012} that
\begin{equation*}
\frac{1}{2\pi}\int_{\mathbb{R}^2}e^{\tilde u_1}\mathrm{d}x=4\alpha_{1,1}+2\alpha_{1,2}+6,\quad
\frac{1}{2\pi}\int_{\mathbb{R}^2}e^{\tilde u_2}\mathrm{d}x=4\alpha_{1,1}+4\alpha_{1,2}+8,
\end{equation*}
which together with \eqref{CYL-Ch-7-Eq-2} implies the conclusion (a).

It is easy to see that conclusion (b) is an easy consequence of (c). So in the following we shall focus on the point (c). By \cite[Theorem 8.1]{Lin-Yang-Zhong-2020}, we have
\begin{equation*}
\sigma_{u_1}=2\sum_{l=1}^{f(0)}\alpha_{1,l}+2\bar N_1,\quad
\sigma_{u_2}=2\sum_{j=0}^1\left(\sum_{l=1}^{f(j)}\alpha_{1,l}-
\sum_{l=1}^{j}\alpha_{1,l}\right)+2\bar N_2,
\end{equation*}
where {$\bar N_i\in\mathbb{Z}$}, $i=1,2$, $\alpha_{1,l}$, $l=1,2,3$ are given as in \eqref{CYL-Ch-7-Eq-5} and $f$ is a permutation map from $\{0,1,2,3\}$ to itself satisfying
\begin{equation*}
f(l)+f(3-l)=3,~l=0,1,2,3.
\end{equation*}
Exhausting all the possibilities of $f$, we have
\begin{equation*}
\begin{aligned}
\left(2\sum_{l=1}^{f(0)}\alpha_{1,l},2\sum_{j=0}^1\left(\sum_{l=1}^{f(j)}\alpha_{1,l}-
\sum_{l=1}^{j}\alpha_{1,l}\right)\right)\in
\Big\{&(0,0),(2\alpha_{1,1},0),(0,2\alpha_{1,2}),(2\alpha_{1,1},4\alpha_{1,1}+2\alpha_{1,2}),\\
&(2\alpha_{1,1}+2\alpha_{1,2},2\alpha_{1,2}),
(2\alpha_{1,1}+2\alpha_{1,2},4\alpha_{1,1}+4\alpha_{1,2}),\\
&(4\alpha_{1,1}+2\alpha_{1,2},4\alpha_{1,1}+2\alpha_{1,2}), (4\alpha_{1,1}+2\alpha_{1,2},4\alpha_{1,1}+4\alpha_{1,2})\Big\}.
\end{aligned}
\end{equation*}
Together with \eqref{CYL-Ch-7-Eq-2}, we get the conclusion (c) and it finishes the proof.
\end{proof}

\section*{Acknowledgement}
J. Wei is partially supported by National Sciences and Engineering Research Council of Canada. W. Yang is partially supported by NSFC No. 12171456 and No. 11871470. L. Zhang is partially supported by Simons Foundation Collaboration Grant.  We thank Prof. Zhaohu Nie for useful conversations.

\end{document}